\documentclass[a4paper, 11pt]{amsart}
\usepackage[margin=1in]{geometry}

\usepackage[utf8x]{inputenc}
\usepackage{graphicx}

\usepackage{algorithm}
\usepackage{algpseudocode}
\usepackage{amsmath}
\usepackage{amsthm}
\usepackage{amsfonts}
\usepackage{amssymb}
\usepackage{bbm}
\usepackage{bm}
\usepackage{mathrsfs}
\usepackage{mathtools}
\usepackage{thmtools}

\usepackage{fancyref}
\usepackage{hyperref}
\usepackage{listings}
\usepackage{nameref}
\usepackage[section]{placeins}
\usepackage{soul}
\usepackage[normalem]{ulem}
\hypersetup{colorlinks=true, linkcolor=black}
\usepackage{xcolor}

\newtheorem{thm}{Theorem}[]
\newtheorem{lem}[thm]{Lemma}

\newtheorem{defi}[thm]{Definition}

\newtheorem{prop}[thm]{Proposition}

\newtheorem{exmp}[thm]{Example}

\declaretheoremstyle[headfont=\normalfont]{normalhead}

\newcommand{\rd}{\mathrm{d}}
\newcommand{\N}{\mathbb{N}}
\newcommand{\Z}{\mathbb{Z}}

\newcommand{\R}{\mathbb{R}}

\newcommand{\defeq}{\vcentcolon=}
\newcommand{\eqdef}{=\vcentcolon}

\DeclareMathOperator{\dive}{div}

\DeclareMathOperator{\supp}{supp}

\DeclareMathOperator{\sgn}{sgn}

\DeclareMathOperator{\Range}{Range}

\DeclareMathOperator{\spn}{span}

\DeclareMathOperator*{\argmin}{arg\,min}

\DeclareMathOperator*{\diam}{diam}
\DeclareMathOperator*{\dist}{dist}

\def\Xint#1{\mathchoice
{\XXint\displaystyle\textstyle{#1}}%
{\XXint\textstyle\scriptstyle{#1}}%
{\XXint\scriptstyle\scriptscriptstyle{#1}}%
{\XXint\scriptscriptstyle\scriptscriptstyle{#1}}%
\!\int}
\def\XXint#1#2#3{{\setbox0=\hbox{$#1{#2#3}{\int}$ }
\vcenter{\hbox{$#2#3$ }}\kern-.6\wd0}}

\def\dashint{\Xint-}

\graphicspath{ {./HW1/} }
\numberwithin{equation}{section}

\begin{document}

\title[Advection on non-cartesian grids]{Discretizing advection equations with rough velocity fields on non-cartesian grids}
\author[P.--E.~Jabin]{Pierre-Emmanuel Jabin}
\address{\sc P.--E. Jabin.  Department of Mathematics and Huck Institutes, Pennsylvania State University, State College, PA 16801, USA}
\email{pejabin@psu.edu}

\author[D. Zhou]{Datong Zhou}
\address{\sc D. Zhou.  Department of Mathematics, Pennsylvania State University, State College, PA 16801, USA}
\email{dbz5086@psu.edu}

\date{}

\maketitle

\begin{abstract}
We investigate the properties of discretizations of advection equations on non-cartesian grids and graphs in general. Advection equations discretized on non-cartesian grids have remained a long-standing challenge as the structure of the grid can lead to strong oscillations in the solution, even for otherwise constant velocity fields. We introduce a new method to track oscillations of the solution for rough velocity fields on any graph. The method in particular highlights some inherent structural conditions on the mesh for propagating regularity on solutions.
\end{abstract}

\tableofcontents

\section{Introduction} \label{sec:introduction}

\subsection{Discretized advection equations}

We introduce a new framework to study the regularity of discretized advection equations. Our method is able to provide quantitative regularity estimates by extending the kernel based approach initially introduced at the continuum level in \cite{BeJa:13,BeJa:19} and further studied in \cite{Le:18,HaSeSmSt:18}. This is particularly helpful when investigating the convergence of numerical schemes for coupled non-linear systems. 

To be more specific, we study discretizations of the classical linear continuity equation,
\begin{equation}\label{eqn:linear_continuity_equation}
\begin{aligned}
\partial_t u(t,x) + \dive_x \big(b(t,x) u(t,x)\big) = 0, \quad t \in \R_+, x\in \R^d,
\end{aligned}
\end{equation}
Those discretized equations usually involve calculating the dynamics of a discrete density $u_i$ that is defined on each cell of a grid or mesh. We specifically focus on upwind schemes that read 
\begin{equation} \label{eqn:continuity_equation_scheme_0}
\begin{aligned}
\frac{\rd u_{i}}{\rd t} = \frac{1}{\pi_i} \sum_{i'} \big( a_{i,i'} \, u_{i'} - a_{i',i} \, u_i \big).
\end{aligned}
\end{equation}
The factor $\pi_i$ represents some notion of volume of the cell $i$ on the mesh and the coefficients $a_{i,i'}, a_{i',i}$ are related to the flux between two cells $i$ and $i'$. We refer to Section \ref{subsec:set_up} for the precise formulas for the schemes that we consider.

The continuity equation constitutes a key relation in a large variety of models, in which the velocity field $b(t,x)$ is typically related to the density $u(t,x)$ in different ways.
As mentioned above, it is this coupling between $b$ and $u$ that makes strong compactness of the density, instead of weak convergence, an essential ingredient and one of the common difficulties when trying to prove convergence for the whole system, whether from numerical approximations or some other approximate system.
We present a few typical examples below that motivate our investigations and have natural applications in biology and fluid mechanics.

The velocity field $b$ in \eqref{eqn:linear_continuity_equation} can first be related to the density $u$ by some convolution $b = K \star g(u)$ for some non-linear function $g$ or by the Poisson equation
\begin{equation} \label{eqn:poissoncoupling}
\begin{aligned}
b(t,x) = \nabla_x \phi(t,x), \quad - \Delta_x \phi(t,x) = g(u(t,x)),
\end{aligned}
\end{equation}
which corresponds to choosing the fundamental solution of the Laplacian as the kernel $K$.
There exist already many examples of such systems in applications: We briefly mention \cite{ToBe:04} for swarming or \cite{BuDoSc:06,PeDa:09} for models of chemotaxis.
The function $g$ represents a non-linear dependence on the density $u$ in the equation for $b$, which can capture more complex phenomena in the model such as logistic effects. 

In a somewhat similar spirit, {\em non-linear} continuity equations may be considered such as
\begin{equation} \label{eqn:nonlinear_continuity_equation}
\begin{aligned}
\partial_t u(t,x) + \dive_x \big(b(t,x) f(u(t,x))\big) = 0, \quad t \in \R_+, x\in \R^d.
\end{aligned}
\end{equation}
This type of non-linear flux combines  non-linear scalar conservation laws with linear advection. Models such as~\eqref{eqn:nonlinear_continuity_equation} are found for example in some biological settings, where the speed of micro-organisms is impacted by their local density; see, for instance, the discussion of Keller-Segel model in \cite{PeDa:09}.
We expect that the results in this paper can be extended to such non-linear models, but for the sake of simplicity in this article, we consider only the linear continuity equation and exclude such nonlinearity in the flux from the rest of our discussion.

The continuity equation \eqref{eqn:linear_continuity_equation} is also naturally a critical component of compressible fluid dynamics such as the barotropic
compressible Navier-Stokes system 
\begin{equation*}
\left\{
\begin{aligned}
&\frac{\partial u}{\partial t} + \dive (b u) = 0
\\
&\frac{\partial}{\partial t} (bu) + \dive \big( (b \otimes b) u \big) + \nabla p - \dive (\tau(b)) = f
\\
&p = P(u),
\end{aligned} \right.
\end{equation*}
with appropriate boundary conditions if considered in a bounded domain. In this system, the velocity field $b$ is coupled with the density $u$ by another evolution PDE, leading to an even more complex non-linearity than in the previous examples.

A related model is the Stokes system,
\begin{equation*}
\left\{
\begin{aligned}
&\frac{\partial u}{\partial t} + \dive (b u) = 0
\\
& -\Delta b + \nabla p = f 
\\
&p = P(u),
\end{aligned} \right.
\end{equation*}
which considerably simplifies the momentum equation and the relation between $b$ and $u$.

There exists a large literature on the numerical analysis of compressible Navier-Stokes system to which we cannot do justice in a few sentences. We only briefly mention \cite{GaGeHeLa:08,GaHeLa:10} for the compressible Navier-Stokes and \cite{EyGaHeLa:10b,EyGaHeLa:10a,GaHeLa:09} for the Stokes system.
To the best of our knowledge however, the numerical analysis of these systems is only well-understood on Cartesian meshes or staggered grids, for example in \cite{GoLlMi:20}, that still rely on Cartesian mesh for the density.
Generally speaking, the regularity of discretized continuity equations such as \eqref{eqn:continuity_equation_scheme_0} remains poorly understood on non-cartesian meshes, which leads  to the main motivation and focus in the present work.

\subsection{Renormalized solutions}
Even without discretization, the well-posedness for advection equations such as \eqref{eqn:linear_continuity_equation} is in itself a delicate question when the velocity field $b$ is not Lipschitz.
By introducing the concept of renormalized solution,  the uniqueness and compactness of \eqref{eqn:linear_continuity_equation} was first obtained in \cite{DiLi:89} for velocity fields $b \in W^{1,p}$.
This was later improved in \cite{Am:04,Bo:01} to only $b \in BV$ with $\dive b\in L^1$.

Renormalized solutions are based on a simple but essential observation: Assume that $b$ and $u$ are smooth and satisfy the continuity equation 
\eqref{eqn:linear_continuity_equation}. Then for all $\beta \in C^1(\R)$, $\beta(u)$ is a solution of
\begin{equation} \label{eqn:renormalized_solution}
\begin{aligned}
\partial_t \beta(u) + \dive (b \, \beta(u)) + \dive b \, (\beta'(u) u - \beta(u)) = 0.
\end{aligned}
\end{equation}
A weak solution of \eqref{eqn:linear_continuity_equation} with a non-smooth field $b$ is said to be renormalized iff \eqref{eqn:renormalized_solution} holds in distributional sense for all $\beta \in C^1(\R)$ with $|\beta (\xi)|\leq C|\xi|$.
Moreover, equation \eqref{eqn:linear_continuity_equation} with a fixed field $b$ is said to be renormalized iff all its weak solutions are renormalized.

Basically, the renormalization property consists in stating that if $u$ is a weak solution then non-linear functions of $u$ are also solutions, with appropriate corrections if $\dive b\neq 0$. This directly implies the uniqueness of a weak solution $u$: Consider two weak solutions $u, v$ with $u(0,x)=v(0,x)$, if $u - v$ is a renormalized solution, then $|u - v|$ is also a weak solution. Hence $\|u(t,.) - v(t,.)\|_{L^1(\R^d)} \leq \|u(0,.) - v(0,.)\|_{L^1(\R^d)}=0$ and $u=v$.
Combining the uniqueness and the renormalization property directly provides compactness in the appropriate $L^p_{loc}$ sense as one can prove that
\begin{equation*}
\begin{aligned}
\textrm{weak-*}\, \lim \beta(u_n) = \beta( \textrm{weak-*}\, \lim u_n).
\end{aligned}
\end{equation*}

A common critical part in the proof of the renormalization property is a so-called commutator estimate.
Consider a classical convolution kernel $K_\epsilon$ together with $K_\epsilon\star_x u$ where $u$ solves~\eqref{eqn:linear_continuity_equation}. Commutator estimates arise when trying to write a similar equation on~$K_\epsilon\star_x u$: One then has that
\begin{equation*}
\begin{aligned}
\partial_t (K_\epsilon \star_x u) + \dive \big(b \; K_\epsilon \star_x u\big) = R_\epsilon,
\end{aligned}
\end{equation*}
where the remainder term $R_\epsilon$ can be written as a commutator
\begin{equation*}
\begin{aligned}
R_\epsilon(x) = \int \big(b(x) - b(y)\big) \nabla K_\epsilon (x-y) u(y) \;\rd y + \int \dive_x b(x) K_\epsilon(x - y) u(y) \;\rd y.
\end{aligned}
\end{equation*}
Commutator estimates then consists in proving that $R_\epsilon\to 0$. If it is possible to prove this, then it is straightforward to deduce the renormalization property \eqref{eqn:renormalized_solution}  by writing an equation on $\beta(K_\epsilon \star_x u)$ and passing to the limit~$\epsilon\to 0$.

Renormalized solutions are also connected to some form of propagation of quantitative regularity. 
It had already been noticed in \cite{AmLeMa:05} that renormalized solutions lead to some approximate differentiability on the solution.
But the first explicit propagation of regularity was obtained in \cite{CrLe:08} at the level of the characteristics flow.
The characteristics method used in \cite{CrLe:08} proved very fruitful with many later extensions. One can mention the study of SDEs in \cite{ChJa:18,FaLuTh:10,Zh:10}, the question of mixing under incompressible flows in \cite{BrOtSe:11,IyKiAlXu:14,Se:13}, well-posedness for velocity fields with less than $1$ derivative  in \cite{ChJa:10,JaMa:15},  and velocity fields obtained through a singular integral  in \cite{BoBoCr:16,BoCr:13,CoCrSp:15}.

The corresponding regularity at the PDE level can be derived by directly quantifying oscillations on the solution.
A first method to do so was introduced  in \cite{BeJa:13,BeJa:19} for non-linear continuity equations of the form \eqref{eqn:nonlinear_continuity_equation}.
In the linear case, sharper estimates were obtained in \cite{Le:18} and \cite{BrNg:18} through a somewhat similar approach.
We also mention \cite{BrJa:17,BrJa:18} which combines those methods with a new notion of weights; this was applied to the compressible Navier-Stokes equation with a large variety of laws of state and stress tensors.
A very different quantitative approach at the PDE level was studied in \cite{Se:17,Se:18}, using certain optimal transport distances.
All those results only propagate a weaker notion of regularity, weaker than full differentiability, usually some sort of $\log$ of derivative. It is indeed  not possible in general to bound any kind of Sobolev regularity on density when the velocity field is merely Sobolev; see some counterexamples for example in~\cite{AlCrMa:14,Ja:16}.

The approach that we follow in the present paper is inspired by the quantitative semi-norms introduced in~\cite{BeJa:13,BeJa:19}, which we briefly describe for this reason. 
The local compactness of a sequence of bounded functions $u_k \in L^p(\R^d)$ with $1 \leq p < \infty$ follows from the following  property:
\begin{equation} \label{eqn:compactness_criteria}
\begin{aligned}
\limsup_{n} \int_{\R^{2d}} \widetilde{K}_h(x-y) |u_n(x) - u_n(y)|^p \;\rd x \rd y \to 0 \quad \text{ as } h \to 0,
\end{aligned}
\end{equation}
where $\{\widetilde{K}_h\}_{h > 0}$ is any family of classical convolution kernels. Scaling \eqref{eqn:compactness_criteria} with a given rate of convergence in $h$ leads to various notions of semi-norms that measure intermediate regularity between $L^p$ and $W^{s,p}$ for any $s > 0$, and all of such regularities are strong enough to imply local compactness in $L^p$.

The particular family of kernels $\{\widetilde{K}_h\}$ proposed in \cite{BeJa:13,BeJa:19} results in semi-norms corresponding to a sort of log-scale derivatives that we denote here by $W_{\log, \theta}^p$.
The $W_{\log, \theta}^p$-regularity defined by kernels $\{\widetilde{K}_h\}$ was then proved to be propagated by \eqref{eqn:linear_continuity_equation} when the velocity field $b \in W^{1,p}$, $\dive b$ is bounded and $\dive b$ is compact or enjoys some similar $W_{\log,\theta}^p$-regularity.

Hence with such assumptions on $b$, the solutions of \eqref{eqn:linear_continuity_equation} are compact if the initial data are $W_{\log, \theta}^p$-regular. The bounds in \cite{Le:18} and \cite{BrNg:18} yield some more precise log-scale derivatives based on somewhat similar semi-norms.
The corresponding spaces have also received increasing attention  in other settings, see for instance~\cite{BoBrMi:00}.

When trying to extend the idea of quantifying oscillations in \cite{BeJa:13,BeJa:19} to our discrete setting, it appears natural to introduce an approximation of the continuous kernel $K_h$ on the mesh. In other words, we would like to estimate the regularity of the discrete density by something like
\begin{equation} \label{eqn:compactness_criteria_discrete}
\begin{aligned}
\limsup_{n} \sum_{i,j} \widetilde{K}_{i,j;(n)}^h |u_{i;(n)} - u_{j;(n)}|^p \pi_{i;(n)}\pi_{j;(n)} \to 0, \quad \text{ as } h \to 0,
\end{aligned}
\end{equation}
where $(\widetilde{K}_{i,j}^h)_{i,j}$ is an approximation of kernel $K_h$ and the double integral is replaced by a double summation over the mesh.

The main issue however is to identify the right family of kernels $(\widetilde{K}_{i,j}^h)_{i,j}$ so that the corresponding semi-norms are propagated by the discrete advection equation. This turns out to be extremely challenging on non-cartesian grids as a straightforward discretization of the kernels $\{\widetilde{K}_h\}$ used for the continuous equation does not appear to work.
The main technical contribution of this paper is a general method to find admissible approximation $(\widetilde{K}_{i,j}^h)_{i,j}$, extending the results in \cite{BeJa:13,BeJa:19} to  upwind schemes. This leads to the study of a non-symmetric diffusion equation on the mesh which we can solve and bound when the mesh show periodic patterns (the exact definition is given in Section \ref{subsec:set_up}).
\subsection{Some of the issues with non-cartesian grids}
At first glance, it may not be apparent why non-cartesian meshes lead to such additional difficulty. Eq.~\eqref{eqn:continuity_equation_scheme_0} may in fact be seen as an advection equation on a graph where the actual velocity field is correspond to some projection of the original velocity field that incorporates the structure of the graph. This means that the graph's topology can lead to additional oscillations in the solution in itself. This is made apparent in the  following straightforward example, that we are grateful to T. Gallou\"et and R. Herbin for pointing out. This shows that  even for very smooth or actually constant velocity fields at the continuous level, one may have strong oscillations in the solution at the discrete level.
\begin{exmp} \label{exmp:mouse_mesh}
Consider the constant velocity field $b(x) \equiv (1,0)$ in dimension~$2$ and the following non-cartesian discretization: Let $h_0$ be the discretization parameter, and use $\Z^2$ to index the cells. The cell indexed by $(i,j) \in \Z^2$ is given by
\begin{equation} \notag
C_{(j,k)} = \left\{
\begin{aligned}
[j h, (j+1) h) \times [k h, (k+1) h) \quad &\text{ if } k \text{ is even},
\\
[j h/2, (j+1) h/2) \times [k h, (k+1) h) \quad &\text{ if } k \text{ is odd}.
\end{aligned} \right.
\end{equation}
That simply means that we keep the vertical discretization $h$, but alternate a row with horizontal discretization $h$, with another row with discretization $h/2$.

Consider a discrete density $(u_{j,k})_{j,k}$ solving the upwind scheme \eqref{eqn:continuity_equation_scheme_0} over such a mesh for a discretization of the constant velocity field $b= (1,0)$.
Assume that the initial data $(u_{j,k}(0))_{j,k}$ is bounded in discrete $W^{1,1}$-norm, uniformly in $h_0$.
Then for any $t>0$, $(u_{j,k}(t))_{j,k}$ is bounded in the discrete $W^{s,1}$-norm, uniformly in $h_0$, if and only if $s < 1/2$.
\end{exmp}
This type of spurious oscillations created by the mesh itself are one of the reasons why the aforementioned quantitative methods (either in ODE or PDE level) have not been extended to non-Cartesian meshes. In fact, there exist only very few qualitative results of strong convergence for non Lipschitz velocity fields and non-Cartesian meshes. One can nevertheless mention \cite{Bo:12} which relies on the renormalization property at the limit. However because this kind of approach is not quantitative, it requires some a priori knowledge of the compactness of the divergence of the velocity field. This appears to make handling coupled non-linear models such as \eqref{eqn:linear_continuity_equation}-\eqref{eqn:poissoncoupling} out of reach.

When one is not trying to handle at the same time non-Cartesian meshes and non-Lipschitz coefficients, stronger results can be obtained.
On non-Cartesian meshes, 
we refer for instance to \cite{Me:08,MeVo:07} for divergence-free velocity field that are Lipschitz in both space and time, and to \cite{DeLa:11} for autonomous (time-independent) Lipschitz velocity fields with non-zero divergence.
For non-Lipschitz velocity fields on Cartesian meshes, one can obtain quantitative convergence results in some suitable weak distances.
When the velocity field is in the appropriate Sobolev space with one-sided bounded divergence, the upwind scheme was proved to converge at rate of $1/2$ in \cite{ScSe:17} in some weak topology.
When the velocity field is one-sided Lipschitz continuous, the convergence with rate $1/2$ of the upwind scheme in Wasserstein distance was proved in \cite{DeLaVa:17}.

To the best of our knowledge however, this article is the first to provide a general approach to the compactness of solutions to discrete advection equations with non-Lipschitz coefficients and non-Cartesian meshes, even if we still require some restrictions on the mesh such as periodic patterns.

Furthermore, the compactness result in this paper is directly applicable to some of the coupled systems discussed at the beginning of this introduction. We are in particular able to derive the compactness of discretizations of the non-linear coupled system~\eqref{eqn:linear_continuity_equation}-\eqref{eqn:poissoncoupling}. The exact result is stated later in this first section. We remark here that the velocity field $b$ obtained from~$u$ through~\eqref{eqn:poissoncoupling} is naturally bounded in  $W^{1,p}$ for  all $1 < p \leq \infty$, if we assume $g(u) \in L^1 \cap L^\infty$. Of course we cannot know a priori the compactness of $\dive b$ but we have the simple relation $\dive b = g(u)$. This is where quantitative, explicit estimates prove critical as we are able to conclude through some sort of Gronwall argument.

However more complex coupled systems would present unique challenges for our approach: This is notably the case of compressible fluid dynamics. Energy estimates would provide Sobolev, $H^1_x$ bounds on the velocity. However the divergence of velocity  is generically unbounded, which would prevent us from applying our method in any straightforward manner. Instead this would likely require the introduction of weights such as was done in~\cite{BrJa:17,BrJa:18} at the continuous level.

\subsection{A basic example of setting for the linear continuity equation} \label{subsec:set_up}

Considering the linear continuity equation \eqref{eqn:linear_continuity_equation},
we introduce here  its basic discretization on a polygon mesh $(\mathcal{C}, \mathcal{F}) = \big( \{V_i\}_{i \in \mathcal{V}}, \{S_{i,j}\}_{(i,j) \in \mathcal{E}} \big)$ over a bounded domain $\Omega \subset \R^d$, which we defined as the following:
\begin{itemize}
\item The pair of indices $(\mathcal{V}, \mathcal{E})$ form a finite graph.
\item
Each cell $V_i$ is a $d$-dim polygon in $\R^d$. The intersection of two cells $V_i$ and $V_j$ is nonempty if and only if $(i,j) \in \mathcal{E}$. In that case $S_{i,j} \defeq V_i \cap V_j$ is a $(d-1)$-dim polygon in $\R^d$.
\item The domain is covered by the mesh: $\Omega \subset \bigcup_{i \in \mathcal{V}} V_i$.

\end{itemize}
We define the discretization size of the mesh as $\delta x \defeq \sup_{i \in \mathcal{V}} \mbox{diam} (V_i)$, where $\mbox{diam} ()$ represents the diameter.

As a first example, we consider the following semi-discrete upwind scheme 
\begin{equation}\label{eqn:continuity_equation_scheme_polygon}
\left\{
\begin{aligned}
&\frac{\rd}{\rd t} u_{i}(t) = \frac{1}{|V_i|} \sum_{i' :\, (i,i') \in \mathcal{E}} \Big( a_{i,i'}(t) \, u_{i'}(t) - a_{i',i}(t) \, u_i(t) \Big), & \text{if } i \in \mathcal{V} \text{ and } V_i \!+\! B_{\delta x} \subset \Omega,
\\
&u_i(t) \equiv 0, & \text{if } i \in \mathcal{V} \text{ but } V_i \!+\! B_{\delta x} \nsubseteq \Omega,
\\
&a_{i,j}(t) = \left( \frac{1}{|B_{\delta x}|} \int_{B_{\delta x}} \int_{S_{i,j}} b(x+y,t) \cdot \bm{N}_{i,j} \;\rd y \rd x \right)^+, & \text{if } (i,j) \in \mathcal{E} \text{ and } S_{i,j} \!+\! B_{\delta x} \subset \Omega,
\\
&u_i(0) = \frac{1}{|V_i|} \int_{V_i} u_0(x) \;\rd x, & \text{if } i \in \mathcal{V} \text{ and } V_i \!+\! B_{\delta x} \subset \Omega,
\end{aligned} \right.
\end{equation}
where $(u_i(t))_{i \in \mathcal{V}}$ are the discrete density on the mesh,
$\bm{N}_{i,j}$ is the unit normal vector on $S_{i,j}$, satisfying $\bm{N}_{i,j} = - \bm{N}_{j,i}$. The functions $b(t,x)$ and $u_0(x)$ are respectively the velocity field and initial condition 
in the linear continuity equation \eqref{eqn:linear_continuity_equation}.
Also, throughout the paper, for $s \in \R$ we use the notation $s^+ = s \vee 0 = \max\{s,0\}$ and $s^- = -(s \wedge 0) = -\min\{s,0\}$.

The total mass on the mesh is given by $\sum_{i \in \mathcal{V}} u_i |V_i|$. It is easy to verify that
the scheme  conserves mass
except near the boundary of $\Omega$, where some leaking may occur.
Such leaking effect can be controlled by no flux (no outward flux) condition of the velocity field or by a priori estimating the distribution of density.

Before we can rigorously state any compactness result, we still need to clarify our assumptions on the mesh.
Throughout the paper, for $A, B \subset \R^d$, we use the notation
\begin{equation*}
\begin{aligned}
A + B \defeq \{x \in \R^d, x = a + b, a \in A, b \in B\}.
\end{aligned}
\end{equation*}
Also, for $x \in \R^d$ we denote $A + x \defeq A + \{x\}$, which is a translation of set $A$ on $\R^d$.

We say that a mesh has a periodic pattern if the following holds:
\begin{defi}\label{defi:periodic_mesh_polygon}
Let  $(\mathcal{C}, \mathcal{F}) = \big( \{V_i\}_{i \in \mathcal{V}}, \{S_{i,j}\}_{(i,j) \in \mathcal{E}} \big)$ be a polygon mesh over $\Omega \subset \R^d$ and let $\mathcal{V}_0 \subset \mathcal{V}$.
The mesh is periodic with pattern $\mathcal{V}_0$ if it satisfies the following properties: The set $\bigcup_{i \in \mathcal{V}_0} V_i$ is connected and one has $\big( \bigcup_{i \in \mathcal{V}_0} V_i \big) + B_{\delta x}\subset \Omega$.
There exists a translation group action
\begin{equation*}
\begin{aligned}
[m] (V_i) \defeq V_i + \sum_{k=1}^d m_{k} L_k,
\quad
\forall m \in \Z^d, i \in \mathcal{V},
\end{aligned}
\end{equation*}
where $L_1, \dots, L_d \in \R^d$ are linearly independent vectors, such that
\begin{equation*}
\begin{aligned}
\sum_{m \in \Z^d} \sum_{i \in \mathcal{V}_0} \mathbbm{1}_{[m]V_i} (x) = 1 \quad \text{for a.e. } x \in \R^d.
\end{aligned}
\end{equation*}
Moreover, there exists an injective map
\begin{equation*}
\begin{aligned}
\sigma : \mathcal{V} &\mapsto \Z^d \times \mathcal{V}_0 ,
\\
i &\mapsto ([n],i_0).
\end{aligned}
\end{equation*}
If $\sigma(i) = ([n],i_0)$ and $([m]+[n],i_0) \in \sigma(\mathcal{V})$, define $[m](i) \defeq \sigma^{-1} \big( ([m]+[n],i_0) \big) \in \mathcal{V}$.
Then one has
\begin{equation*}
\begin{aligned}
[m](V_i) = V_{[m](i)}, \quad \text{if } i \in \mathcal{V} \text{ and } [m](i) \in \mathcal{V}.
\end{aligned}
\end{equation*}

If the mesh is periodic with pattern $\mathcal{V}_0$ we call $|\mathcal{V}_0|$ the pattern size; of course for a given mesh, the choice of $\mathcal{V}_0$ and $|\mathcal{V}_0|$ may not be unique. If one can choose $\mathcal{V}' \subseteq \mathcal{V}, \mathcal{E}' \subseteq \mathcal{E}$ such that $(\mathcal{C}', \mathcal{F}') = \big( \{V_i\}_{i \in \mathcal{V}'}, \{S_{i,j}\}_{(i,j) \in \mathcal{E}'} \big)$ forms a mesh over $\Omega' \subseteq \Omega$, and is a periodic mesh by the definition above, then we say that $(\mathcal{C}, \mathcal{F})$ is periodic over $\Omega'$.

\end{defi}

We also require some additional assumptions on the meshes, though those are rather standard. 
Throughout the discussion, any mesh $(\mathcal{C}, \mathcal{F})$ of our interest should satisfy that for all $i \in \mathcal{V}$ and $x \in \R^d$:
\begin{equation} \label{eqn:mesh_comparable_polygon}
\begin{aligned}
C^{-1} \delta x \leq &\diam(V_{i}) \leq C\delta x, \quad C^{-1} (\delta x)^d \leq |V_{i}| \leq C(\delta x)^d,
\\
&\big| \{k \in \mathcal{V} : B(x;\delta x) \cap V_{k} \neq \varnothing\} \big| \leq C,
\end{aligned}
\end{equation}
for some uniform constant $C$. These conditions exclude some pathological situations where some parts of the mesh would be too singular  in some regard.

Finally, we observe that since we are considering the limit to the continuous equation, then we naturally expect the discretization size to vanish. Namely, let 
\begin{equation*}
\begin{aligned}
\big(\mathcal{C}^{(n)}, \mathcal{F}^{(n)}\big) = \big( \{V_{i;(n)}\}_{i \in \mathcal{V}^{(n)}}, \{S_{i,j;(n)}\}_{(i,j) \in \mathcal{E}^{(n)}} \big), \quad n \in \N_+
\end{aligned}
\end{equation*}
be a family of meshes and let $\delta x_{(n)}$ , $n \in \N_+$ denote the discretization sizes, then we ask that
\begin{equation} \label{eqn:mesh_convergence}
\begin{aligned}
\delta x_{(n)} \to 0 \text{ as } n \to \infty. 
\end{aligned}
\end{equation}

We are now ready to state a first example of our compactness result:
\begin{thm} \label{thm:periodic_compactness_result}
Consider $T > 0$ and a bounded domain $\Omega \subset \R^d$ with piecewise smooth boundary.
Let $b(t,x)$ be a velocity field with $b \in L^{\infty}_tL^{1}_x \cap L^{q}_t(W^{1,q}_x) \cap L^{1}_x(W^{s,1}_t) ([0,T] \times \Omega)$ and divergence $\dive_x b \in L^{\infty}_t(L^{\infty}_x) \cap L^{1}_t(W^{s,1}_x)([0,T] \times \Omega)$,  for some $1 < q \leq \infty$, $0 < s \leq 1$.
Let $u_0\in L^\infty \cap W^{s,1}(\Omega)$ be the initial data.

Consider
a sequence of polygonal meshes $\{(\mathcal{C}^{(n)}, \mathcal{F}^{(n)})\}_{n=1}^\infty$ over $\Omega$, having discretization size $\delta x_{(n)} \to 0$, satisfying the structural assumptions \eqref{eqn:mesh_comparable_polygon} with some uniform constant, and being periodic on $\Omega$ with their pattern size also uniformly bounded.
Let $(u_{i;(n)}(t))_{i \in \mathcal{V}^{(n)}}$ be  solutions to the semi-discrete scheme \eqref{eqn:continuity_equation_scheme_polygon} and denote by $u_{(n)}(t,x)$ the piecewise constant function extending  $(u_{i;(n)}(t))_{i \in \mathcal{V}^{(n)}}$. 
Assume finally that the total mass $\sum_{i \in \mathcal{V}^{(n)}} u_{i;(n)}(T) |V_{i;(n)}| \to \|u_0\|_{L^1}$ as $n \to \infty$.
Then
\begin{equation*}
\begin{aligned}
u_{(n)}(t,x) \text{ is compact in } L^1([0,T] \times \Omega).
\end{aligned}
\end{equation*}

\end{thm}
The proof of the theorem is postponed to Section~\ref{sec:quantitative_regularity}, where it follows from the propagation of some discrete regularity of the form~\eqref{eqn:compactness_criteria_discrete}.
%
\subsection{The more complete setting} \label{subsec:set_up_partition_of_unity}

We demonstrate the potential of our method by also deriving the compactness for a simple non-linear coupled system, namely
\begin{equation} \label{eqn:Poisson_coupling}
\left\{
\begin{aligned}
&\partial_t u(t,x) + \dive_x \big(b(t,x) u(t,x)\big) = 0
\\
&b(t,x) = \nabla_x \phi(t,x), \quad - \Delta_x \phi(t,x) = g(u(t,x)).
\end{aligned} \right.
\end{equation}
However, the setting of polygon meshes described above may no longer be the most appropriate.
The difficulty comes from the coupling of numerical schemes between the elliptic Poisson equation, for which one may want to use finite elements for example, and the hyperbolic advection equation for which we use upwind schemes. This is one of the motivations for our more general formulation.

We define cell functions, face functions and meshes that replace the polygonal cells. We discuss later in subsection~\ref{subsec:connection_two_settings} how the previous polygon meshes can be related to this formulation.
\begin{defi} \label{defi:partition_of_unity_cell}
Consider a piecewise differentiable function $\chi$ with value $0 \leq \chi \leq 1$ and vector values functions $\{\bm{n}_j\}_{j = 1}^m$ on $\R^d$.
Then $\chi$ is said to be a cell function with $\{\bm{n}_j\}_{j = 1}^m$ as its face functions if
\begin{equation*}
\begin{aligned}
\textstyle \sum_{j=1}^m \bm{n}_j(x) = - \nabla \chi(x) \quad \text{for a.e. } x \in \R^d  \quad \text{ and } \quad \supp \bm{n}_j \subseteq \supp \chi, \quad \forall j \in \mathcal{V}.
\end{aligned}
\end{equation*}
\end{defi}
With the cell functions and face functions defined, we give the following definition of meshes.
\begin{defi} \label{defi:partition_of_unity_mesh}
We define as a generalized mesh over $\Omega \subseteq \R^d$ a pair $(\mathcal{C}, \mathcal{F}) = \big( \{\chi_i\}_{i \in \mathcal{V}}, \{\bm{n}_{i,j}\}_{(i,j) \in \mathcal{E}} \big)$ satisfying the following conditions:
\begin{itemize}
\item The pair of indices $(\mathcal{V}, \mathcal{E})$ forms a finite graph.
\item If $i \in \mathcal{V}$ and $\supp \chi_i \subset \Omega$, then $\chi_i$ must be a cell function with $\{\bm{n}_{j,i}\}_{(j,i) \in \mathcal{E}}$ as its face functions.
  \item Finally,
\begin{equation*}
\begin{aligned}
\textstyle \sum_{i \in \mathcal{V}} \chi_i(x) = 1,\quad \forall x \in \Omega \quad \text{ and } \quad \bm{n}_{i,j} = - \bm{n}_{j,i}, \quad
\forall i,j \in \mathcal{V}.
\end{aligned}
\end{equation*}
\end{itemize}
We also extend $\{\bm{n}_{i,j}\}_{(i,j) \in \mathcal{E}}$ to $\{\bm{n}_{i,j}\}_{i,j \in \mathcal{V}}$ by defining $\bm{n}_{i,j} = \bm{n}_{j,i} = 0$ for $(i,j) \notin \mathcal{E}$.

The discretization size of a mesh $(\mathcal{C}, \mathcal{F})$ is defined as $\delta x \defeq \max_{i \in \mathcal{V}} \diam ( \supp \chi_i )$. The volume of cell $i \in \mathcal{V}$ is defined as $\pi_i \defeq \|\chi_i\|_{L^1}$.
\end{defi}

The semi-discrete scheme we consider in this paper is of form 
\begin{equation}\label{eqn:continuity_equation_scheme}
\left\{
\begin{aligned}
&\frac{\rd}{\rd t} u_{i}(t) = \frac{1}{\pi_i} \sum_{j: (i,j)\in \mathcal{E}} \Big( a_{i,j}(t)  \, u_{j}(t) - a_{j,i}(t) \, u_i(t) \Big), \quad && \text{if } i \in \mathcal{V} \text{ and } \supp \chi_i \subset \Omega,
\\
&u_i(t) \equiv 0 , \quad && \text{if } i \in \mathcal{V} \text{ but } \supp \chi_i \nsubseteq \Omega.
\end{aligned} \right.
\end{equation}
Given $b(t,x)$ and $u_0(x)$ as the field and initial condition 
in linear continuity equation \eqref{eqn:linear_continuity_equation} respectively, we choose
the coefficients and initial data in the scheme as
\begin{equation}\label{eqn:continuity_equation_scheme_coef}
\left\{
\begin{aligned}
&u_i(0) = \frac{1}{\pi_i} \int_{\R^d} \chi_i(x) \, u_0(x) \;\rd x, \quad && \text{if } i \in \mathcal{V} \text{ and } \supp \chi_i \subset \Omega,
\\
&a_{i,j}(t) = \int_{\R^d} \Big(b(t,x) \cdot \bm{n}_{i,j}(x) \Big)^+ \;\rd x, \quad && \text{if } (i,j) \in \mathcal{E} \text{ and } \supp \bm{n}_{i,j} \subset \Omega.
\end{aligned} \right.
\end{equation}
Notice that if $\supp \chi_i \subset \Omega$, then for all $j \in \mathcal{V}$, either one has $(i,j) \in \mathcal{E}$, $\supp \bm{n}_{i,j} \subset \chi_i \subset \Omega$ or one has $(i,j) \notin \mathcal{E}$, $\bm{n}_{i,j} = 0$ and $\supp \bm{n}_{i,j} \subset \Omega$ trivially holds. 
Hence $a_{i,j}$ and $a_{i,j}$ in \eqref{eqn:continuity_equation_scheme} are always well-defined.
In addition, we let 
\begin{equation*}
\begin{aligned}
a_{i,j}(t) \equiv 0 \quad \text{if } (i,j) \notin \mathcal{E} \text{ or } \supp \bm{n}_{i,j} \nsubseteq \Omega.
\end{aligned}
\end{equation*}
Then the summation in \eqref{eqn:continuity_equation_scheme} can be taken over all $j \in \mathcal{V}$, instead of only $j$ such that $(i,j)\in \mathcal{E}$, and the scheme is essentially unchanged. In some of the later calculations, this adaption can be convenient.

The structural assumptions to meshes we have made should also be adapted.
In particular, the new definition of being periodic is the following:
\begin{defi}\label{defi:periodic_mesh}
Let  $(\mathcal{C}, \mathcal{F}) = \big( \{\chi_i\}_{i \in \mathcal{V}}, \{\bm{n}_{i,j}\}_{(i,j) \in \mathcal{E}} \big)$ be a mesh over $\Omega \subset \R^d$ and let $\mathcal{V}_0 \subset \mathcal{V}$. We say that $(\mathcal{C}, \mathcal{F})$ is a periodic mesh with pattern $\mathcal{V}_0$ if it satisfies the following properties: 
\begin{itemize}
  \item The set  $\bigcup_{i \in \mathcal{V}_0} \supp \chi_i$ is connected and one has $\bigcup_{i \in \mathcal{V}_0} \supp \chi_i \subset \Omega$.
\item There exists a translation group action
\begin{equation*}
\left\{
\begin{aligned}
([m] \chi_i) \big(x \big) &\defeq \textstyle \chi_i \left(x - \sum_{k=1}^d m_{k} L_k \right),
\\
([m] \bm{n}_{i,j}) \big(x \big) &\defeq \textstyle \bm{n}_{i,j} \left(x - \sum_{k=1}^d m_{k} L_k \right),
\end{aligned} \right.
\quad
\forall [m] \in \Z^d, i,j \in \mathcal{V},
\end{equation*}
where $L_1, \dots, L_n \in \R^d$ are linearly independent vectors, such that
\begin{equation*}
\begin{aligned}
\sum_{[m] \in \Z^d} \sum_{i \in \mathcal{V}_0} [m] \chi_i (x) = 1 \quad \forall x \in \R^d.
\end{aligned}
\end{equation*}
\item Moreover, there exists an injective map
\begin{equation*}
\begin{aligned}
\sigma : \mathcal{V} &\mapsto \Z^d \times \mathcal{V}_0 ,
\\
i &\mapsto ([n],i_0).
\end{aligned}
\end{equation*}
\end{itemize}
If $\sigma(i) = ([n],i_0)$ and $([m]+[n],i_0) \in \sigma(\mathcal{V})$, define $[m](i) \defeq \sigma^{-1} \big( ([m]+[n],i_0) \big) \in \mathcal{V}$.
Then one has
\begin{equation*}
\begin{aligned}
[m] \chi_i = \chi_{[m](i)}, \quad [m] \bm{n}_{i,j} = \bm{n}_{[m](i),[m](j)}, \quad \text{if } i, [m](i) \in \mathcal{V}.
\end{aligned}
\end{equation*}

If the mesh is periodic with pattern $\mathcal{V}_0$ we call $|\mathcal{V}_0|$ the pattern size. As before, for a given mesh, the choice of $\mathcal{V}_0$ and $|\mathcal{V}_0|$ may not be unique.
If one can choose $\mathcal{V}' \subseteq \mathcal{V}, \mathcal{E}' \subseteq \mathcal{E}$ such that $(\mathcal{C}', \mathcal{F}') = \big( \{\chi_i\}_{i \in \mathcal{V}'}, \{\bm{n}_{i,j}\}_{(i,j) \in \mathcal{E}'} \big)$ forms a mesh over $\Omega' \subseteq \Omega$, and is a periodic mesh by the definition above, then we say that $(\mathcal{C}, \mathcal{F})$ is periodic over $\Omega'$.

\end{defi}

The other structural assumptions on the mesh can be adapted in a straightforward manner.
We limit our discussion to meshes $(\mathcal{C}, \mathcal{F})$ that satisfy that for all $i \in \mathcal{V}$, $(j,j') \in \mathcal{E}$ and $x \in \R^d$:
\begin{equation} \label{eqn:mesh_comparable_1}
\begin{aligned}
&C^{-1} \delta x \leq \diam ( \supp \chi_{i}) \leq C\delta x, \quad C^{-1} (\delta x)^d \leq \|\chi_{i}\|_{L^1} \leq C(\delta x)^d,
\\
&\delta x \, \|\bm{n}_{j,j'}\|_{L^\infty} \leq C
, \quad \big| \{k \in \mathcal{V} : (\supp \chi_{k}) \cap B(x;\delta x) \neq \varnothing\} \big| \leq C,
\end{aligned}
\end{equation}
for some uniform constant $C$.
Also, we assume that the discretization size vanishes when considering a family of meshes,
\begin{equation*}
\begin{aligned}
\big(\mathcal{C}^{(n)}, \mathcal{F}^{(n)}\big) = \big( \{\chi_{i;(n)}\}_{i \in \mathcal{V}^{(n)}}, \{\bm{n}_{i,j;(n)}\}_{(i,j) \in \mathcal{E}^{(n)}} \big), \quad n \in \N_+
\end{aligned}
\end{equation*}
that \eqref{eqn:mesh_convergence} holds where  $\delta x_{(n)}$ , $n \in \N_+$ denote the discretization sizes as given in Definition~\ref{defi:partition_of_unity_mesh}. 

With this more general formulation, one can couple the upwind scheme for advection and the finite elements for Poisson equation in the following way:
Consider convex bounded domains with piecewise smooth boundary $\Omega_{v} \subset \Omega_{e} \subset \R^d$.
Let $(\mathcal{P}, \mathcal{N})$ be a finite element discretization of $\Omega_{e}$, where $\mathcal{P}$ is the set of shape functions and $\mathcal{N}$ is the set of nodal variables.
Choose the mesh $(\mathcal{C}, \mathcal{F}) =  \big( \{\chi_i\}_{i \in \mathcal{V}}, \{\bm{n}_{i,j}\}_{(i,j) \in \mathcal{E}} \big)$ over $\Omega_{v}$ as in Definition~\ref{defi:partition_of_unity_mesh} such that $\mathcal{C} \subset \mathcal{P}$.

The coupled system \eqref{eqn:Poisson_coupling} is numerically discretized through  by \eqref{eqn:continuity_equation_scheme}. The coefficients $(a_{i,j})_{i,j \in \mathcal{V}}$ derive \eqref{eqn:continuity_equation_scheme_coef} where the field $b(t,x)$ is now a solution of the variational problem
\begin{equation} \label{eqn:Poisson_equation_scheme}
\left\{
\begin{aligned}
&\phi(t,\cdot) \in \mathcal{P}, \quad
\int_{\R^d} \nabla v(x) \cdot \nabla \phi(t,x) \;\rd x = \int_{\R^d} v(x) g^\chi(t,x) \;\rd x, \quad \forall v \in \mathcal{P},
\\
& g^\chi(t,x) = \sum_{i \in \mathcal{V}} g\big(u_{i}(t)\big) \chi_{i}(x),
\\
&b(t,x) = \nabla \phi(t,x).
\end{aligned} \right.
\end{equation}
We only consider here Dirichlet boundary conditions for \eqref{eqn:Poisson_equation_scheme}, as Neumann boundary conditions would require an extra condition
$\int_{\R^d} g^\chi(t,x) \;\rd x = 0$ for all $t$, which does not naturally hold when $g$ contains some nonlinear function of the density $u$.

When investigating this more complex coupling, we require further structural assumptions on the pair of finite element $(\mathcal{P}, \mathcal{N})$ and the mesh $(\mathcal{C}, \mathcal{F})$ of our interest.
Namely, the exact solution of $- \Delta \widetilde{\phi} = g^\chi$ and its approximated solution $\phi$ of the finite element variational method
\begin{equation*}
\begin{aligned}
\phi \in \mathcal{P}, \quad
\int_{\R^d} \nabla v(x) \cdot \nabla \phi(x) \;\rd x = \int_{\R^d} v(x) u(x) \;\rd x, \quad \forall v \in \mathcal{P},
\end{aligned}
\end{equation*}
are assumed to satisfy the priori estimates
\begin{equation} \label{eqn:coupling_finite_element_bound}
\begin{aligned}
\|\widetilde{\phi} - \phi\|_{H^1(\Omega_{e})} &\leq C \delta x \|\widetilde{\phi}\|_{H^2(\Omega_{e})},
\\
\|\phi\|_{W^{1,\infty}(\Omega_{e})} &\leq C \|\widetilde{\phi}\|_{W^{1,\infty}(\Omega_{e})}.
\end{aligned}
\end{equation}
Such a priori estimates can be proved under rather mild conditions on the finite element discretization; we refer to Section 5.4 and 8.1 of \cite{BrSc:08}.

We are now ready to state our main theorem on this coupled system, whose proof is again postponed to Section \ref{sec:quantitative_regularity}.
\begin{thm} \label{thm:Poisson_propagating}
Consider bounded domains $\Omega_{v} \subset \Omega_{e} \subset \R^d$ with piecewise smooth boundary, a sequence of finite element discretizations $\{(\mathcal{P}^{(n)}, \mathcal{N}^{(n)})\}_{n=1}^{\infty}$ on $\Omega_{e}$ and a sequence of meshes 
\begin{equation*}
\begin{aligned}
\{(\mathcal{C}^{(n)}, \mathcal{F}^{(n)})\}_{n=1}^{\infty} = \{ \big( \{\chi_{i;(n)}\}_{i \in \mathcal{V}^{(n)}}, \{\bm{n}_{i,j;(n)}\}_{i,j \in \mathcal{V}^{(n)}} \big)\}_{n=1}^{\infty}
\end{aligned}
\end{equation*}
over $\Omega_{v}$ as in Definition~\ref{defi:partition_of_unity_mesh}, satisfying $\mathcal{C}^{(n)} \subset \mathcal{P}^{(n)}$.
Assume that the
discretization size $\delta x_{(n)} \to 0$, the meshes $\{(\mathcal{C}^{(n)}, \mathcal{F}^{(n)})\}_{n=1}^{\infty}$ satisfy the structural assumptions \eqref{eqn:mesh_comparable_1} by some uniform constant, and each mesh $(\mathcal{C}^{(n)}, \mathcal{F}^{(n)})$ is periodic on $\Omega_{v}$ with pattern size uniformly bounded.
Moreover, assume that the finite element discretization $(\mathcal{P}^{(n)}, \mathcal{N}^{(n)})$ satisfy the
a priori estimates
\eqref{eqn:coupling_finite_element_bound} with some uniform constants.

Consider bounded, Lipschitz and concave nonlinearity $g : [0,+\infty) \to \R$  with $g(0)=0$.
Assume that the
initial data $u_0$ satisfies $u_0 \in L^\infty \cap W^{s,1}(\Omega_{v})$ for some $s > 1$, and $\dist(\supp u_0, \partial \Omega_{v}) > 0$.
For all $n \in \N_+$, let $(u_{i;(n)}(t))_{i \in \mathcal{V}^{(n)}}$ and $(a_{i,j;(n)})_{i,j \in \mathcal{V}^{(n)}}$ be the solution of the coupled scheme \eqref{eqn:continuity_equation_scheme}, \eqref{eqn:continuity_equation_scheme_coef} and \eqref{eqn:Poisson_equation_scheme} solved on 
$(\mathcal{C}^{(n)}, \mathcal{F}^{(n)})$ and $(\mathcal{P}^{(n)}, \mathcal{N}^{(n)})$.
Define
\begin{equation*}
\begin{aligned}
u_{(n)}^\chi(t,x) \defeq \sum_{i \in \mathcal{V}^{(n)}} \chi_{i;(n)}(x) u_{i;(n)}(t).
\end{aligned}
\end{equation*}
Then there exists $T > 0$ such that
\begin{equation*}
\begin{aligned}
u_{(n)}^\chi \text{ is compact in } L^1([0,T] \times \R^d).
\end{aligned}
\end{equation*}
Moreover, $T$ could be arbitrarily large by choosing large $\Omega_{v}$ such that $\dist(\supp u_0, \partial \Omega_{v}) \to \infty$.

\end{thm}

\subsection{Connection between the two settings} \label{subsec:connection_two_settings}

We now discuss why the polygon meshes in Section~\ref{subsec:set_up} can be understood as a special case of the more general setting in Section~\ref{subsec:set_up_partition_of_unity}.

Starting with any polygon mesh $(\mathcal{C}, \mathcal{F}) = \big( \{V_i\}_{i \in \mathcal{V}}, \{S_{i,j}\}_{(i,j) \in \mathcal{E}} \big)$ over $\Omega \subset \R^d$ with discretization size $\delta x$, one can construct a mesh as in Definition~\ref{defi:partition_of_unity_mesh} through the following process. First, add more cells to $(\mathcal{C}, \mathcal{F})$ if necessary, to ensure $\Omega + B_{\delta x} \subset \bigcup_{i \in \mathcal{V}} V_i$.
Second, construct the extended mesh $\big( \{\chi_i\}_{i \in \mathcal{V}}, \{\bm{n}_{i,j}\}_{(i,j) \in \mathcal{E}} \big)$ with the cell and face functions
\begin{equation} \label{eqn:partition_of_unity_polygon}
\begin{aligned}
\chi_i(x) &= \frac{1}{|B_r(0)|} \int_{B_r(0)} \mathbbm{1}_{V_i} (x - y) \;\rd y, \quad \forall i \in \mathcal{V},
\\
\bm{n}_{i,j}(x) &= \int_{S_{i,j}} \frac{1}{|B_r(0)|} \mathbbm{1}_{B_r(0)}(x-y) \bm{N}_{i,j} \;\rd y, \quad \forall (i,j) \in \mathcal{E}.
\end{aligned}
\end{equation}
where $\bm{N}_{i,j}$ is the unit normal vector of $S_{i,j}$.
It is then straightforward to check that if $i \in \mathcal{V}$ and $\supp \chi_i \subset \Omega$, then $\chi_i$ is indeed a cell function with $\{\bm{n}_{j,i}\}_{(j,i) \in \mathcal{E}}$ as its face functions.
Also, one has
$\sum_{i \in \mathcal{V}} \chi_i(x) = 1$, $\forall x \in \Omega$ and $\bm{n}_{i,j} = - \bm{n}_{j,i}$, $\forall i,j \in \mathcal{V}$. Therefore, this construction does yield a mesh over $\Omega$ as in Definition~\ref{defi:partition_of_unity_mesh}.

With this construction, the upwind scheme \eqref{eqn:continuity_equation_scheme} for $(\{\chi_i\}, \{\bm{n}_{i,j}\})$ with coefficients \eqref{eqn:continuity_equation_scheme_coef} and the upwind scheme \eqref{eqn:continuity_equation_scheme_polygon} for $(\{V_i\}, \{S_{i,j}\})$
are very similar.
The conditions $\supp \chi_i \subset \Omega$ and $\supp \bm{n}_{i,j} \subset \Omega$ are now nothing but $V_i + B_{\delta x} \subset \Omega$ and $S_{i,j} + B_{\delta x} \subset \Omega$.
It is also immediate to see that 
\begin{equation*}
\begin{aligned}
\pi_i = \int \chi_i(x) \;\rd x = |V_i|.
\end{aligned}
\end{equation*}
Finally, the coefficients $a_{i,j}$ in \eqref{eqn:continuity_equation_scheme_coef} (when $\supp \bm{n}_{i,j} \subset \Omega$) now read
\begin{equation*} 
\begin{aligned}
a_{i,j} &\defeq \frac{1}{|B_{\delta x}(0)|} \int_{B_{\delta x}} \int_{S_{i,j}} \Big( b(x+y) \cdot \bm{N}_{i,j} \Big)^+ \;\rd y \rd x, 
\end{aligned}
\end{equation*}
which is only slightly different from the coefficients $a_{i,j}$ in \eqref{eqn:continuity_equation_scheme_polygon}, though we do emphasize the order of $(\cdot)^+$ and integration in this formula. Notice that if $b(x)$ is constant, then the integrand $b \cdot N_{i,j}$ is also constant, hence $a_{i,j}$ given by \eqref{eqn:continuity_equation_scheme_coef} and \eqref{eqn:continuity_equation_scheme_polygon} coincide.
So, when $b(x)$ has $W^{1,p}$ regularity, we can naturally expect the two ways of determining  $a_{i,j}$ to differ only by a term that is vanishing in $L^p$ as discretization size goes to zero.

As mentioned earlier, both compactness results in Theorem~\ref{thm:periodic_compactness_result} and Theorem~\ref{thm:Poisson_propagating} are derived by propagating some discrete regularity like \eqref{eqn:compactness_criteria_discrete}, where
the discrete density $(u_{i;(n)}(t))_{i \in \mathcal{V}^{(n)}}$ are both governed by the upwind scheme, with the coefficients originally defined in different ways but now formulated all in the setting of Section~\ref{subsec:set_up_partition_of_unity}.
In Section~\ref{sec:main_and_outline}, we give the precise definition of such regularity as Definition~\ref{defi:semi-norm_discrete} and state the propagation of such regularity by the upwind scheme as Theorem~\ref{thm:periodic_regularity_result}. Theorem~\ref{thm:periodic_regularity_result} can then be applied to prove both Theorem~\ref{thm:periodic_compactness_result} and Theorem~\ref{thm:Poisson_propagating}.

While the main elements of the proofs rely on the same result, namely Theorem~\ref{thm:periodic_regularity_result}, we do need to mention that some settings in Theorem~\ref{thm:periodic_compactness_result} and Theorem~\ref{thm:Poisson_propagating} are not identical.
Apart from the aforementioned choice of $a_{i,j}$, the way we extend discrete density to continuous functions are also slightly different: In Theorem~\ref{thm:periodic_compactness_result}, $u_n$ is defined as piecewise constant on each cell, while in Theorem~\ref{thm:Poisson_propagating}, $u_n^\chi$ is reconstructed from cell functions and is thus not piecewise constant.
Nevertheless, these differences are only minor issues once all necessary definitions and notations are properly introduced, which we do in Section~\ref{sec:main_and_outline}.

Let us also remark that $u_n^\chi$ and $u_n$ could be made identical by formally choosing $\chi_i=\mathbb{I}_{V_i}$.
However such choice would come with some additional issues.
The indicator functions are not cell functions according to Definition \ref{defi:partition_of_unity_cell} because they are not even continuous. 
One may still try to understand the gradients $\nabla \chi_i$ and $\bm{n}_{i,j}$ in distributional sense to have that $\nabla \chi_i = \sum_{j} \bm{n}_{i,j}$ and
\begin{equation*}
\begin{aligned}
\int_{\R^d} f(x) \cdot \bm{n}_{i,j}(x) \;\rd x = \int_{S_{i,j}} f(x) \cdot \bm{N}_{i,j} \;\rd x \quad \forall (i,j) \in \mathcal{E}, f \in C^\infty_c(\R^d, \R^d).
\end{aligned}
\end{equation*}
In such cases we formally have
\begin{equation*}
\begin{aligned}
a_{i,j} &\defeq \int_{S_{i,j}} \Big( b(x) \cdot \bm{N}_{i,j} \Big)^+ \;\rd x.
\end{aligned}
\end{equation*}
where the extra mollification in the current choice is removed.
But this extra mollification appears to be necessary for our formulation. 
For example integrating on $S_{i,j}$ without any mollification would require trace embedding and in turn more stringent conditions on the mesh, which we try to avoid.


\section{Main technical results of the paper} \label{sec:main_and_outline}

The goal of this section is to introduce the technical setting that we need for our approach and to state the main precise, quantitative results that underlies our compactness results.
First, we introduce some necessary notations. Then in subsection \ref{subsec:compactness_and_propagation}, we introduce the discrete kernel and semi-norm we use to prove compactness, which is modified from the continuous kernel and semi-norm introduced in \cite{BeJa:13, BeJa:19}.
We next state Theorem~\ref{thm:periodic_regularity_result} about the the propagation of regularity on periodic meshes.
This is the main quantitative result in the paper and, in particular, Theorem~\ref{thm:periodic_compactness_result} and Theorem~\ref{thm:Poisson_propagating} are deduced from it.
The proof of Theorem~\ref{thm:periodic_regularity_result} depends on multiple lemmas and theorems, which we state in subsections \ref{subsec:step_1}, \ref{subsec:step_2} and \ref{subsec:step_3}. But the actual proofs of these lemmas and theorems are postponed to later sections.

\subsection{Definitions and notations} \label{subsec:definitions_and_notations}

Consider a mesh $(\mathcal{C}, \mathcal{F})$ over $\Omega \subset \R^d$ as in Definition~\ref{defi:partition_of_unity_mesh}.
We introduce the following notations:
\begin{equation*}
\begin{aligned}
&\mathcal{V}_\Omega \defeq \{ i \in \mathcal{V}: \supp \chi_{i} \cap \Omega \neq \varnothing \},
\quad \mathcal{V}_\Omega^\circ \defeq \{ i \in \mathcal{V}: \supp \chi_{i} \subseteq \Omega \},
\\
&\overline{\mathcal{V}}_\Omega \defeq \left\{i \in \mathcal{V} : \exists j , \, j = i \text{ or } (i,j) \in \mathcal{E}, \text{ s.t. } \supp \chi_j \cap \Omega \neq \varnothing \right\},
\\
&\mathcal{E}_\Omega \defeq \{(i,j) \in \mathcal{E}:\supp \chi_{k} \cap \Omega \neq \varnothing \text{ for } k = i,j \}, \quad \mathcal{E}_\Omega^\circ \defeq \{(i,j) \in \mathcal{E}:\supp \bm{n}_{i,j} \subset \Omega\}.
\end{aligned}
\end{equation*}
For a function $f \in L^1_\text{loc}(\R^d)$ (or $L^1_\text{loc}(\R^d; \R^d)$), define the ``projection-to-cell operator'' $P_{\mathcal{C}}f$ as
\begin{equation} \label{eqn:proj_to_cell}
(P_{\mathcal{C}}f)_{i} \defeq \left\{
\begin{aligned}
&\frac{1}{\pi_i} \int_{\R^d} f(x) \chi_i(x) \;\rd x, \quad &&\forall i \in \mathcal{V}_{\Omega}^\circ,
\\
&0, \quad &&\forall i \in (\mathcal{V} \setminus \mathcal{V}_{\Omega}^\circ).
\end{aligned} \right.
\end{equation}
Moreover, for $f \in L^1_\text{loc}(\R^d; \R^d)$, define the ``projection-to-face operator'' $P_{\mathcal{F}}f$ as
\begin{equation} \label{eqn:proj_to_edge}
(P_{\mathcal{F}}f)_{i,j} \defeq \left\{
\begin{aligned}
&\int_{\R^d} \Big( f(x) \cdot \bm{n}_{i,j}(x) \Big)^+ \;\rd x, \quad && \forall (i,j) \in \mathcal{E}_\Omega^\circ,
\\
&0, \quad && \forall (i,j) \in (\mathcal{V}^2 \setminus \mathcal{E}_\Omega^\circ).
\end{aligned} \right.
\end{equation}
With these notations, the coefficients and initial data in \eqref{eqn:continuity_equation_scheme_coef} can be rewritten as $(a_{i,j})_{i,j \in \mathcal{V}} = P_{\mathcal{F}} b$ and $(u_i(0))_{i \in \mathcal{V}} = P_{\mathcal{C}} u_0$.
Next, we define the discrete divergence of $(a_{i,j})_{i,j \in \mathcal{V}}$ as
\begin{equation} \label{eqn:divergence_discrete}
D_k = D\big((a_{i,j})_{i,j \in \mathcal{V}}\big)_k \defeq
\left\{
\begin{aligned}
 &\frac{1}{\pi_k} \sum_{i \in \mathcal{V}} \big( a_{i,k} - a_{k,i} \big), \quad && \forall k \in \mathcal{V}_{\Omega}^\circ
 \\
 &0, \quad && \forall k \in  (\mathcal{V} \setminus \mathcal{V}_{\Omega}^\circ).
\end{aligned} \right.
\end{equation}
The definition of discrete divergence is justified by the following observation:
When choosing $(a_{i,j})_{i,j \in \mathcal{V}} = P_{\mathcal{F}} b$, one has
\begin{equation*}
\begin{aligned}
\forall k \in \mathcal{V}_{\Omega}^\circ, \quad D_k &= \frac{1}{\pi_k} \sum_{i \in \mathcal{V}} \left( \int_{\R^d} \Big( b(x) \cdot \bm{n}_{i,k}(x) \Big)^+ \;\rd x - \int_{\R^d} \Big( b(x) \cdot \bm{n}_{k,i}(x) \Big)^+ \;\rd x \right)
\\
&= \frac{1}{\pi_k} \sum_{i \in \mathcal{V}} \left( \int_{\R^d} \Big( b(x) \cdot \bm{n}_{i,k}(x) \Big)^+ \;\rd x - \int_{\R^d} \Big( b(x) \cdot \bm{n}_{i,k}(x) \Big)^- \;\rd x \right)
\\
&= \frac{1}{\pi_k} \sum_{i \in \mathcal{V}} \int_{\R^d} b(x) \cdot \bm{n}_{i,k}(x) \;\rd x \;\rd x
\\
&= - \frac{1}{\pi_k} \int_{\R^d} b(x) \nabla \chi_k(x) \;\rd x = \frac{1}{\pi_k} \int_{\R^d} \dive_x b(x) \chi_k(x) \;\rd x.
\end{aligned}
\end{equation*}
Hence, at least on $\mathcal{V}_{\Omega}^\circ$,
\begin{equation} \label{eqn:divergence_commutative}
\begin{aligned}
D(P_{\mathcal{F}} b) = P_{\mathcal{C}} \big( \dive_x b \big).
\end{aligned}
\end{equation}
For any $(v_i)_{i \in \mathcal{V}}$, we define its discrete $L^p$ norm by
\begin{equation*}
\begin{aligned}
\|(v_i)_{i \in \mathcal{V}}\|_{L^p(\mathcal{C})} &\defeq \left( \sum_{i \in \mathcal{V}} \big| v_{i} \big|^p \pi_i \right)^{1/p}.
\end{aligned}
\end{equation*}
In addition, we define the $L^p$ norm for the discretized velocity field $(a_{i,j})_{i,j \in \mathcal{V}}$ by
\begin{equation*}
\begin{aligned}
\|(a_{i,j})_{i,j \in \mathcal{V}}\|_{L^p(\mathcal{F})} = \left( \sum_{i,j \in \mathcal{V}} \big| a_{i,j} \big|^p  \right)^{1/p} (\delta x)^{d/p - (d-1)},
\end{aligned}
\end{equation*}
where the factor~$(\delta x)^{d/p - (d-1)}$ attempts to account for the expected size of the faces.

One motivation to define the discrete norms as above, and especially the scaling factor in $L^p(\mathcal{F})$, 
is that we can easily bound them by their continuous counterparts.
It is easy to verify the following proposition when $p=1$ and $p=\infty$, and the general case follows by an interpolation.
\begin{prop} \label{prop:norm_interpolate_discrete_density}
Let $(\mathcal{C},\mathcal{F})$ be a mesh satisfying \eqref{eqn:mesh_comparable_1}, then one has the following inequalities:
\begin{equation*}
\begin{aligned}
\|P_{\mathcal{F}} b\|_{L^p(\mathcal{F})} \leq C \|b\|_{L^p(\Omega)},
\quad
\|P_{\mathcal{C}} u\|_{L^{p}(\mathcal{C})} \leq C \|u\|_{L^{p}(\Omega)},
\end{aligned}
\end{equation*}
where the constant $C$ in the above inequalities only depends on the constants in the structural assumptions \eqref{eqn:mesh_comparable_1}.
\end{prop}

In this paper, we consider a sequence of discrete densities $u_{(n)} = (u_{i;(n)}(t))_{i \in \mathcal{V}^{(n)}}$ defined on a sequence of meshes $(\mathcal{C}^{(n)}, \mathcal{F}^{(n)})$. Since the $n$-th density is always defined on the $n$-th mesh, as an abuse of notation, we write $\|u_{(n)}\|_{L^p(\mathcal{C}^{(n)})}$ as $\|u_{(n)}\|_{L^p(\mathcal{C})}$ for simplicity. 
Similarly, for the discrete coefficients $a_{(n)} = (a_{i,j;(n)})_{i,j \in \mathcal{V}^{(n)}}$ on the meshes we write $\|a_{(n)}\|_{L^p(\mathcal{F}^{(n)})}$ as $\|a_{(n)}\|_{L^p(\mathcal{F})}$.

The following notations are also useful in later discussions: 
For each $i \in \mathcal{V}$, define the ``barycenter'' of cell function $\chi_i$ by 
\begin{equation*}
\begin{aligned}
x_i \defeq \frac{1}{\pi_i} \int x\, \chi_i(x) \;\rd x.
\end{aligned}
\end{equation*}
For any $(v_i)_{i \in \mathcal{V}}$, define its extension to $\R^d$ by
\begin{equation*}
\begin{aligned}
v^\chi \defeq \sum_{i \in \mathcal{V}} v_i\, \chi_{i}.
\end{aligned}
\end{equation*}

\subsection{Compactness via quantitative regularity estimates} \label{subsec:compactness_and_propagation}
In this subsection we introduce the explicit semi-norms that we are going to use in the paper, together with lemmas and propositions about some basic properties of those objects.
The proof of all lemmas and propositions are postponed to Section~\ref{sec:kernel_equiv}.

The following continuous kernels and semi-norms are introduced in \cite{BeJa:13, BeJa:19} to prove the compactness of density:
\begin{defi} \label{defi:semi-norm_continuous}
Define the kernel $K^h$ for all $h > 0$ by
\begin{equation*}
\begin{aligned}
K^h(x) \defeq \frac{\phi(x)}{(|x| + h)^d}, \quad \forall x \in \R^d,
\end{aligned}
\end{equation*}
where $\phi$ is some smooth function with compact support in $B(0,2)$ and s.t. $\phi = 1$ inside $B(0,1)$. 
Then for $1 \leq p < \infty$, $0 < \theta < 1$, the semi-norm $\|\cdot\|_{p,\theta}$ for density $u \in L^p(\R)$ is defined as
\begin{equation} \label{eqn:log_Sobolev_continuous}
\begin{aligned}
\|u\|_{p,\theta}^p \defeq \sup_{h \leq 1/2} |\log h|^{-\theta} \int_{\R^{2d}} K^h(x-y) |u(x) - u(y)|^p \;\rd x \rd y.
\end{aligned}
\end{equation}
\end{defi}
We define the corresponding discretization of such kernels and semi-norms.
\begin{defi} \label{defi:semi-norm_discrete}
Consider a mesh $(\mathcal{C}, \mathcal{F})$ over $\Omega \subset \R^d$, on which there exists a discrete density $(u_i)_{i \in \mathcal{V}}$ such that $\mbox{supp}\,u_i \subset \mathcal{V}_{\Omega}^\circ$. Assume that
\begin{equation} \label{eqn:mesh_comparable_2}
\begin{aligned}
\sum_{i \in \mathcal{V}} \chi_i(x) = 1,\quad \forall x \in \Omega + B(0,4).
\end{aligned}
\end{equation}
Given any so-called {\em virtual coordinates} $\widetilde{x} = (\widetilde{x}_i)_{i \in \mathcal{V}} \in (\R^d)^{\mathcal{V}}$, we define an approximate kernel $\widetilde{K}_{i,j}^h$ on the mesh by
\begin{equation*} 
\begin{aligned}
\widetilde{K}_{i,j}^h = K^h(\widetilde{x}_i - \widetilde{x}_j), \quad \forall i,j \in \mathcal{V}.
\end{aligned}
\end{equation*}
Then for $1 \leq p < \infty$, $0 < \theta < 1$ and $0 < h_0 < 1/2$,
the discrete semi-norm $\|\cdot\|_{h_0,p,\theta;\widetilde{x}}$ on the mesh is defined as
\begin{equation} \label{eqn:log_Sobolev_discrete}
\begin{aligned}
\|u\|_{h_0,p,\theta;\widetilde{x}}^p \defeq \sup_{h_0 \leq h \leq 1/2} |\log h|^{-\theta} \sum_{i,j \in \mathcal{V}} \widetilde{K}_{i,j}^h |u_i - u_j|^p \pi_i \pi_j.
\end{aligned}
\end{equation}
\end{defi}

The following lemma is the cornerstone of deriving compactness from the discrete regularity in Definition \ref{defi:semi-norm_discrete}, in the particular case where $\widetilde{x} \in (\R^d)^{\mathcal{V}}$ in Definition \ref{defi:semi-norm_discrete} is simply chosen by the barycenters, i.e. $\widetilde{x}_i = x_i$.
In that case, we use the specific notation
\begin{equation*}
\begin{aligned}
K_{i,j}^h = K^h(x_i - x_j) \quad \text{ and } \quad \|\cdot\|_{h_0,p,\theta} = \|\cdot\|_{h_0,p,\theta;(x_i)_{i \in \mathcal{V}}}
\end{aligned}
\end{equation*}
to specify the discrete semi-norm derived from the barycenters $(x_i)_{i \in \mathcal{V}}$.
\begin{lem} \label{lem:compactness_regularity}
Let $(\mathcal{C}, \mathcal{F})$
be a mesh as in Definition~\ref{defi:partition_of_unity_mesh} over $\Omega \subset \R^d$ such that \eqref{eqn:mesh_comparable_1} and \eqref{eqn:mesh_comparable_2} hold. Consider a discrete function $(u_{i})_{i \in \mathcal{V}}$ on the mesh, satisfying the bound
\begin{equation*}
\begin{aligned}
\|(u_{i})_{i \in \mathcal{V}}\|_{h_0,p,\theta} \leq L
\end{aligned}
\end{equation*}
for some $0 < h_0 < 1/2$.
Define the renormalized kernel $\bar K^h(x) \defeq K^h(x) / \|K^h\|_{L^1}$ and let $u^\chi$ be the extension of $(u_{i}(t))_{i \in \mathcal{V}}$ to $\R^d$. Then
\begin{equation*}
\begin{aligned}
\forall h > h_0, \quad \|u^\chi - \bar K^h \star u^\chi\|_{L^p}^p \leq C |\log h|^{\theta - 1}
\end{aligned}
\end{equation*}
where the constant $C$ only depends on $L$ and the constant in structural assumptions \eqref{eqn:mesh_comparable_1}.
If the mesh is given by a polygon mesh $(\{V_i\}, \{S_{i,j}\})$ via
\eqref{eqn:partition_of_unity_polygon}, then the above inequality also holds when $u^\chi$ is replaced by the piecewise constant extension.

\end{lem}

Consider a sequence of meshes $(\mathcal{C}^{(n)}, \mathcal{F}^{(n)})$ and a sequence of discrete density $u_{(n)} = (u_{i;(n)})_{i \in \mathcal{V}^{(n)}}$ defined on them. It is just natural to study the compactness of such discrete densities on different meshes by some sort of extension on $\R^d$.

To see how Lemma~\ref{lem:compactness_regularity} helps to derive compactness of such sequence, assume that one has uniform boundedness
\begin{equation*}
\begin{aligned}
\sup_{0 < h < 1/2} \limsup_{n \to \infty}
 \|u_{(n)}\|_{h_0,1,\theta} < \infty,
\end{aligned}
\end{equation*}
and uniform boundedness of discrete $L^p(\mathcal{C})$ norm.
Then for any fixed $h > 0$, the difference between extended functions $u_{(n)}^\chi$ and their mollifications are uniformly bounded by $C |\log h|^{\theta - 1}$ up to discarding finitely many terms of the sequence.
On the other hand for any fixed $h$ greater than $0$, the sequence of mollified functions is locally compact.
Therefore, the sequence of extended functions is also locally compact.

However there are several big issues that one should be aware of when moving from the continuous to the discrete setting: 
\begin{itemize}
\item
The kernel parameter $h$ has to be bounded from below in Definition \ref{defi:semi-norm_discrete}, because a kernel too sharp is not suitable for a coarser grid.
Generically $h_0$ should be chosen much greater than the discretization size $\delta x$.
But for a sequence of meshes with $\delta x$ converging to zero, $h_0$ could be chosen converging to zero as well (with a possibly much slower speed).
As we just discussed below Lemma~\ref{lem:compactness_regularity},
this sort of regularity in asymptotic sense would be sufficient to continue our discussion of compactness.

\item
Moreover, in Definition \ref{defi:semi-norm_continuous} the integral is taking on $\R^{2d}$, while in Definition \ref{defi:semi-norm_discrete} we are restricted to a finite double summation.
Nevertheless, any kernel $K^h$ has bounded support in ball $B(0;2)$, hence for any density $u$ with bounded support, the double integral in \eqref{eqn:log_Sobolev_continuous} can be taken on $\big(\supp u + B(0;2)\big)^2$ instead of $\R^{2d}$.
Therefore, to reasonably approximate the integral in \eqref{eqn:log_Sobolev_continuous}, it is natural to made the additional assumption \eqref{eqn:mesh_comparable_2} for the summation in \eqref{eqn:log_Sobolev_discrete}. The larger ball $B(0;4)$ is used for the convenience of later analysis.
Starting from a mesh $(\mathcal{C}, \mathcal{F})$ over $\Omega \subset \R^d$ on which the upwind scheme \eqref{eqn:continuity_equation_scheme} is defined, one can always put additional cell functions to make \eqref{eqn:mesh_comparable_2} hold.
The scheme is not really affected as the density is set as zero at any $i \notin \mathcal{V}_{\Omega}^\circ$.
For this reason, when discussing quantitative regularity, we always add \eqref{eqn:mesh_comparable_2} as part of our assumption to meshes.

\item
The more delicate issue and the one that leads to most technical difficulties in this paper is how to choose the virtual coordinates $(\widetilde{x}_i)_{i \in \mathcal{V}}$. 
While it would seem natural to take $(\widetilde{x}_i)_{i \in \mathcal{V}} = (x_i)_{i \in \mathcal{V}}$, the corresponding semi-norm does not seem to be propagated well on the scheme \eqref{eqn:continuity_equation_scheme}. This will force the use of $(\widetilde{x}_i)_{i \in \mathcal{V}} \neq (x_i)_{i \in \mathcal{V}}$ to obtain semi-norms that we can propagate well.
On the other hand, by Lemma~\ref{lem:compactness_regularity} we can clearly see compactness from the semi-norms induced by $(x_i)_{i \in \mathcal{V}}$, but not from semi-norms induced by arbitrary $(\widetilde{x}_i)_{i \in \mathcal{V}}$.
Therefore, we will also have to show that the approximate kernels $K^h(\widetilde{x}_i - \widetilde{x}_j)$ are equivalent to $K^h(x_i - x_j)$, for a choice of virtual coordinates $(\widetilde{x}_i)_{i \in \mathcal{V}}$ that are only slightly different from the barycenters $(x_i)_{i \in \mathcal{V}}$.

\end{itemize}

We can make the last issue somewhat more precise by a more general estimate that consider the discrete kernels as some sort of perturbation of the continuous kernels.
\begin{lem} \label{lem:kernel_equiv_Euclidean}
Consider measurable functions $f_i,g_i : \R^d \to \R^d$, $i = 1,2$ and $0 < h_1 < 1/4$, such that $|x - f_i(x)| \leq h_1$, $|x - g_i(x)| \leq h_1$, $\forall x \in \R^d, i = 1,2$. Consider the kernels
\begin{equation*}
\begin{aligned}
K_f^h(x,y) = K^h(f_1(x),f_2(y)) = \frac{\phi(|f_1(x)- f_2(y)|)}{(|f_1(x)- f_2(y)| + h)^d}, 
\\
K_g^h(x,y) = K^h(g_1(x),g_2(y)) = \frac{\phi(|g_1(x)- g_2(y)|)}{(|g_1(x)- g_2(y)| + h)^d}.
\end{aligned}
\end{equation*}
Then for $1 \leq p < \infty$ and $0 < h < 1/2$,
\begin{equation} \label{eqn:key_estimate_perturb}
\begin{aligned}
\int_{\R^{2d}} K_g^h(x,y) |u(x) - v(y)|^p \;\rd x\rd y \leq \left(1 +  Ch_1/h \right) \int_{\R^{2d}} K_f^h(x,y) |u(x) - v(y)|^p \;\rd x \rd y,
\end{aligned}
\end{equation}
where the constant $C$ only depends on the fixed choice of $\phi$ in the definition of kernels $K^h$.
\end{lem}

Notice that the double summation in \eqref{eqn:log_Sobolev_discrete} can be rewritten as a double integral form by carefully choosing some function $f=f_1 = f_2$ and a piecewise constant, which we state as the next lemma:
\begin{lem} \label{lem:kernel_equiv_extension_trick}
Consider a mesh $(\mathcal{C}, \mathcal{F})$ as in Definition~\ref{defi:partition_of_unity_mesh} over $\Omega \subset \R^d$ with discretization size $\delta x < 1/16$, such that \eqref{eqn:mesh_comparable_2} hold.
Introduce some measurable sets $(V_i)_{i \in \mathcal{V}} \subset \R^d$ such that
\begin{equation*}
\begin{aligned}
|V_i| = \pi_i = \int_{\R^d} \chi_i, \quad \sup_{x \in V_i} |x - x_i| < 2\delta x, \quad \forall i \in \mathcal{V}, \quad \text{ and } \quad V_{i} \cap V_{j} = \varnothing, \quad \forall i,j \in \mathcal{V}.
\end{aligned}
\end{equation*}
Define the  piecewise constant extension $u^V \defeq \sum_{i \in \mathcal{V}} u_i \mathbbm{1}_{V_i}$ for the discrete density function. Then
\begin{equation*}
\begin{aligned}
\supp u^V \subset \Big( \Omega + B(0,1) \Big) \subset \Big( \Omega + B(0,3) \Big) \subset \Big( \bigcup_{i \in \mathcal{V}} V_i \Big).
\end{aligned}
\end{equation*}
Assume that the virtual coordinates $(\widetilde{x}_i)_{i \in \mathcal{V}}$ are such that $\left|\widetilde{x}_i - x_i\right| < h_2$, $\forall i \in \mathcal{V}$ and for some $0< h_2 < 1/16$.
Define $f: \R^d \to \R^d$ as
\begin{equation*}
f(x) = \left\{
\begin{aligned}
&\widetilde{x}_i, \quad &&\text{ for } x \in V_i,\; i \in \mathcal{V},
\\
&x, \quad &&\text{ for } x \notin \textstyle \bigcup_{i \in \mathcal{V}} V_i.
\end{aligned} \right.
\end{equation*}
Then $|x - f(x)| \leq  2\delta x + h_2 < 1/4$, $\forall x \in \R^d$.
Moreover the double summation in \eqref{eqn:log_Sobolev_discrete} can be rewritten as a double integral:
\begin{equation*}
\begin{aligned}
\sum_{i,j \in \mathcal{V}} \widetilde{K}_{i,j}^h |u_i - u_j|^p \pi_i \pi_j
=
\int_{\R^{2d}} K^h\big(f(x) - f(y)\big) |u^V(x) - u^V(y)|^p \;\rd x\rd y,
\end{aligned}
\end{equation*}
for all $1 \leq p < \infty$ and $0 < h < 1/2$.

\end{lem}

From these two lemmas, one may deduce the following proposition supporting our use of $(\widetilde{x}_i)_{i \in \mathcal{V}} \neq (x_i)_{i \in \mathcal{V}}$.
\begin{prop} \label{prop:kernel_equiv}
Consider a mesh $(\mathcal{C}, \mathcal{F})$ with discretization size $\delta x < h_2 < 1/16$, such that \eqref{eqn:mesh_comparable_1} and \eqref{eqn:mesh_comparable_2} hold.
Let $(\widetilde{x}_i^{(1)})_{i \in \mathcal{V}}$, $(\widetilde{x}_i^{(2)})_{i \in \mathcal{V}} \in (\R^d)^{\mathcal{V}}$ be two sets of virtual coordinates on the mesh such that
\begin{equation*}
\begin{aligned}
\forall i \in \mathcal{V}, \; k = 1,2, \quad\quad \left|\widetilde{x}_i^{(k)} - x_i\right| < h_2.
\end{aligned}
\end{equation*}
Then for $1 \leq p < \infty$, $0 < \theta < 1$ and $0 < h_0 < 1/2$, the two resulting semi-norms
and equivalent and satisfy
\begin{equation*}
\begin{aligned}
\|u\|_{h_0,p,\theta;\widetilde{x}^{(2)}} \leq \left(1 + Ch_2/h_0 \right)\|u\|_{h_0,p,\theta;\widetilde{x}^{(1)}},
\end{aligned}
\end{equation*}
where the constant $C$ is fixed.
\end{prop}

The proposition implies that the equivalence of semi-norms can be derived from the closeness of virtual coordinates.
Therefore, a large part of our technical analysis is actually devoted to finding appropriate $(\widetilde{x}_i)_{i \in \mathcal{V}}$ ensuring the propagation of regularity while remaining reasonably close to barycenters $(x_i)_{i \in \mathcal{V}}$.

The next proposition is also a consequence of Lemma~\ref{lem:kernel_equiv_Euclidean}:

\begin{prop} \label{prop:discrete_regularity_extended}
Consider a mesh $(\mathcal{C}, \mathcal{F})$ such that \eqref{eqn:mesh_comparable_1} and \eqref{eqn:mesh_comparable_2} hold. Then for $1 \leq p < \infty$,
\begin{equation*}
\begin{aligned}
\int_{\R^{2d}} K^h(x,y) |u^\chi(x) - u^\chi(y)|^p \;\rd x\rd y 
\leq \left(1 + C\delta x/h \right) \sum_{i,j \in \mathcal{V}} K^h(x_i - x_j) |u_{i} - u_{j}|^p \;\pi_{i} \pi_{j},
\end{aligned}
\end{equation*}
and
\begin{equation*}
\begin{aligned}
\sum_{i,j \in \mathcal{V}} K_{i,j}^h \Big| (P_{\mathcal{C}}u)_{i} - (P_{\mathcal{C}}u)_{j} \Big|^p \;\pi_i \pi_j &=  \sum_{i,j \in \mathcal{V}} K_{i,j}^h \left|\frac{1}{\pi_{i}} \int_{\R^d} u(x) \chi_{i}(x) \;\rd x - \frac{1}{\pi_{j}} \int_{\R^d} u(y) \chi_{j}(y) \;\rd y\right|^p \;\pi_{i} \pi_{j}
\\
&\leq \left(1 + C\delta x/h \right) \int_{\R^{2d}} K^h(x,y) |u(x) - u(y)|^p \;\rd x\rd y.
\end{aligned}
\end{equation*}
for some constant $C$ depending only on $p$ and the constants in the structural assumptions in~\eqref{eqn:mesh_comparable_1}.
\end{prop}
This proposition ensures that the regularity of the extended function is comparable to the regularity of the discrete density and vice versa, which is needed in our proof of Lemma~\ref{lem:compactness_regularity} and Theorem~\ref{thm:Poisson_propagating}.

\subsection{Our main quantitative regularity result} \label{sec:quantitative_regularity}

We are now ready to state our main quantitative theorem about the propagation of regularity on periodic mesh.
\begin{thm} \label{thm:periodic_regularity_result}
Consider $T > 0$, and a bounded domain $\Omega \subset \R^d$ with piecewise smooth boundary.
Let $\{(\mathcal{C}^{(n)}, \mathcal{F}^{(n)})\}_{n=1}^\infty$ be a sequence of meshes over $\Omega \subset \R^d$ as in Definition~\ref{defi:partition_of_unity_mesh},
having discretization size $\delta x_{(n)} \to 0$, satisfying the structural assumptions \eqref{eqn:mesh_comparable_1} and \eqref{eqn:mesh_comparable_2} by some uniform constant, and being periodic on $\Omega$ with pattern size uniformly bounded.

For all $n \geq \N_+$, $t \in [0,T]$, let $(a_{i,j;(n)}(t))_{i,j \in \mathcal{V}^{(n)}}$ be the coefficients of the upwind scheme \eqref{eqn:continuity_equation_scheme} on $(\mathcal{C}^{(n)}, \mathcal{F}^{(n)})$ and let
$D_{(n)}(t) = (D_{i;(n)}(t))_{i \in \mathcal{V}^{(n)}}$ be the discrete divergence defined as in \eqref{eqn:divergence_discrete}.
Let $u_{(n)} = (u_{i;(n)}(t))_{i \in \mathcal{V}^{(n)}}$ be a sequence of discrete density solved by the upwind scheme.
With some $1 \leq p < q \leq \infty$ and $0 < s \leq 1$, assume that there exists a sequence of velocity field $\widetilde{b}_{(n)}(t,x)$, bounded uniformly in
$L^{q}_t(W^{1,q}_x) \cap L^{p}_x(W^{s,p}_t)([0,T] \times \Omega)$, and approximating the coefficients $(a_{i,j;(n)}(t))_{i,j \in \mathcal{V}^{(n)}}$ with vanishing error
\begin{equation*}
\begin{aligned}
&\big\| (a_{i,j;(n)})_{i,j \in \mathcal{V}^{(n)}} - P_{\mathcal{F}^{(n)}} \widetilde{b}_{(n)}\, \big\|_{L^{p}([0,T] \times \mathcal{F}^{(n)})} \to 0 \quad \text{as } n \to \infty.
\end{aligned}
\end{equation*}
Assume moreover that the solutions have uniformly bounded norms $\sup_n \|u_{(n)}\|_{L^\infty_t L^{p^*}_x([0,T] \times \mathcal{C}^{(n)})}<\infty$, and that  mass leaking vanishes
\begin{equation*}
\begin{aligned}
\|u_{(n)}(0)\|_{L^1(\mathcal{C})} - \|u_{(n)}(T)\|_{L^1(\mathcal{C})} \to 0 \quad \text{as } n \to \infty.
\end{aligned}
\end{equation*}

Then for all $\theta \geq \max\{1 - 1/q, 1/2\}$, $0 < h_0 < 1/2$, there exists sufficiently large $N \in \N^+$ such that for all $n \geq N$,
\begin{equation} \label{eqn:main_regularity_result}
\begin{aligned}
&\quad 
 \|u_{(n)}(t)\|_{h_0,1,\theta} \leq {\alpha}^* \bigg[ \|u_{(n)}(0)\|_{h_0,1,\theta}
 \\
 &\quad \quad \quad \quad + C\int_0^t 
 \bigg( \|\dive \widetilde{b}_{(n)}(s)\|_{L^\infty(\mathcal{C})} \|u_{(n)}(s)\|_{h_0,1,\theta} + \|\widetilde{b}_{(n)}(s)\|_{W^{1,q}} \|u_{(n)}(s)\|_{L^{p^*}(\mathcal{C})}
\\ 
&\quad \quad \quad \quad + \|D_{(n)}(s)\|_{L^\infty(\mathcal{C})} \|u_{(n)}(s)\|_{h_0,1,\theta} + \|u_{(n)}(s)\|_{L^{p^*}(\mathcal{C})} \|D_{(n)}(s)\|_{h_0,p,p(\theta - 1/p^*)} \bigg) \;\rd s
\\
&\quad \quad \quad \quad + L_2^* + L_3^*
\bigg],
\end{aligned}
\end{equation}
with additional terms due to discretization
\begin{equation} \label{eqn:main_regularity_appendix}
\begin{aligned}
&L_2^* = C \big( |\log h_0|^{-\theta} / h_0^2 \big) (\delta x_{(n)}) \int_0^t \Big( \|a_{(n)}(s)\|_{L^q(\mathcal{F})} + \|\widetilde{b}_{(n)}(s)\|_{W^{1,q}} \Big) \|u_{(n)}(s)\|_{L^{q^*}(\mathcal{C})} \;\rd s
\\
& \quad\quad+ C ( |\log h_0|^{1-\theta} ) \big( \|u_{(n)}(0)\|_{L^1(\mathcal{C})} - \|u_{(n)}(t)\|_{L^1(\mathcal{C})} \big),
\\
&L_3^* = C 
(|\log h_0|^{-\theta}/h_0)
\Bigg[
\big\| a_{(n)} - P_{\mathcal{F}^{(n)}} \widetilde{b}_{(n)}\, \big\|_{L^{p}([0,T] \times \mathcal{F})}
+(\delta x_{(n)})^{s/(1+s)}  \|\widetilde{b}_{(n)}\|_{L^{p}_x(W^{s,p}_t)}
\\
& \quad\quad\quad\quad
+ (\delta x_{(n)})^{\frac{1/p-1/q}{1+(1/p-1/q)}} \bigg( \|a_{(n)}\|_{L^{q}([0,T] \times \mathcal{F})}
+ \|\widetilde{b}_{(n)}\|_{L^{q}_t(W^{1,q}_x)}
\bigg)
\Bigg] \|u_{(n)}\|_{L^\infty_t L^{p^*}_x([0,t] \times \mathcal{C})}
,
\\
&{\alpha}^* = \exp\big( C(1/h_0) (\delta x_{(n)})^{s/(1+s)} \big).
\end{aligned}
\end{equation}
The constant $C$ in \eqref{eqn:main_regularity_result} only depends on $\Omega$, the exponents $p,q$ and the constant in the structural assumption \eqref{eqn:mesh_comparable_1}, while the constant $C$ in \eqref{eqn:main_regularity_appendix} also depends on $T$, the exponent $s$ and the constant bounding pattern size. Nevertheless, none of the constants depends on $h_0$ or $\delta x_{(n)}$.
The index $N \in \N^+$ is chosen to make $\delta x_{(n)}$ sufficiently small, which only depends on $h_0$ and the constant bounding pattern size.

In particular, for any fixed $h_0 > 0$, the additional terms $L_2^*, L_3^*$ converge to zero and ${\alpha}^*$ converges to one as $n \to \infty$.
\end{thm}

Proving Theorem~\ref{thm:periodic_regularity_result} is the main technical challenge of the paper.
We split our proof into three theorems, namely Theorem~\ref{thm:propagation_regularity_discrete},~\ref{thm:residue_term_estimate}~and~\ref{thm:auxiliary_function_existence}. These three theorems are stated in subsection \ref{subsec:step_1}, \ref{subsec:step_2} and \ref{subsec:step_3} and we conclude Section~\ref{sec:outline_of_proof} by how they are used to prove Theorem~\ref{thm:periodic_regularity_result}.
Each of the three theorems requires its own proof on which we spend an entire section after Section~\ref{sec:outline_of_proof}.


Before we move to the proof, 
we conclude this section by showing how to deduce Theorem~\ref{thm:periodic_compactness_result} and Theorem~\ref{thm:Poisson_propagating} from Theorem~\ref{thm:periodic_regularity_result}.

\begin{proof}[Proof of Theorem~\ref{thm:periodic_compactness_result}]

Let us begin with the discussion of mesh properties as
Theorem~\ref{thm:periodic_compactness_result} is stated in the setting of polygon meshes. We first recall the construction \eqref{eqn:partition_of_unity_polygon} in Section~\ref{subsec:connection_two_settings}, restated here
\begin{equation*} 
\begin{aligned}
\chi_i(x) &= \frac{1}{|B_r(0)|} \int_{B_r(0)} \mathbbm{1}_{V_i} (x - y) \;\rd y, \quad \forall i \in \mathcal{V},
\\
\bm{n}_{i,j}(x) &= \int_{S_{i,j}} \frac{1}{|B_r(0)|} \mathbbm{1}_{B_r(0)}(x-y) \bm{N}_{i,j} \;\rd y, \quad \forall (i,j) \in \mathcal{E}.
\end{aligned}
\end{equation*}
for the entire sequence $\{(\mathcal{C}^{(n)}, \mathcal{F}^{(n)})\}_{n=1}^\infty$. As an abuse of notation, we still use $\{(\mathcal{C}^{(n)}, \mathcal{F}^{(n)})\}_{n=1}^\infty$ to denote the generated sequence $\{\big( \{\chi_{i;(n)}\}_{i \in \mathcal{V}^{(n)}}, \{\bm{n}_{i,j;(n)}\}_{(i,j) \in \mathcal{E}^{(n)}} \big)\}_{n=1}^\infty$.
It is easy to verify that if the polygon meshes satisfy the structural assumptions \eqref{eqn:mesh_comparable_polygon}, then the constructed new meshes as in Definition~\ref{defi:partition_of_unity_mesh} satisfy the structural assumptions \eqref{eqn:mesh_comparable_1}, with a possibly larger constant.
Moreover, as explained in Section~\ref{subsec:compactness_and_propagation}, one can always put additional cell functions to make \eqref{eqn:mesh_comparable_2} hold. This yields a sequence of meshes that fulfills the requirements of Theorem~\ref{thm:periodic_regularity_result}.

Now we define a linear operator $P_{\mathcal{F}}'$ as an alternate of the ``projection-to-face operator'' $P_{\mathcal{F}}$, such that in Theorem~\ref{thm:periodic_compactness_result}, the coefficients of upwind scheme on each $(\mathcal{C}^{(n)}, \mathcal{F}^{(n)})$ is chosen exactly by $(a_{i,j;(n)})_{i,j \in \mathcal{V}^{(n)}} = P_{\mathcal{F}^{(n)}}' b$.
Such operator $P_{\mathcal{F}}'$ is given by
\begin{equation*}
(P_{\mathcal{F}}' f)_{i,j} \defeq \left\{
\begin{aligned}
&\left(\int_{\R^d} f(x) \cdot \bm{n}_{i,j}(x) \;\rd x\right)^+, \quad && \forall (i,j) \in \mathcal{E}_\Omega^\circ,
\\
&0, \quad && \forall (i,j) \in (\mathcal{V}^2 \setminus \mathcal{E}_\Omega^\circ).
\end{aligned} \right.
\end{equation*}
It is easy to verify the divergence identity $D(P_{\mathcal{F}}' b) = P_{\mathcal{C}} \big( \dive_x b \big)$ by the same approach with which we obtain~\eqref{eqn:divergence_commutative}. Also, it is straightforward that $\|P_{\mathcal{F}}' b\|_{L^p(\mathcal{F})} \leq \|P_{\mathcal{F}} b\|_{L^p(\mathcal{F})} \leq C \|b\|_{L^p(\Omega)}$.

We can then do some a priori estimates for the norms required by Theorem~\ref{thm:periodic_regularity_result}. To avoid writing too many index $(n)$ in the calculation,
let $(\mathcal{C}, \mathcal{F}) = \big( \{\chi_{i}\}_{i \in \mathcal{V}}, \{\bm{n}_{i,j}\}_{(i,j) \in \mathcal{E}} \big)$ be any mesh in the sequence of meshes we consider.
Firstly, notice that for $i \in \mathcal{V}_{\Omega}^\circ$, one has
\begin{equation} \label{eqn:density_sup_norm_1}
\begin{aligned}
\frac{\rd u_{i}}{\rd t} &= \frac{1}{\pi_i} \sum_{j \in \mathcal{V}} \big( a_{i,j} \, u_{j} - a_{j,i} \, u_i \big) \leq \frac{1}{\pi_i} \left( \sup_{k \in \mathcal{V}} u_{k} \sum_{j \in \mathcal{V}}  a_{i,j} - u_i \sum_{j \in \mathcal{V}} a_{j,i} \right)
\\
&= -D_i \, \sup_{k \in \mathcal{V}} u_{k} + \frac{1}{\pi_i}\sum_{j \in \mathcal{V}} a_{j,i} \Big(\sup_{k \in \mathcal{V}} u_k - u_i \Big),
\end{aligned}
\end{equation}
and for all $i \in  \mathcal{V}_{\Omega} \setminus \mathcal{V}_{\Omega}^\circ$, one has $u_i \equiv 0$.
By the assumption $\dive_x b \in L^{\infty}_t L^{\infty}_x$, one can conclude $\|D(t)\|_{L^\infty(\mathcal{C})} \leq C \|\dive_x b\|_{L^{\infty}}$ and
\begin{equation} \label{eqn:density_sup_norm_2}
\begin{aligned}
\frac{\rd}{\rd t} \sup_{k \in \mathcal{V}} u_k &\leq \sup_{i \in \mathcal{V}} (-D_j) \sup_{k \in \mathcal{V}} u_{k} \leq C\|\dive_x b\|_{L^{\infty}} \sup_{k \in \mathcal{V}} u_k.
\end{aligned}
\end{equation}
The constants $C$ just above do not depend on $(n)$, so that $(u_{i;(n)})_{i \in \mathcal{V}^{(n)}}$ and  $(D_{i;(n)})_{i \in \mathcal{V}^{(n)}}$ have uniform a priori bound in $L^\infty_t L^{\infty}_x([0,T] \times \mathcal{C}^{(n)})$.
By H\"older estimate one can obtain uniform bound in any $L^\infty_t L^{p}_x([0,T] \times \mathcal{C}^{(n)})$ where $1 \leq p \leq \infty$.

Secondly, for any $s > 0$ and $\theta \geq 1 - 1/p$, the semi-norm of the divergence is bounded by
\begin{equation*}
\begin{aligned}
\|(D_{i}(t))_{i \in \mathcal{V}}\|_{h_0,p,p(\theta - 1/p^*)}
\leq
C \| \dive_x b(t) \|_{p,p(\theta - 1/p^*)}
\leq C \| \dive_x b(t) \|_{W^{s,p}}.
\end{aligned}
\end{equation*}
The first inequality is an application of the divergence identity $D(P_{\mathcal{F}}' b) = P_{\mathcal{C}} \big( \dive_x b \big)$ and Proposition~\ref{prop:kernel_equiv}, while the second inequality is due to Sobolev estimates. Similarly,
\begin{equation*}
\begin{aligned}
\|(u_{i}(0))_{i \in \mathcal{V}}\|_{h_0,1,\theta}
\leq C \| u_0 \|_{W^{s,1}}.
\end{aligned}
\end{equation*}
The constants $C$ again do not depend on $(n)$. This gives uniform bounds to the semi-norm of $(D_{i;(n)}(t))_{i \in \mathcal{V}^{(n)}}$ and $(u_{i;(n)}(0))_{i \in \mathcal{V}^{(n)}}$.

We want to apply Theorem~\ref{thm:periodic_regularity_result} with $\widetilde{b}_{(n)}(t,x) = b(t,x)$, $p = 1$ and $q = q$ (recall that $b \in L^{q}_t(W^{1,q}_x) \cap L^{1}_x(W^{s,1}_t)$).
But the issue remains is that our newly defined $P_{\mathcal{F}}'$ is not identical to $P_{\mathcal{F}}$. Hence, the $L^1$ difference
\begin{equation*}
\begin{aligned}
\big\| (a_{i,j;(n)})_{i,j \in \mathcal{V}^{(n)}} - P_{\mathcal{F}^{(n)}} \widetilde{b}_{(n)} \big\|_{L^1([0,T] \times \mathcal{F}^{(n)})} = \big\| P_{\mathcal{F}^{(n)}}' b - P_{\mathcal{F}^{(n)}} b \big\|_{L^1([0,T] \times \mathcal{F}^{(n)})}
\end{aligned}
\end{equation*}
is not zero, and one has to argue it converges to zero as $n \to \infty$  and as $\delta x_{(n)} \to 0$.
This is guaranteed by the following proposition whose proof is postponed to Section \ref{sec:kernel_equiv}.
\begin{prop} \label{prop:two_ways_projection}
Let $(\mathcal{C}, \mathcal{F})$
be a mesh as in Definition~\ref{defi:partition_of_unity_mesh} over $\Omega \subset \R^d$ such that \eqref{eqn:mesh_comparable_1} hold.
Assume that each face function $\bm{n}_{i,j} \in \mathcal{F}$ is of form $\bm{n}_{i,j}(x) = \bm{N}_{i,j} w_{i,j}(x), \forall x \in \R^d$, where $\bm{N}_{i,j}\in \mathcal{S}^{d-1}$ is a unit vector and $w_{i,j}$ is a scalar function.
Then for $1 \leq p \leq \infty$,
\begin{equation*}
\begin{aligned}
\big\| P_{\mathcal{F}}' b - P_{\mathcal{F}} b \big\|_{L^{p}([0,T] \times \mathcal{F})} \leq C \delta x \|b\|_{L^{p}_t(W^{1,p}_x)},
\end{aligned}
\end{equation*}
where the constant $C$ only depends on $p$ and the constant in the structural assumption~\eqref{eqn:mesh_comparable_1}.

\end{prop}
We only have to observe that the construction \eqref{eqn:partition_of_unity_polygon} indeed ensures that any $\bm{n}_{i,j} \in \mathcal{F}$ is of form $\bm{n}_{i,j}(x) = \bm{N}_{i,j} w_{i,j}(x)$. Hence this proposition applies to the setting of Theorem~\ref{thm:periodic_compactness_result}. And we can finally apply Theorem~\ref{thm:periodic_regularity_result} to obtain \eqref{eqn:main_regularity_result} with $\widetilde{b}_{(n)}(t,x) = b(t,x)$, $p = 1$ and $q = q$.

By Gronwall estimate one can conclude
\begin{equation*}
\begin{aligned}
C_{\log, \theta}^1 \defeq \sup_{0 < h_0 < 1/2} \limsup_{n \to \infty} \|u_{(n)}(t)\|_{h_0,1,\theta}  < \infty
\end{aligned}
\end{equation*}
for some $0 < \theta <1$. This directly implies compactness in space of the density $u_{(n)}$. Compactness in time now follows by reproducing the Aubin-Lions argument in the semi-discrete setting.

For any $0 < h < 1/2$, let $h_0 = h$. By the previous estimates, one can choose $N(h) \in \N_+$ such that
\begin{equation*}
\begin{aligned}
\sup_{n \geq N(h)} \|u_{(n)}(t)\|_{h,1,\theta}  < 2 C_{\log, \theta}^1.
\end{aligned}
\end{equation*}
By Lemma~\ref{lem:compactness_regularity}, one has
\begin{equation} \label{Khun}
\begin{aligned}
\|\bar K^h \star_x u_{(n)}(t) - u_{(n)}(t)\|_{L^1} \leq C |\log h|^{\theta - 1}  \quad \text{ for } n \geq N(h), t \in [0,T],
\end{aligned}
\end{equation}
where $C$ depends on $C_{\log, \theta}^1$ and the total mass of $u_{(n)}$.

On the other hand for any fixed $h > 0$, $U_{h,C,\Omega} \defeq \{\bar K^h \star u : \|u\|_{L^\infty} < C, \supp u \subset \Omega \}$ is a compact set by Arzel\`a–Ascoli theorem, on which we consider $t \mapsto \bar K_h \star_x u_{(n)}(t)$, the trajectory of  mollified density. Notice that
\begin{equation} \notag
\begin{aligned}
V_h(t) &\defeq \int_{\R^d}\left |\frac{\rd}{\rd t} \Big( \bar K^h \star_x u_{(n)}(t,x)\Big) \right| \;\rd x
= \int_{\R^d}\left | \sum_{i \in \mathcal{V}^{(n)}} \left(\int_{C_{i;(n)}} \bar K^h(y-x) \rd y\right) \frac{\rd}{\rd t} u_{i;(n)}(t) \right| \;\rd x.
\end{aligned}
\end{equation}
By \eqref{eqn:continuity_equation_scheme_0} we can bound $V_h(t)$ by
\begin{equation} \notag
\begin{aligned}
V_h(t) 
&= \int_{\R^d}\left |\sum_{i \in \mathcal{V}^{(n)}} \left(\int_{C_{i;(n)}} \!\!\!\!\!\!\!\! K^h(y-x) \rd y\right)
\frac{1}{\pi_i} \sum_{i' :\, (i,i') \in \mathcal{E}^{(n)}} \!\!\!\!\!\!\!\!\big( a_{i,i';(n)} \, u_{i';(n)} - a_{i',i;(n)} \, u_{i;(n)} \big)
\right| \;\rd x
\\
&= \int_{\R^d}\Bigg |\sum_{(i,i') \in \mathcal{E}^{(n)}} 
\left(\frac{1}{\pi_i} \int_{C_{i;(n)}} \!\!\!\!\!\!\!\! K^h(y-x) \rd y - \frac{1}{\pi_{i'}} \int_{C_{i';(n)}} \!\!\!\!\!\!\!\! K^h(z-x) \rd z\right)
\big( a_{i,i';(n)} \, u_{i';(n)} \big)
\Bigg| \;\rd x
\\
&\leq \sum_{(i,i') \in \mathcal{E}^{(n)}} 
\left( \int_{\R^d} \left|\frac{1}{\pi_i} \int_{C_{i;(n)}} \!\!\!\!\!\!\!\! K^h(y-x) \rd y - \frac{1}{\pi_{i'}} \int_{C_{i';(n)}} \!\!\!\!\!\!\!\! K^h(z-x) \rd z\right| \rd x \right)
\big( a_{i,i';(n)} \, u_{i';(n)} \big)
\\
&\leq \sum_{(i,i') \in \mathcal{E}^{(n)}} 
L_h \delta x_{(n)}
\big( a_{i,i';(n)} \, u_{i';(n)} \big)
\\
&\leq C L_h \|b(t)\|_{L^1} \|u_{(n)}(t)\|_{L^{\infty}(\mathcal{C}^{(n)})},
\end{aligned}
\end{equation}
where $L_h$ denotes the Lipschitz constant of $\bar K^h$ and $C$ depends on the constant in Proposition~\ref{prop:norm_interpolate_discrete_density}.
Since we assume $b \in L^\infty_t L^1_x$ and have obtained uniform a priori bound of $(u_{i;(n)})_{i \in \mathcal{V}^{(n)}}$ in $L^\infty_t L^{\infty}_x([0,T] \times \mathcal{C}^{(n)})$, we have $V_h(t)$ uniformly bounded in $t$.

Therefore, for any fixed $h \geq 0$, $K^h \star u_{(n)}$ is equicontinuous as a trajectory in $U_{h,C,\Omega} \subset L^1(\R^d)$.
By Arzel\`a–Ascoli theorem, $\{\bar K^h \star u_{(n)} \}$ is compact for all $h > 0$, which implies $\{u_{(n)}\}$ is also compact thanks to~\eqref{Khun}.

\end{proof}

\bigskip

We now turn to the proof of Theorem~\ref{thm:Poisson_propagating} which requires more work but follows somewhat similar steps.

\begin{proof}[Proof of Theorem~\ref{thm:Poisson_propagating}]
As in the proof of Theorem~\ref{thm:periodic_compactness_result}, our goal is to apply Theorem~\ref{thm:periodic_regularity_result}.
As a comparison, this time we begin with a sequence of meshes as in Definition~\ref{defi:partition_of_unity_mesh} so the mesh properties are obvious, but since we are discussing a coupled system, we also need to actually derive regularity estimates on the velocity field.

As before, when there is no ambiguity, we omit the index $(n)$ by letting $(\mathcal{C}, \mathcal{F}), (\mathcal{P}, \mathcal{N})$ be any pair of mesh and finite element in the sequence, and let $(u_i)_{i \in \mathcal{V}}$, $(a_{i,j})_{i,j \in \mathcal{V}}$ be the discrete solution.

\emph{Step 1: Discrete a priori bounds.} The discrete divergence at $i \in \mathcal{V}_{\Omega}^\circ$ is given by
\begin{equation*}
\begin{aligned}
D_{i} &= \frac{1}{\pi_{i}} \sum_{i' \in \mathcal{V}} \big( a_{i',i} - a_{i,i'} \big)
= \frac{1}{\pi_{i}} \sum_{i' \in \mathcal{V}} \int_{\R^d} \big( b(y) \cdot \bm{n}_{i',i}(y) \big)^+ - \big( b(y) \cdot \bm{n}_{i',i}(y) \big)^- \;\rd y
\\
&= \frac{1}{\pi_{i}} \sum_{i' \in \mathcal{V}} \int_{\R^d} \nabla \phi(y) \cdot \bm{n}_{i',i}(y) \;\rd y
= \frac{1}{\pi_{i}} \int_{\R^d} \nabla \phi(y) \cdot (- \nabla \chi_{i}(y)) \;\rd y
\\
&= \frac{1}{\pi_{i}} \int_{\R^d} \big(-g^\chi(y)\big) \chi_{i}(y) \;\rd y.
\end{aligned}
\end{equation*}
The last identity is due to the assumption that $\chi_i \in \mathcal{C} \subset \mathcal{P}$ and \eqref{eqn:Poisson_equation_scheme}.

By Proposition~\ref{prop:discrete_regularity_extended}, for all $h \geq \delta x$, one has
\begin{equation*}
\begin{aligned}
\sum_{i,j \in \mathcal{V}} K_{i,j}^h |D_{i} - D_{j}| \;\pi_{i} \pi_{j}
&= \sum_{i,j \in \mathcal{V}} K_{i,j}^h \left|\frac{1}{\pi_{i}} \int_{\R^d} g^\chi(x) \chi_{i}(x) \;\rd x - \frac{1}{\pi_{j}} \int_{\R^d} g^\chi(y) \chi_{j}(y) \;\rd y\right| \;\pi_{i} \pi_{j}
\\
&\leq C \int_{\R^{2d}} K^h(x,y) |g^\chi(x) - g^\chi(y)| \;\rd x\rd y
\\
&\leq C \sum_{i,j \in \mathcal{V}} K_{i,j}^h |g(u_{i}) - g(u_{j})| \;\pi_{i} \pi_{j}
\\
&\leq CL_g \sum_{i,j \in \mathcal{V}} K_{i,j}^h |u_{i} - u_{j}| \;\pi_{i} \pi_{j}.
\end{aligned}
\end{equation*}
where $L_g$ be the Lipschitz constant of $g$.
Hence,
\begin{equation}
\begin{aligned}
\big\| \big(D_{i}(t)\big)_{i \in \mathcal{V}} \big\|_{h_0,1,\theta} \leq C L_g \big\| \big(u_{i}(t)\big)_{i \in \mathcal{V}} \big\|_{h_0,1,\theta}
\end{aligned}\label{divfromu}
\end{equation}
Also, from the above discussion it is straightforward to see that
\begin{equation*}
\begin{aligned}
D_{i}(t) = \frac{1}{\pi_{i}} \int_{\R^d} \big(-g^\chi(y)\big) \chi_{i}(y) \;\rd y &= \frac{1}{\pi_{i}} \int_{\R^d} \bigg(-\sum_{j \in \mathcal{V}} g\big(u_{j}(t)\big) \chi_{j}(y)\bigg) \chi_{i}(y) \;\rd y
\\
&= - \sum_{j \in \mathcal{V}} A_{i,j} \, g\big( u_{j}(t) \big),
\end{aligned}
\end{equation*}
where the coefficients satisfies
\begin{equation*}
\begin{aligned}
A_{i,j} \geq 0, 
\quad 
\sum_{j \in \mathcal{V}} A_{i,j}  = 1.
\end{aligned}
\end{equation*}
Thus
\begin{equation*}
\begin{aligned}
\sup_{i \in \mathcal{V}}\big(D_{i}(t)\big)_{i \in \mathcal{V}} &\leq \sup \Big\{-g(a)\,|\; a\in [\inf_{i \in \mathcal{V}} \big(u_{i}(t)\big)_{i \in \mathcal{V}},\ \sup_{i \in \mathcal{V}} \big(u_{i}(t)\big)_{i \in \mathcal{V}}] \Big\},
\\
\inf_{i \in \mathcal{V}}\big(D_{i}(t)\big)_{i \in \mathcal{V}} &\geq \inf \Big\{-g(a)\,| \;a\in [\inf_{i \in \mathcal{V}} \big(u_{i}(t)\big)_{i \in \mathcal{V}},\ \sup_{i \in \mathcal{V}} \big(u_{i}(t)\big)_{i \in \mathcal{V}}] \Big\}.
\end{aligned}
\end{equation*}
Since we assumed that $g\in L^\infty(\R)$, this means that the divergence is bounded uniformly in $n$.
By \eqref{eqn:density_sup_norm_1} and \eqref{eqn:density_sup_norm_2}, we also obtain a uniform in $n$ bound of $(u_i)_{i \in \mathcal{V}}$ in $L^\infty([0,T] \times \mathcal{C})$ for any $T > 0$.

Recalling moreover that $\widetilde{\phi}$ is the solution of $- \Delta \widetilde{\phi} = g^\chi$ (with Dirichlet BC on $\Omega_{e}$), one also has a uniform bound on $\nabla \widetilde{\phi}$ in $L^\infty_t L^\infty_x \cap L^\infty_t H^1_x$.

\medskip

\emph{Step 2: Control of mass leaking by Markovian interpretation.} 
The discrete scheme can also be represented by a Poisson random process model. 
Without loss of generality, we may assume that the total mass $\sum_{i \in \mathcal{V}} u_{i}(0) \pi_{i} = 1$
and define the initial condition of the random process $X(t)$ by $\mathbb{P}\{\bm{X}(0) = i\} = u_{i}(0) \pi_{i}$. We choose the rate of the Poisson process as
\begin{equation*}
\begin{aligned}
\lambda_{i',i}(t) = \frac{a_{i',i}^+(t)}{\pi_{i}}.
\end{aligned}
\end{equation*}
Define now the stopping time and number of jumps through
\begin{equation*}
\begin{aligned}
\bm{\tau} &= \inf \{t : \bm{X}(t) \notin \mathcal{V}_{\Omega_{v}}^\circ \},
\\
\bm{N}(t) &= \sup \{N: 0 = s_0 < s_1 < \dots < s_N \leq t, \; \forall 1\leq i \leq N, \bm{X}(s_i) \neq \bm{X}(s_{i-1})\}.
\end{aligned}
\end{equation*}
Then we have a straightforward identity of the density for $t > 0$  by
\begin{equation*}
\begin{aligned}
u_{i}(t) = \frac{1}{\pi_{i}} \mathbb{P}\{\bm{\tau} > t, \; \bm{X}(t) = i\}.
\end{aligned}
\end{equation*}
By the fact that $\nabla \widetilde{\phi}$ is bounded in $L^\infty$ and \eqref{eqn:coupling_finite_element_bound}, one can conclude that $b = \nabla \phi$ is also uniformly bounded in $L^\infty$. Denote $M_b$ the a priori bound of $\|b\|_{L^\infty}$, then
\begin{equation*}
\begin{aligned}
\lambda_{i',i}(t) = \frac{a_{i',i}^+(t)}{\pi_{i}} \leq \frac{M_\lambda}{\delta x},
\end{aligned}
\end{equation*}
where $M_\lambda = C M_b$.

Define $L_{\partial\Omega_v}=\mbox{dist}\,(\mbox{supp}\,u_0,\,\partial\Omega_v)$. By its definition, $\mbox{dist}\,(\bm{X}_{0},\,\partial\Omega_v)\geq L_{\partial\Omega_v}$, so that it requires at least $L_{\partial\Omega_v}/ \delta x$ jumps to reach the boundary, i.e.
\begin{equation*}
\begin{aligned}
\bm{\tau} \leq t \quad \text{ only if } \quad \bm{N}(t) \geq \left\lfloor \frac{L_{\partial}}{\delta x} \right\rfloor.
\end{aligned}
\end{equation*}
In particular, one can bound the probability of $\bm{T} \leq t$ by a homogeneous Poisson process, i.e.
\begin{equation*}
\begin{aligned}
\mathbb{P}\{\bm{\tau} \leq t\} \leq \bar P\left(\frac{M_\lambda}{\delta x}, \left\lfloor \frac{L_{\partial\Omega_v}}{\delta x} \right\rfloor - 1\right),
\end{aligned}
\end{equation*}
where $\bar P(\lambda,.)$ denotes the probability distribution of a homogeneous Poisson process with rate $\lambda$ and starting from $0$.

Consider $T > 0$ such that $L_{\partial\Omega_v} - M_\lambda T > 0$, and let $\Lambda_{T} = (L_{\partial\Omega_v} - M_\lambda T)/2$.
Then  one can deduce
\begin{equation*}
\begin{aligned}
\mathbb{P}\left\{\bm{\tau} \leq T \right\} \leq C \exp(- \Lambda_{T} /\delta x),
\end{aligned}
\end{equation*}
by standard estimates for homogeneous Poisson processes.

Hence choosing $T$ s.t. $\Lambda_T>0$, we obtain the mass leaking estimate, for any $t\leq T$
\[
0\leq 1-\sum_{i\in\mathcal{V}}  \pi_i\,u_i(t)\leq C \exp(- \Lambda_{T} /\delta x)\longrightarrow 0,\quad \mbox{as}\ n\to \infty.
\]

\medskip

\emph{Step 3: Regularity of the continuous velocity field.}
We choose the continuous velocity field in Theorem~\ref{thm:periodic_regularity_result} as  $\widetilde{b} = \nabla \widetilde{\phi}$, $- \Delta \widetilde{\phi} = g^\chi$ (with Dirichlet BC on $\Omega_{e}$).
By our discussion in Step 1, 
we have uniform in $n$ bounds on $\widetilde{b} = \nabla \widetilde{\phi}$ in $L^\infty_t L^\infty_x \cap L^\infty_t H^1_x$.

Moreover, by our assumption \eqref{eqn:coupling_finite_element_bound} on the finite elements, one can bound the $L^2$ difference between $(a_{i,j})_{i,j \in \mathcal{V}} = P_{\mathcal{F}}b$ and $P_{\mathcal{F}}\widetilde{b}$ by
\begin{equation*}
\begin{aligned}
\big\| (a_{i,j})_{i,j \in \mathcal{V}} - P_{\mathcal{F}} \widetilde{b} \, \big\|_{L^\infty_t L^{2}_x([0,T] \times \mathcal{F})} &= \big\| P_{\mathcal{F}} b - P_{\mathcal{F}} \widetilde{b} \big\|_{L^\infty_t L^{2}_x([0,T] \times \mathcal{F})}
\\
&\leq C \|\nabla \phi - \nabla \widetilde{\phi}\|_{L^\infty_t L^2_x} \leq C\delta x \|\widetilde{\phi}\|_{L^\infty_t H^2_x},
\end{aligned}
\end{equation*}
which converges to zero as $n\to\infty$ and $\delta x \to 0$.

We can also show that our choice of $\widetilde{b}$ has Sobolev regularity in time, namely
$\widetilde{b} = \nabla \widetilde{\phi} \in W^{s,1}_t L^{1}_x([0,T] \times \Omega_{e})$ for any $s<1$. Notice that
\begin{equation*}
\begin{aligned}
- \Delta \widetilde{\phi} = g^\chi = \sum_{i \in \mathcal{V}} \chi_{i} g(u_{i}),
\end{aligned}
\end{equation*}
so that the issue is to show that $\sum_{i \in \mathcal{V}} \chi_{i} g(u_{i}) \in W^{1,1}_t W^{-1,1}_x(\Omega_{e})$.

Any solution $u$ of the first-order scheme \eqref{eqn:continuity_equation_scheme} satisfies the following identity for all $i$,
\begin{equation*}
\begin{aligned}
\frac{\rd}{\rd t} u_i = 
\frac{1}{\pi_i} \sum_{j \in \mathcal{V}} \Big( a_{i,j} u_{j} - a_{j,i} u_{i} \Big)
- \mathbbm{1}_{\mathcal{V}\setminus \mathcal{V}_{\Omega}^\circ}(i) \frac{1}{\pi_i} \sum_{j \in \mathcal{V}} a_{i,j} u_j.
\end{aligned}
\end{equation*}
When $i \in \mathcal{V}_{\Omega}^\circ$, the above equality is exactly the upwind scheme. When  $i \in (\mathcal{V}\setminus \mathcal{V}_{\Omega}^\circ)$, one has $u_i \equiv 0$ and the above equality reduce to $0 = 0$.
Therefore,
\begin{equation*}
\begin{aligned}
\frac{\rd}{\rd t}g(u_i) = g'(u_i) \frac{\rd u_i}{\rd t}
&= \left( g'(u_i) \frac{1}{\pi_i} \sum_{j \in \mathcal{V}} \big(a_{i,j} u_j - a_{j,i} u_i \big) \right)
- \left( \mathbbm{1}_{\mathcal{V}\setminus \mathcal{V}_{\Omega_{v}}^\circ}(i) g'(u_i) \frac{1}{\pi_i} \sum_{j \in \mathcal{V}} a_{i,j} u_j \right)
\\
&\eqdef G^{\textrm{int}}_i - G^{\textrm{bd}}_i.
\end{aligned}
\end{equation*}
The term $G^{\textrm{bd}}_i$ measures the possible leaking at boundary. It is non-negative and from the previous step, it satisfies
\begin{equation*}
\begin{aligned}
\int_{\R^d} \sum_{i} G^{\textrm{bd}}_i(t) \, \chi_i(x) \;\rd x\rd t
=
\sum_{i \in \mathcal{V}} G^{\textrm{bd}}_i(t) \pi_i \;\rd t \leq C L_g \exp(- \Lambda_T /\delta x),
\end{aligned}
\end{equation*}
which implies that $\sum_{i \in \mathcal{V}} G^{\textrm{bd}}_i \chi_{i} \in L^\infty_t L^1_x$.

In addition, the term $G^{\textrm{int}}_i$ can be reformulated as
\begin{equation*}
\begin{aligned}
G^{\textrm{int}}_i &= g'(u_i) \frac{1}{\pi_i} \sum_{j \in \mathcal{V}} \big(a_{i,j} u_j - a_{j,i} u_i \big)
\\
&= \bigg( \frac{1}{\pi_i} \sum_{j \in \mathcal{V}} \big(a_{i,j} g(u_j) - a_{j,i} g(u_i) \big) \bigg) + \bigg( \frac{1}{\pi_i} \sum_{j \in \mathcal{V}} \big( a_{i,j} - a_{j,i} \big) \big( g'(u_i) u_i - g(u_i) \big) \bigg)
\\
&\quad + \bigg( \frac{1}{\pi_i} \sum_{j \in \mathcal{V}} a_{i,j} \big( g'(u_i)(u_j - u_i) - [g(u_j) - g(u_i)] \big) \bigg)
\\
&\eqdef G^A_i + G^B_i + G^M_i.
\end{aligned}
\end{equation*}
Since $|G^B_i| \leq 2 \sup_i |D_i| L_g M \leq 2(L_gM)^2$, it is straightforward that $\sum_{i \in \mathcal{V}} G^B_i \chi_{i} \in L^\infty_t L^\infty_x$.
Also, by the concavity of nonlinearity $g$, one has $G^M_i \geq 0$.
Furthermore,
\begin{equation*}
\begin{aligned}
&\quad \int_0^T \int_{\Omega_e} \sum_{i} G^M_i(t) \, \chi_i(x) \;\rd x\rd t
\\
&=
\int_0^T \sum_{i \in \mathcal{V}} G^M_i(t) \pi_i \;\rd t
\\
&= \int_0^T \sum_{i \in \mathcal{V}} \Big( \frac{\rd}{\rd t}g(u_i(t)) - G^A_i(t) - G^B_i(t) + G^{\textrm{bd}}_i(t) \Big) \pi_i \;\rd t
\\
&\leq \left| \sum_{i \in \mathcal{V}} \Big( g(u_i(T)) - g(u_i(0)) \Big) \right| + 0 + \left|\int_0^T \sum_{i \in \mathcal{V}} G^B_i(t) \pi_i \;\rd t \right| + \left| \int_0^T \sum_{i \in \mathcal{V}} G^{\textrm{bd}}_i(t) \pi_i \;\rd t \right|
\\
&\leq 2L_gM + 2T(L_gM)^2 + C T L_g \exp(- \Lambda_T /\delta x),
\end{aligned}
\end{equation*}
where in the first inequality we use the observation that $\sum_{i \in \mathcal{V}} G^A_i(t) \pi_i \equiv 0$.
As a consequence, one has $\sum_{i \in \mathcal{V}} G^M_i \chi_{i} \in L^1_t L^1_x$.

Finally, for any test function $\varphi \in W^{1,\infty}$,
\begin{equation*}
\begin{aligned}
&\quad \int_{\Omega_e} \varphi(x) \sum_{i \in \mathcal{V}} G^A_i(t) \chi_i(x) \;\rd x
\\
&= \int_{\Omega_e} \varphi(x) \sum_{i,j \in \mathcal{V}} \frac{1}{\pi_i} \big(a_{i,j} g(u_j(t)) - a_{j,i} g(u_i(t)) \big) \chi_i(x) \;\rd x
\\
&= \sum_{i,j \in \mathcal{V}} \big(a_{i,j} g(u_j(t)) - a_{j,i} g(u_i(t)) \big) \frac{1}{\pi_i} \int_{\Omega_e} \varphi(x) \chi_i(x) \;\rd x
\\
&= \sum_{i,j \in \mathcal{V}} \big(a_{i,j} g(u_j(t)) \big) \left( \frac{1}{\pi_i} \int_{\Omega_e} \varphi(x) \chi_i(x) \;\rd x - \frac{1}{\pi_j} \int_{\Omega_e} \varphi(y) \chi_j(y) \;\rd y \right)
\\
&= \sum_{(i,j) \in \mathcal{E}} \big(a_{i,j} g(u_j(t)) \big) \left( \frac{1}{\pi_i}\frac{1}{\pi_j} \int_{\Omega_e}\int_{\Omega_e} \big[\varphi(x) - \varphi(y)\big] \chi_i(x) \chi_j(y) \;\rd x \rd y \right).
\end{aligned}
\end{equation*}
By our structural assumptions, the number of terms in the last sum is at most $C (\delta x)^{-d}$ and each term is bounded by $C (\delta x)^d L_g M^2 \|\varphi\|_{W^{1,\infty}}$.
Hence
\begin{equation*}
\begin{aligned}
\int_{\Omega_e} \varphi(x) \sum_{i \in \mathcal{V}} G^A_i(t) \chi_i(x) \;\rd x
\leq
C L_g M^2 \|\varphi\|_{W^{1,\infty}},
\end{aligned}
\end{equation*}
which means $\sum_{i \in \mathcal{V}} G^A_i \chi_{i} \in L^\infty_t W^{-1,1}_x$.

By combining all estimates above, we conclude that
\begin{equation*}
\begin{aligned}
\frac{\rd}{\rd t} g^\chi = \frac{\rd}{\rd t} \sum_{i \in \mathcal{V}} \chi_{i} g(u_{i}) \in L_t^1 W^{-1,1}_x.
\end{aligned}
\end{equation*}
It is also straightforward that $g^\chi \in L^\infty_t L^\infty_x$. Thus one indeed has that $g^\chi \in W^{1,1}_t W^{-1,1}_x$, which implies that $\widetilde{b} = \nabla \widetilde{\phi} \in W^{s,1}_t L^{1}_x$ for any $s<1$.

\emph{Step 4.} (Compactness)
Combine the previous results and apply them to $(\mathcal{C}^{(n)}, \mathcal{F}^{(n)})$ and $(\mathcal{P}^{(n)}, \mathcal{N}^{(n)})$ for all $n \in \N_+$. Since all functions are defined on a bounded domain, Sobolev embeddings also directly apply. 
We may then use Theorem~\ref{thm:periodic_regularity_result} with $p = 1$, $q = 2$ and and $s< 1$, yielding the following asymptotic estimate for $\theta > 1/2$,
\begin{equation*}
\begin{aligned}
&\limsup_{n \to \infty}
 \|u_{(n)}(t)\|_{h_0,1,\theta}
\leq \limsup_{n \to \infty}
 \bigg[ \|u_{(n)}(0)\|_{h_0,1,\theta}
 \\
 &\quad\quad\quad\quad\quad\quad\quad\quad + C\int_0^t 
 \bigg( \|u_{(n)}(s)\|_{L^\infty(\mathcal{C})} \|u_{(n)}(s)\|_{h_0,1,\theta} + \|\widetilde{b}_{(n)}(s)\|_{W^{1,1}} \|u_{(n)}(s)\|_{L^\infty(\mathcal{C})} \bigg) \;\rd s \bigg],
\end{aligned}
\end{equation*}
where we used step 1 to bound the discrete divergence terms in Theorem~\ref{thm:periodic_regularity_result} by the corresponding bound on $u_{(n)}$.

By Gronwall estimate we conclude that
\begin{equation*}
\begin{aligned}
C_{\log, \theta}^1 \defeq \sup_{0 < h_0 < 1/2} \limsup_{n \to \infty} \|u_{(n)}(t)\|_{h_0,1,\theta}  < \infty.
\end{aligned}
\end{equation*}
The last part is to argue that $t \mapsto \bar K_h \star_x u_{(n)}^\chi(t)$, the trajectory of  mollified extended density, is equicontinuous on a compact subset of space $L^1(\R^d)$, which is performed in the same manner as in the proof of Theorem~\ref{thm:periodic_compactness_result}.

\end{proof}

\section{Proving Theorem~\ref{thm:periodic_regularity_result}}
\label{sec:outline_of_proof}

In this section we start the proof of Theorem~\ref{thm:periodic_regularity_result}, which is spread into Section~\ref{sec:outline_of_proof},~\ref{sec:propagation_regularity_discrete},~\ref{sec:residue_term_estimate}~and~\ref{sec:residue_term_estimate}. We introduce in this section three theorems that each corresponds to a specific step and  explain why they together prove Theorem~\ref{thm:periodic_regularity_result}.

The first step in the proof is naturally an estimate of the time evolution of $\|u(t)\|_{h_0,1,\theta}$, which we state as Theorem~\ref{thm:propagation_regularity_discrete}.
Most terms in the estimate behave as one can expect from the continuous model.
However, there is one additional term involving what we call a residue $r_i(t)$ which given by linear equation \eqref{eqn:linear_system_time_dependent} and restated here,
\begin{equation*}
\begin{aligned}
 \sum_{i' \in \partial \{i\}} (\widetilde{x}_{i'}^{(k)} - \widetilde{x}_i^{(k)}) a_{i',i}^+(t) = \widetilde{b}_i(t) \pi_i + r_i(t) \pi_i, \quad \forall i \in \mathcal{V}, \; t\in [t_{k-1}, t_{k}], \; 1\leq k \leq m.
\end{aligned}
\end{equation*}
We need to control this residue to conclude the bound on $\|u(t)\|_{h_0,1,\theta}$ through  Gronwall lemma. The study of the residue is where our proof fully deviates from the continuous setting.

In essence the size of the residue follows from the choice of virtual coordinates $\tilde x_i$. We correspondingly introduce two theorems: The first one  identifies some good assumptions for the virtual coordinates to make the residue small, which we state as Theorem~\ref{thm:residue_term_estimate}.
The second theorem~\ref{thm:auxiliary_function_existence} shows that virtual coordinates satisfying such assumptions actually exist, at least where the mesh has periodic patterns.
%
%
\subsection{Step 1: Propagation of regularity in the discrete setting} \label{subsec:step_1}

Our first result reproduces the propagation of regularity in \cite{BeJa:19}
for scheme \eqref{eqn:continuity_equation_scheme} but with additional terms caused by the discretization. The proof of the theorem is postponed to Section~\ref{sec:propagation_regularity_discrete}.

\begin{thm} \label{thm:propagation_regularity_discrete}

Consider the semi-discrete scheme \eqref{eqn:continuity_equation_scheme}
on a mesh $(\mathcal{C}, \mathcal{F})$ over a bounded domain $\Omega \subset \R^d$ with piecewise smooth boundary as in Definition~\ref{defi:partition_of_unity_mesh}, having discretization size $\delta x$ and satisfying the structural assumptions \eqref{eqn:mesh_comparable_1} and \eqref{eqn:mesh_comparable_2}. 
Let $(a_{i,j}(t))_{i,j \in \mathcal{V}}$
be the coefficients of scheme \eqref{eqn:continuity_equation_scheme}
and $D(t) = (D_i(t))_{i \in \mathcal{V}}$ be the discrete divergence given by \eqref{eqn:divergence_discrete}.
Let $\widetilde{b}(t,x)$ be a continuous velocity field on $\R^d$ and denote its discretization by $(\widetilde{b}_i(t))_{i \in \mathcal{V}} = P_{\mathcal{F}}\widetilde{b}(t,\cdot)$.

Choose $M_\beta, M_\gamma > 0$, such that $\delta x \leq M_\gamma \leq M_\beta < 1/32$.
Divide the time interval $[0,T]$ as $0 = t_0 < t_1 < \dots < t_m = T$. 
For each interval $[t_{k-1},t_{k}]$, let $(\widetilde{x}_i^{(k)})_{i \in \mathcal{V}}$ be virtual coordinates on the mesh satisfying
\begin{subequations} \label{eqn:drift}
\begin{align} \label{eqn:relative_drift}
|\widetilde{x}_i^{(k)} - \widetilde{x}_{i'}^{(k)}| &\leq 2 M_\gamma, \quad \forall (i,i') \in \mathcal{E},
\\ \label{eqn:absolute_drift}
|\widetilde{x}_i^{(k)} - x_i| &\leq 2 M_\beta, \quad \forall i \in \mathcal{V}.
\end{align}
\end{subequations}
Let $(r_i(t))_{i \in \mathcal{V}}, t \in [0,T]$ be the residue function given by
\begin{equation} \label{eqn:linear_system_time_dependent}
\begin{aligned}
 \sum_{i' \in \partial \{i\}} (\widetilde{x}_{i'}^{(k)} - \widetilde{x}_i^{(k)}) a_{i',i}(t) = \widetilde{b}_i(t) \pi_i + r_i(t) \pi_i, \quad \forall i \in \mathcal{V}, \; t\in [t_{k-1}, t_{k}], \; 1\leq k \leq m.
\end{aligned}
\end{equation}
Let $k(t) = \min \{k: t < t_k\}, \forall t \in [0,T]$.
Then any solution $u(t) = (u_i(t))_{i \in \mathcal{V}}$, $t \in [0,T]$ of the semi-discrete scheme \eqref{eqn:continuity_equation_scheme},
satisfies for  $0<h_0<1/2$
\begin{equation} \label{eqn:propagation_Gronwall}
\begin{aligned}
  &\|u(t)\|_{h_0,1,\theta} \leq {\alpha} (L_0 + L_1 + L_2 + L_3),
\end{aligned}
\end{equation}
where
\begin{equation} \label{eqn:propagation_Gronwall2}
\begin{aligned}
&L_0 = \|u(0)\|_{h_0,1,\theta}, \quad {\alpha} = \big(1 + C (M_\beta / h_0)\big)^{k(t)+1}, 
\\
&L_1 = C\int_0^t 
 \bigg( \|\dive \widetilde{b}(s)\|_{L^\infty(\mathcal{C})} \|u(s)\|_{h_0,1,\theta}  + \|\widetilde{b}(s)\|_{W^{1,q}} \|u(s)\|_{L^{p^*}(\mathcal{C})}
\\ 
&\quad \quad \quad \quad + \|D(s)\|_{L^\infty(\mathcal{C})} \|u(s)\|_{h_0,1,\theta} + \|u(s)\|_{L^{p^*}(\mathcal{C})} \|D(s)\|_{h_0,p,p(\theta - 1/p^*)} \bigg) \;\rd s
\\
&L_2 =  C \int_0^t \bigg(\; \big( |\log h_0|^{-\theta} M_\gamma^2 / h_0^2 \delta x \big) \|(a_{i,j}(s))_{i,j \in \mathcal{V}}\|_{L^q(\mathcal{F})} \|u(s)\|_{L^{q^*}(\mathcal{C})}
\\ 
&\quad\quad\quad\quad\quad + \big(|\log h_0|^{-\theta} M_\beta / h_0^2 \big) \|\widetilde{b}(s)\|_{L^q} \|u(s)\|_{L^{q^*}(\mathcal{C})}
\\
&\quad\quad\quad\quad\quad + (|\log h_0|^{-\theta} \delta x / h_0^2) \|\widetilde{b}(s)\|_{W^{1,q}} \|u(s)\|_{L^{q^*}(\mathcal{C})} \bigg) \;\rd s
 \\
&\quad\quad +C ( |\log h_0|^{1-\theta} ) \Big( \|u(0)\|_{L^1(\mathcal{C})} - \|u(t)\|_{L^1(\mathcal{C})} \Big),
\\ 
&L_3 = C \int_0^t  (|\log h_0|^{-\theta}/h_0) \|(r_i(s))_{i \in \mathcal{V}}\|_{L^p(\mathcal{C})} \|u(s)\|_{L^{p^*}(\mathcal{C})} \;\rd s,
\end{aligned}
\end{equation}
provided that $1 \leq p < q \leq \infty$, $\theta \geq \max\{1 - 1/q, 1/2\}$, $M_\beta \leq h_0 < 1/2$.
The constant $C$ depends on $\Omega$, the exponents $p,q$ and the constant in structural assumptions $\eqref{eqn:mesh_comparable_1}$.

\end{thm}

What we have exhibited in this subsection is a rather incomplete result. The terms $\alpha,\, L_2$ and $L_3$ in Theorem~\ref{thm:propagation_regularity_discrete}
are due to the discretization error.
The term $L_2$  would tend to zero as the discretization size $\delta x \to 0$, provided that the virtual coordinates are chosen such that $\max\{ M_\gamma / (\delta x)^{1/2}, M_{\beta}^{1/2} \} \to 0$ and $h_0$ are chosen decaying to zero with an even slower speed.

Unfortunately, $L_3$ will not vanish so easily.
To eliminate it asymptotically, 
one should find a sophisticated way to determine suitable division $0 = t_0 < \dots < t_m = T$ and virtual coordinates $(\widetilde{x}_i^{(k)})_{i \in \mathcal{V}}$ on each $[t_{k-1}, t_{k}]$, making the norm of residue $\|r\|_{L^1_t L^p_x([0,T] \times \mathcal{C})}$ decay. To control the term $\alpha$, one should also control the number $m$ in the division of $[0,T]$.

Finally we emphasize that we do not need to assume that the coefficients $a_{i,j}(t)$ be given from the discretization of the velocity field~$\tilde b$. The connection between $a_{i,j}$ and $\tilde b$ stems only from the definition of the residue in~\eqref{eqn:linear_system_time_dependent}. And in fact, as we will see in the conclusion of the proof, the $a_{i,j}$ are typically derived from a slightly different field $b\neq \tilde b$. 

\subsection{Step 2: Controlling the residue through virtual coordinates} \label{subsec:step_2}

It remains to investigate in what circumstance Theorem~\ref{thm:propagation_regularity_discrete} give useful results, in the sense that all the additional terms due to the discretization error vanish asymptotically. The main question is how to control the so-called residue through a proper selection of virtual coordinates. We first introduce the key notion of admissible family of virtual coordinates that works for any constant field. 

\begin{defi}\label{defi:admissible}
Consider a mesh
$(\mathcal{C}, \mathcal{F})$
over $\Omega \subset \R^d$.
Let
\begin{equation*}
\begin{aligned}
(\hat x_{i}(b_\text{c}))_{i \in \mathcal{V}} \in (\R^d)^{\mathcal{V}}, \quad \forall b_\text{c} \in \R^d
\end{aligned}
\end{equation*}
be a family of virtual coordinates.
We say that it is an admissible family of virtual coordinates on $\Omega$ for constant fields
with relative drift $M_\gamma$, absolute drift $M_\beta$ and residue bound $M_\xi$ in $L^p$
if the following properties are true:
\begin{enumerate}
\item For each $n \in \N_+$ and any vector $b_\text{c} \in \R^d$, one has
\begin{equation*}
\begin{aligned}
\hat x_{i}(b_\text{c}) = \hat x_{i}(\lambda b_\text{c}), \quad \forall \lambda > 0, i \in \mathcal{V}.
\end{aligned}
\end{equation*}
\item The following bounds hold:
\begin{equation*} 
\begin{aligned} 
\sup_{|b_\text{c}| = 1, (i,i') \in \mathcal{E}}
\big|\hat x_{i}(b_\text{c}) - \hat x_{i'}(b_\text{c})\big| \leq M_\gamma,
\\ 
\sup_{|b_\text{c}| = 1, i \in \mathcal{V}}
|\hat x_{i}(b_\text{c}) - x_i| \leq M_\beta.
\end{aligned}
\end{equation*}
\item For any vector $b_\text{c} \in \R^d$, let $(a_{i,j}(b_\text{c}))_{i,j \in \mathcal{V}} = P_{\mathcal{F}}b$ and $(b_{i}(b_\text{c}))_{i \in \mathcal{V}} = P_{\mathcal{C}}b$ be the discretization of the constant velocity field $b(x) \equiv b_\text{c}$.
Define $(\hat r_{i}(b_c))_{i \in \mathcal{V}}$ as
\begin{equation}\label{eqn:linear_system_constant_residue}
\begin{aligned}
 \sum_{i' \in \mathcal{V}} \Big(\hat x_{i'}(b_\text{c}) - \hat x_{i}(b_\text{c})\Big) a_{i',i}(b_\text{c}) = b_{i}(b_\text{c}) \pi_i + \hat r_{i}(b_c) \pi_i, \quad \forall i \in \mathcal{V}.
\end{aligned}
\end{equation}
Define the maximal residue function $((\hat r_\text{max})_{i})_{i \in \mathcal{V}}$ by
\begin{equation*} 
\begin{aligned}
(\hat r_\text{max})_{i} \defeq \sup_{|b_c| = 1} |r_{i}(b_c)|, \quad \forall i \in \mathcal{V},
\end{aligned}
\end{equation*}
then its $L^p$ norm is bounded by
\begin{equation*} 
\begin{aligned}
\| (\hat r_\text{max})_{i} \|_{L^p(\mathcal{C})} \leq |\Omega|^{1/p} M_\xi.
\end{aligned}
\end{equation*}
\end{enumerate}

\end{defi}

By assuming that an admissible family of virtual coordinates exist, we have the following theorem that controls the residue for any Sobolev velocity field and whose proof is performed in Section \ref{sec:residue_term_estimate}.
\begin{thm} \label{thm:residue_term_estimate}
Consider the semi-discrete scheme \eqref{eqn:continuity_equation_scheme}
on a mesh $(\mathcal{C}, \mathcal{F})$ over a bounded domain $\Omega \subset \R^d$ with piecewise smooth boundary as in Definition~\ref{defi:partition_of_unity_mesh}, having discretization size $\delta x$ and satisfying the structural assumptions \eqref{eqn:mesh_comparable_1} and \eqref{eqn:mesh_comparable_2}. 
Let $(a_{i,j}(t))_{i,j \in \mathcal{V}}$ be the coefficients of scheme \eqref{eqn:continuity_equation_scheme}
and $D(t) = \{D_i(t)\}_{i \in \mathcal{V}}$ be the discrete divergence given by \eqref{eqn:divergence_discrete}.
Assume that $u(t) = (u_i(t))_{i \in \mathcal{V}}$, $t \in [0,T]$ is a solution of the semi-discrete scheme \eqref{eqn:continuity_equation_scheme}.

Moreover, with some $1 \leq p < q \leq \infty$ and $0 < s \leq 1$, assume that there exists a continuous velocity field $\widetilde{b}(t,x)$ bounded in
$L^{q}_t(W^{1,q}_x) \cap L^{p}_x(W^{s,p}_t)([0,T] \times \Omega)$,
and an admissible family of virtual coordinates on $\Omega$ for constant fields, defined in Definition~\ref{defi:admissible}, with relative drift $M_\gamma$, absolute drift $M_\beta$ and residue bound $M_\xi$ in $L^{1/(1/p - 1/q)}$.
In addition, let $\theta \geq \max\{1 - 1/q, 1/2\}$, and assume that $M_\beta \leq h_0 < 1/2$.

Then one can choose a division of time interval $[0,T]$ as $0 = t_0 < t_1 < \dots < t_m = T$, and on each interval $[t_{k-1},t_{k}]$, virtual coordinates $(\widetilde{x}_i^{(k)})_{i \in \mathcal{V}}$, such that the terms ${\alpha}$ and $L_3$ in estimate~\eqref{eqn:propagation_Gronwall} are bounded by
\begin{equation} \label{eqn:propagation_Gronwall_optimize}
\begin{aligned}
&L_3 = C 
(|\log h_0|^{-\theta}/h_0)
\Bigg[
(M_\gamma / \delta x) \big\| (a_{i,j})_{i,j \in \mathcal{V}} - P_{\mathcal{F}} \widetilde{b} \, \big\|_{L^p([0,T] \times \mathcal{F})}
\\
& \quad\quad\quad\quad
+(M_\beta^s M_\gamma / \delta x)^{1/(1+s)}  \|\widetilde{b}\|_{L^{p}_x(W^{s,p}_t)}
\\
& \quad\quad\quad\quad
+ \left( M_\gamma (\delta x)^{-\frac{1}{1+(1/p-1/q)}} \right) \bigg(
\|(a_{i,j}(t))_{i,j \in \mathcal{V}}\|_{L^{q}([0,T] \times \mathcal{F})}
+\|\widetilde{b}\|_{L^{q}_t(W^{1,q}_x)}
\bigg)
\\
& \quad\quad\quad\quad
+ M_\xi   \|\widetilde{b}\|_{L_t^{q}(L_x^{q})}
 \Bigg] \|u\|_{L^\infty_t L^{p^*}_x([0,t] \times \mathcal{C})}
,
\\
&{\alpha} = \exp\big( C(1/h_0) (M_\beta^s M_\gamma / \delta x)^{1/(1+s)} \big),
\end{aligned}
\end{equation}
where the constant $C$ depends on $T, \Omega$, the exponents $p,q,s$, and the constant in structural assumption \eqref{eqn:mesh_comparable_1}.

\end{thm}

\subsection{Step 3: Constructing admissible virtual coordinates for periodic meshes} \label{subsec:step_3}

It remains to study how to find admissible virtual coordinates for constant fields as in Definition~\ref{defi:admissible}.
At this moment, it is still unclear to us whether this is possible for any arbitrary mesh.
Nevertheless,  
for a mesh with periodic pattern, we are able to ensure that one can find an admissible family of virtual coordinates such that the residue $(\hat r_\text{max})_{i}$ actually vanish on any inner cell of the mesh.
\begin{thm} \label{thm:auxiliary_function_existence}
Let $(\mathcal{C}, \mathcal{F})$ be a periodic mesh over $\Omega \subset \R^d$ as in Definition~\ref{defi:periodic_mesh}.
Let $(a_{i,j}(b_\text{c}))_{i,j \in \mathcal{V}} = P_{\mathcal{F}}b$ be the discretization of the constant velocity field $b(x) \equiv b_\text{c}$.
Then for any constant velocity field $b_\text{c} \in \R^d$, there exist virtual coordinates $(\hat x_{i}(b_\text{c}))_{i \in \mathcal{V}} \in (\R^d)^{\mathcal{V}}$ solving the linear system
\begin{equation} \label{eqn:linear_system}
\begin{aligned}
 \sum_{i' \in \mathcal{V}} \Big(\hat x_{i'}(b_\text{c}) - \hat x_{i}(b_\text{c})\Big) a_{i',i}(b_\text{c}) = b_\text{c} \pi_i, \quad \forall i \in \mathcal{V}_0,
\end{aligned}
\end{equation}
and satisfying the following properties:
\begin{enumerate}
\item The virtual coordinates are homogeneous in the sense that
\begin{equation*}
\begin{aligned}
\hat x_{i}(b_\text{c}) = \hat x_{i}(\lambda b_\text{c}), \quad \forall b_\text{c} \in \R^d, \lambda > 0, i \in \mathcal{V}.
\end{aligned}
\end{equation*}

\item The virtual coordinates are uniformly bounded by
\begin{equation*}
\begin{aligned}
\sup_{|b_\text{c}| = 1, i \in \mathcal{V}}
|\hat x_{i}(b_\text{c}) - x_i| < C(|\mathcal{V}_0|) \delta x,
\end{aligned}
\end{equation*}
where
$\delta x$ is the discretization size
and $C(|\mathcal{V}_0|)$ depends only on the number of cell functions in a period.

\item The virtual coordinates are periodic in the sense that
\begin{equation*}
\begin{aligned}
\hat x_{[m](i)}(b_\text{c}) = \hat x_{i}(b_\text{c}) + \sum_{k=1}^d m_{k} L_k, \quad \forall b_\text{c} \in \R^d, \; i, [m](i) \in \mathcal{V},
\end{aligned}
\end{equation*}
where $L_1, \dots, L_n \in \R^d$ are given as in Definition~\ref{defi:periodic_mesh}.
\end{enumerate}
\end{thm}
Notice that for constant velocity field, one has $a_{[m](i'),[m](i)} = a_{i',i}$ when $[m](i'),[m](i)$ are well-defined. Hence \eqref{eqn:linear_system} can be naturally extend to all $i \in \mathcal{V}_\Omega^\circ$.
That is, the maximal residue function $((\hat r_\text{max})_{i})_{i \in \mathcal{V}}$ in Definition~\ref{defi:admissible} vanishes at all $i \in \mathcal{V}_\Omega^\circ$.

\subsection{Proof of Theorem~\ref{thm:periodic_regularity_result}}

We are now ready to complete the proof of Theorem~\ref{thm:periodic_regularity_result}.

\begin{proof}[Proof of Theorem~\ref{thm:periodic_regularity_result}]

We apply Theorem~\ref{thm:auxiliary_function_existence}, Theorem~\ref{thm:propagation_regularity_discrete} and Theorem~\ref{thm:residue_term_estimate} successively:

\begin{itemize}

\item \textbf{Apply Theorem~\ref{thm:auxiliary_function_existence}:} Let $(\mathcal{C}, \mathcal{F})$ be one mesh in the sequence of meshes we consider in Theorem~\ref{thm:periodic_regularity_result} and let $\delta x$ be the discretization size.
Since the meshes have periodic patterns over $\Omega$, one can choose $\hat x_i(b_\text{c})$ as the periodic solution in Theorem~\ref{thm:auxiliary_function_existence} when $i \in \mathcal{V}_\Omega$, and choose $\hat x_i(b_\text{c}) = x_i$ as barycenter when $i \in \mathcal{V} \setminus \mathcal{V}_\Omega$.
Then the maximal residue function $((\hat r_\text{max})_{i})_{i \in \mathcal{V}}$ in Definition~\ref{defi:admissible} vanishes at all $i \in \mathcal{V}_\Omega^\circ$.
Notice that
\begin{equation*}
\begin{aligned}
\sup_{|b_\text{c}| = 1, i \in \mathcal{V}}
|\hat x_{i}(b_\text{c}) - x_i| < C(|\mathcal{V}_0|) \delta x
\end{aligned}
\end{equation*}
where $C(|\mathcal{V}_0|)$ is the constant in Theorem~\ref{thm:auxiliary_function_existence}.

By our definition of discretization size and barycenter one has
\begin{equation*}
\begin{aligned}
\sup_{(i,i') \in \mathcal{E}}
\big|x_{i} - x_{i'}\big| \leq 2\delta x.
\end{aligned}
\end{equation*}
Moreover, if $i \in \mathcal{V} \setminus \overline{\mathcal{V}}_\Omega$, by definition, one has $(P_{\mathcal{C}}b)_{i} = 0$ and $(P_{\mathcal{F}}b)_{i',i} = 0$ for all $i' \in \mathcal{V}$.
Hence, one has $b_i \equiv 0$, $a_{i',i} \equiv 0$ in \eqref{eqn:linear_system_constant_residue} and the residue $((\hat r_\text{max})_{i})_{i \in \mathcal{V}}$ in Definition~\ref{defi:admissible} actually vanishes on all $i \in \mathcal{V} \setminus \overline{\mathcal{V}}_\Omega$.

In summary, the residue can be non-zero only at $i \in \overline{\mathcal{V}}_\Omega \setminus \mathcal{V}_\Omega^\circ$. Thus one has
that
\begin{equation*}
\begin{aligned}
|\Omega|^{-1/r} \| (\hat r_\text{max})_{i} \|_{L^r(\mathcal{V})} &\leq 
\left(|\Omega|^{-1} \sum_{i \in \overline{\mathcal{V}}_\Omega \setminus \mathcal{V}_\Omega^\circ}
\sup_{|b_\text{c}| = 1} \bigg(
\sum_{i' \in \mathcal{V}} \Big(\hat x_{i'}(b_\text{c}) - \hat x_{i}(b_\text{c})\Big) a_{i',i} (b_\text{c}) - b_\text{c} \pi_i \bigg) \right)^{1/r}
\\
&\leq
\left(|\Omega|^{-1} \sum_{i \in \overline{\mathcal{V}}_\Omega \setminus \mathcal{V}_\Omega^\circ}
C(\delta x)^d \right)^{1/r}
\\
&\leq
C \left(|\Omega|^{-1} |\partial \Omega| \delta x \right)^{1/r}.
\end{aligned}
\end{equation*}
Let $r = {1/(1/p - 1/q)}$ as required in Theorem~\ref{thm:residue_term_estimate}. One can see that such choice of $\hat x_i(b_\text{c})$ forms an admissible family of virtual coordinates in Definition~\ref{defi:admissible}, with relative drift $M_\gamma$, absolute drift $M_\beta$ and residue bound $M_\xi$ in $L^{1/(1/p - 1/q)}$ given by
\begin{equation*}
\begin{aligned}
M_\beta = M_\gamma = 2\big(C(|\mathcal{V}_0|) + 1\big)\delta x, \quad M_\xi = C \left(|\Omega|^{-1} |\partial \Omega| \delta x \right)^{1/p - 1/q}.
\end{aligned}
\end{equation*}
The constant $C$ here depends on $\Omega$, the exponents $p,q$, the constant in structural assumption \eqref{eqn:mesh_comparable_1}, and, in particular, the number of cell functions in a period. This is why the constant bounding pattern size is part of the requirement of Theorem~\ref{thm:periodic_regularity_result}.

Moreover, to fulfill the requirement in Theorem~\ref{thm:propagation_regularity_discrete}, the discretization size $\delta x$ needs to be chosen small so that $M_\beta \leq \min\{h_0, 1/32\}$, which is why the inequality only holds asymptotically.

\medskip

\item \textbf{Apply Theorem~\ref{thm:propagation_regularity_discrete} and Theorem~\ref{thm:residue_term_estimate}:}
We have proved the existence of an admissible family of virtual coordinates for constant velocity fields. Recall that Theorem~\ref{thm:residue_term_estimate} shows how to reduce the residue term for non-constant velocity fields provided such family exists.
Hence, we choose the coefficients $a(t) = (a_{i,j})_{i,j \in \mathcal{V}}$, $\widetilde{b}(t)$ and the solution $u(t) = (u_i(t))_{i \in \mathcal{V}}$ as in Theorem~\ref{thm:periodic_regularity_result}.
Then the rest conditions required in Theorem~\ref{thm:residue_term_estimate}, namely the boundedness of their Lebesgue and Sobolev norms and the boundedness of
\begin{equation*}
\begin{aligned}
\big\| (a_{i,j})_{i,j \in \mathcal{V}} - P_{\mathcal{F}} \widetilde{b}\, \big\|_{L^q([0,T] \times \mathcal{F})},
\end{aligned}
\end{equation*}
are directly guaranteed by the assumptions in Theorem~\ref{thm:periodic_regularity_result}.
According to Theorem~\ref{thm:residue_term_estimate}, one can choose virtual coordinates for specific coefficients, such that the propagation of regularity \eqref{eqn:propagation_Gronwall} in Theorem~\ref{thm:propagation_regularity_discrete}, is bounded by \eqref{eqn:propagation_Gronwall_optimize}.

For clarity, we recall the full result,
\begin{equation*}
\begin{aligned}
&\|u(t)\|_{h_0,1,\theta} \leq \alpha\, (L_0 + L_1 + L_2 + L_3),
\end{aligned}
\end{equation*}
where the precise formulations of these terms are given by
\begin{equation*}
\begin{aligned}
&L_0 = \|u(0)\|_{h_0,1,\theta},
\\
&L_1 = C\int_0^t 
 \bigg( \|\dive \widetilde{b}(s)\|_{L^\infty(\mathcal{C})} \|u(s)\|_{h_0,1,\theta}  + \|\widetilde{b}(s)\|_{W^{1,q}} \|u(s)\|_{L^{p^*}(\mathcal{C})}
\\ 
&\quad \quad \quad \quad + \|D(s)\|_{L^\infty(\mathcal{C})} \|u(s)\|_{h_0,1,\theta} + \|u(s)\|_{L^{p^*}(\mathcal{C})} \|D(s)\|_{h_0,p,p(\theta - 1/p^*)} \bigg) \;\rd s,
\end{aligned}
\end{equation*}
where the constant $C$ depends on $\Omega$, the exponents $p,q$ and the constant in structural assumptions $\eqref{eqn:mesh_comparable_1}$.
Also,
\begin{equation*}
\begin{aligned}
&\alpha = \exp\big( C(1/h_0) (M_\beta^s M_\gamma / \delta x)^{1/(1+s)} \big),
\\
&L_2 =  C \int_0^t \bigg(\; \big( |\log h_0|^{-\theta} M_\gamma^2 / h_0^2 \delta x \big) \|(a_{i,j}(s))_{i,j \in \mathcal{V}}\|_{L^q(\mathcal{F})} \|u(s)\|_{L^{q^*}(\mathcal{C})}
\\ 
&\quad\quad\quad\quad\quad + \big(|\log h_0|^{-\theta} M_\beta / h_0^2 \big) \|\widetilde{b}(s)\|_{L^q} \|u(s)\|_{L^{q^*}(\mathcal{C})}
\\
&\quad\quad\quad\quad\quad + (|\log h_0|^{-\theta} \delta x / h_0^2) \|\widetilde{b}(s)\|_{W^{1,q}} \|u(s)\|_{L^{q^*}(\mathcal{C})} \bigg) \;\rd s
 \\
&\quad\quad +C ( |\log h_0|^{1-\theta} ) \Big( \|u(t)\|_{L^1(\mathcal{C})} - \|u(0)\|_{L^1(\mathcal{C})} \Big),
\\ 
&L_3 = C 
(|\log h_0|^{-\theta}/h_0)
\Bigg[
(M_\gamma / \delta x) \big\| (a_{i,j})_{i,j \in \mathcal{V}} - P_{\mathcal{F}} \widetilde{b} \, \big\|_{L^q([0,T] \times \mathcal{F})}
\\
& \quad\quad\quad\quad
+(M_\beta^s M_\gamma / \delta x)^{1/(1+s)}  \|\widetilde{b}\|_{L^{p}_x(W^{s,p}_t)}
\\
& \quad\quad\quad\quad
+ \left( M_\gamma (\delta x)^{-\frac{1}{1+(1/p-1/q)}} \right) \bigg( 
\|(a_{i,j}(t))_{i,j \in \mathcal{V}}\|_{L^{q}([0,T] \times \mathcal{F})}
+ \|\widetilde{b}\|_{L^{q}_t(W^{1,q}_x)}
\bigg)
\\
& \quad\quad\quad\quad
+ M_\xi \|\widetilde{b}\|_{L_t^{q}(L_x^{q})}
 \Bigg] \|u\|_{L^\infty_t L^{p^*}_x([0,t] \times \mathcal{C})}.
\end{aligned}
\end{equation*}
where the constant $C$ depends on $T, \Omega$, the exponents $p,q,s$, the constant in structural assumption \eqref{eqn:mesh_comparable_1} and the constant bounding pattern size.

\item \textbf{Conclude Theorem~\ref{thm:periodic_regularity_result}}
We can now make all the constants explicit in the propagation of regularity with,
\begin{equation*}
\begin{aligned}
(M_\beta^s M_\gamma / \delta x)^{1/(1+s)} \lesssim (\delta x)^{1/(1+s)}, \quad \quad
&M_\gamma^2 / \delta x, M_\beta \lesssim \delta x,
\\
M_\gamma (\delta x)^{-\frac{1}{1+(1/p-1/q)}} \lesssim (\delta x)^{\frac{1/p-1/q}{1+(1/p-1/q)}}, \quad \quad
&M_\xi \lesssim (\delta x)^{1/p-1/q} \lesssim (\delta x)^{\frac{1/p-1/q}{1+(1/p-1/q)}}.
\end{aligned}
\end{equation*}
Thus $\alpha,\, L_2,\, L_3$ can be expressed as
\begin{equation*}
\begin{aligned}
&\alpha = \exp\big( C(1/h_0) (\delta x)^{s/(1+s)} \big),
\\
&L_2 = C \int_0^t \big( |\log h_0|^{-\theta} / h_0^2 \big) (\delta x) \Big( \|a(s)\|_{L^q(\mathcal{F})} + \|\widetilde{b}(s)\|_{W^{1,q}} \Big) \|u(s)\|_{L^{q^*}(\mathcal{C})} \;\rd s
\\
& \quad\quad\quad\quad+ ( |\log h_0|^{1-\theta} ) \Big( \|u(t)\|_{L^1(\mathcal{C})} - \|u(0)\|_{L^1(\mathcal{C})} \Big),
\\
&L_3 = C 
(|\log h_0|^{-\theta}/h_0)
\Bigg[
\big\| a - P_{\mathcal{F}} \widetilde{b}\, \big\|_{L^{q}([0,T] \times \mathcal{F})}
+(\delta x)^{s/(1+s)}  \|\widetilde{b}\|_{L^{p}_x(W^{s,p}_t)}
\\
& \quad\quad\quad\quad
+ (\delta x)^{\frac{1/p-1/q}{1+(1/p-1/q)}} \bigg(
\|a(t)\|_{L^{q}([0,T] \times \mathcal{F})} + \|\widetilde{b}\|_{L^{q}_t(W^{1,q}_x)}
\bigg)
\Bigg] \|u\|_{L^\infty_t L^{p^*}_x([0,t] \times \mathcal{C})}.
\end{aligned}
\end{equation*}
One can check that this exactly corresponds to \eqref{eqn:main_regularity_result} of Theorem~\ref{thm:periodic_regularity_result}.
Applying the argument to the entire sequence of meshes $(\mathcal{C}^{(n)}, \mathcal{F}^{(n)})$ concludes the proof of Theorem~\ref{thm:periodic_regularity_result}.
\end{itemize}
\end{proof}

\section{Proof of Theorem~\ref{thm:propagation_regularity_discrete}} \label{sec:propagation_regularity_discrete}

In this section we complete the proof of Theorem~\ref{thm:propagation_regularity_discrete}. For simplicity we omit the variable $t$ in the derivation when there is no ambiguity.
Throughout the calculation, $C$ denotes a generic constant that only depends on $\Omega$, the exponents $p,q$ and the constant in the structural assumptions~$\eqref{eqn:mesh_comparable_1}$.

\subsection{The Kruzkov's doubling of variables for the semi-discrete scheme}
Notice that any solution $u$ of the first-order scheme \eqref{eqn:continuity_equation_scheme} satisfies the following identity:
\begin{equation} \label{eqn:upwind_one_line}
\begin{aligned}
\frac{\rd}{\rd t} u_i = 
\frac{1}{\pi_i} \sum_{j \in \mathcal{V}} \Big( a_{i,j} u_{j} - a_{j,i} u_{i} \Big)
- \mathbbm{1}_{\mathcal{V}\setminus \mathcal{V}_{\Omega}^\circ}(i) \frac{1}{\pi_i} \sum_{j \in \mathcal{V}} a_{i,j} u_j.
\end{aligned}
\end{equation}
When $i \in \mathcal{V}_{\Omega}^\circ$ the above equality is the upwind scheme. When  $i \in (\mathcal{V}\setminus \mathcal{V}_{\Omega}^\circ)$, one has $u_i \equiv 0$ and the above equality reduce to $0 = 0$.
In this section let us use the notation
\begin{equation*}
\begin{aligned}
R_i \defeq
- \mathbbm{1}_{\mathcal{V}\setminus \mathcal{V}_{\Omega}^\circ}(i) \frac{1}{\pi_i} \sum_{j \in \mathcal{V}} a_{i,j} u_j.
\end{aligned}
\end{equation*}
The term $R_i$ measures the possible leaking of mass at boundary. It is easy to verify that $R_i \leq 0$ and
\begin{equation} \label{eqn:total_leaking}
\begin{aligned}
\frac{\rd}{\rd t} \sum_{i \in \mathcal{V}} u_i \pi_i = \sum_{i \in \mathcal{V}} R_i \pi_i.
\end{aligned}
\end{equation}
The next proposition explains how to bound the time derivative of our semi-norms.
\begin{prop} \label{prop:scheme_Kruzkov}
For any solution $u$ of the first-order scheme \eqref{eqn:continuity_equation_scheme} and any non-negative discrete kernel $\{K_{i,j}\}_{i,j \in \mathcal{V}}$, the following inequality holds in the sense of distribution:
\begin{equation*}
\begin{aligned}
\quad \frac{\rd}{ \rd t} \sum_{i,j \in \mathcal{V}} K_{i,j} |u_i - u_j| \pi_i \pi_j
&\leq 2 \sum_{i,j \in \mathcal{V}} \sum_{i' \in \mathcal{V}} (K_{i',j} - K_{i,j}) a_{i',i} |u_i - u_j| \pi_j 
\\ 
&\quad +(- 2)\sum_{i,j \in \mathcal{V}} K_{i,j} \sgn (u_i - u_j)
\left( D_i u_j \right) \pi_i \pi_j
\\
&\quad + 2 \sum_{i,j \in \mathcal{V}} K_{i,j} \sgn(u_i - u_j) R_i \pi_i \pi_j
\\
&\quad \eqdef A_K + D_K + R_K.
\end{aligned}
\end{equation*}
\end{prop}

\begin{proof}[Proof of Proposition~\ref{prop:scheme_Kruzkov}]
The following equality holds in distributional sense:
\begin{equation*}
\begin{aligned}
&\quad \frac{\rd}{ \rd t} \sum_{i,j \in \mathcal{V}} K_{i,j} |u_i - u_j| \pi_i \pi_j
\\
&= \sum_{i,j \in \mathcal{V}} K_{i,j} \sgn(u_i - u_j) \left( \frac{1}{\pi_i} \sum_{i' \in \mathcal{V}} \Big( a_{i,i'} u_{i'} - a_{i',i} u_{i} \Big) + R_i - \frac{1}{\pi_j} \sum_{j' \in \mathcal{V}} \Big( a_{j,j'} u_{j'} - a_{j',j} u_{j} \Big) - R_j \right) \pi_i \pi_j
\\
&= 2 \sum_{i,j \in \mathcal{V}} K_{i,j} \sgn(u_i - u_j) \left( \frac{1}{\pi_i} \sum_{i' \in \mathcal{V}} \Big( a_{i,i'} u_{i'} - a_{i',i} u_{i} \Big) + R_i \right) \pi_i \pi_j.
\end{aligned}
\end{equation*}
Proving the first equality is nothing but the chain rule applied to the semi-discrete scheme. The second equality follows from symmetry,  by switching the indexes $i$ and $j$, $i'$ and $j'$.

The next step is to check that this can be further decomposed into our sum $A_K + D_K + R_K$ plus a non-positive term. Indeed,
\begin{equation*}
\begin{aligned} 
&\quad \sum_{i,j \in \mathcal{V}} K_{i,j} \sgn(u_i - u_j) \left( \frac{1}{\pi_i} \sum_{i' \in \mathcal{V}} \Big( a_{i,i'} u_{i'} - a_{i',i} u_{i} \Big)  \Big) \right) \pi_i \pi_j
\\ 
&= \sum_{i,j \in \mathcal{V}}\sum_{i' \in \mathcal{V}} \Bigg\{ \;{K_{i,j} \sgn(u_i - u_j) a_{i,i'} u_{i'}}  \;{- K_{i,j} \sgn(u_i - u_j) a_{i',i} u_{i} } \Bigg\} \pi_j
\\ 
&= \sum_{i,j \in \mathcal{V}}\sum_{i' \in \mathcal{V}} \Bigg\{ \;{ K_{i',j} \sgn(u_i - u_j) a_{i',i} u_{i}}  \;{- K_{i',j} \sgn(u_i - u_j) a_{i',i} u_{j}}
\\ 
&\quad\quad\quad\quad\quad\quad\quad \;{-K_{i,j} \sgn(u_i - u_j) a_{i',i} u_{i}}  \;{+ K_{i,j} \sgn(u_i - u_j) a_{i',i} u_{j}}
\Bigg\} \pi_j
\\ 
&\quad + \sum_{i,j \in \mathcal{V}}\sum_{i' \in \mathcal{V}} \Bigg\{ \;{- K_{i,j} \sgn(u_i - u_j) a_{i',i} u_j} \;{+ K_{i,j} \sgn(u_i - u_j) a_{i,i'} u_j} \Bigg\} \pi_j
\\ 
&\quad + \sum_{i,j \in \mathcal{V}}\sum_{i' \in \mathcal{V}} \Bigg\{ \;{ - K_{i,j} \sgn(u_{i'} - u_j) a_{i,i'} u_{i'}} \;{ + K_{i,j} \sgn(u_i - u_j) a_{i,i'} u_{i'}} \Bigg\} \pi_j
\\ 
&\quad\quad\quad\quad\quad\quad\quad \;{+ K_{i,j} \sgn(u_{i'} - u_j) a_{i,i'} u_{j}} \;{ - K_{i,j} \sgn(u_i - u_j) a_{i,i'} u_{j}} \Bigg\} \pi_j
\\
&\quad \eqdef A_K' + D_K' + N_K.
\end{aligned}
\end{equation*}
By our assumption, $K_{i,j} \geq 0$, $a_{i,j} \geq 0$ for all $i,j \in \mathcal{V}$.
It is easy to verify that the third term
\begin{equation*}
\begin{aligned} 
N_K &= \sum_{i,j \in \mathcal{V}}\sum_{i' \in \mathcal{V}} \Bigg\{ \;{ - K_{i,j} \sgn(u_{i'} - u_j) a_{i,i'} u_{i'}} \;{+ K_{i,j} \sgn(u_i - u_j) a_{i,i'} u_{i'}} \Bigg\} \pi_j
\\ 
&\quad\quad\quad\quad\quad\quad \;{+ K_{i,j} \sgn(u_{i'} - u_j) a_{i,i'} u_{j}} \;{ - K_{i,j} \sgn(u_i - u_j) a_{i,i'} u_{j}} \Bigg\} \pi_j
\\
&= \sum_{i,j \in \mathcal{V}}\sum_{i' \in \mathcal{V}} K_{i,j} \big[ \sgn(u_{i'} - u_j) - \sgn(u_i - u_j) \big] a_{i,i'} (u_{j} - u_{i'}) \pi_j
\end{aligned}
\end{equation*}
is always non-positive, and it does not vanish only at edges that $\sgn(u_{i'} - u_{j}) \neq \sgn(u_{i} - u_{j})$.
In addition, one can reformulate $A_K'$ and $D_K'$ as
\begin{equation*}
\begin{aligned} 
A_K' &= \sum_{i,j \in \mathcal{V}}\sum_{i' \in \mathcal{V}} \Bigg\{ \;{ K_{i',j} \sgn(u_i - u_j) a_{i',i} u_{i}}  \;{- K_{i',j} \sgn(u_i - u_j) a_{i',i} u_{j}}
\\ 
&\quad\quad\quad\quad\quad\quad\quad \;{-K_{i,j} \sgn(u_i - u_j) a_{i',i} u_{i}}  \;{+ K_{i,j} \sgn(u_i - u_j) a_{i',i} u_{j}}
\Bigg\} \pi_j
\\
&= \sum_{i,j \in \mathcal{V}} \sum_{i' \in \mathcal{V}} (K_{i',j} - K_{i,j}) a_{i',i} |u_i - u_j| \pi_j,
\\
D_K' &= \sum_{i,j \in \mathcal{V}} \sum_{i' \in \mathcal{V}} \Bigg\{ \;{- K_{i,j} \sgn(u_i - u_j) a_{i',i} u_j} \;{+ K_{i,j} \sgn(u_i - u_j) a_{i,i'} u_j} \Bigg\} \pi_j
\\
&= - \sum_{i,j \in \mathcal{V}} K_{i,j} \sgn (u_i - u_j)
\left( \sum_{i' \in \mathcal{V}} \big(a_{i',i} - a_{i,i'}\big) \right) u_j \pi_i \pi_j.
\end{aligned}
\end{equation*}
It is straightforward to see that $A_K' + D_K' = A_K/2 + D_K/2$. Hence
\begin{equation*}
\begin{aligned} 
&\quad \sum_{i,j \in \mathcal{V}} K_{i,j} \sgn(u_i - u_j) \left( \frac{1}{\pi_i} \sum_{i' \in \mathcal{V}} \Big( a_{i,i'} u_{i'} - a_{i',i} u_{i} \Big) + R_i \right) \pi_i \pi_j
\\
&= A_K' + D_K' + N_K + R_K/2
\\
&\leq A_K/2 + D_K/2 + R_K/2.
\end{aligned}
\end{equation*}
Multiplying both sides by $2$, one obtains the inequality in the proposition.
\end{proof}

From now on we fix the kernel $K_{i,j}$ in the above proposition as $\widetilde{K}_{i,j}^h$ in Definition~\ref{defi:semi-norm_discrete}
for $0 < h < 1/2$ and $\widetilde{x} = (\widetilde{x}_i)_{i \in \mathcal{V}} \in (\R^d)^{\mathcal{V}}$. 
Moreover, assume that $h \geq \max\{ \delta x, \sup_{i \in \mathcal{V}}|\widetilde{x}_i - x_i| \}$.
Then the term $R_K$ can be bounded by
\begin{equation} \label{eqn:estimate_residue_discrete}
\begin{aligned}
|R_K| &= \bigg| 2 \sum_{i,j \in \mathcal{V}} \widetilde{K}_{i,j}^h \sgn(u_i - u_j) R_i \pi_i \pi_j \bigg|
\leq 2 \sum_{i \in \mathcal{V}}\bigg\{ \Big(\sum_{j \in \mathcal{V}} \widetilde{K}_{i,j}^h \pi_j \Big) |R_i| \pi_i \bigg\}
\\
&\leq C|\log h| \|(R_i)_{i \in \mathcal{V}}\|_{L^1(\mathcal{C})}.
\end{aligned}
\end{equation}
Moreover, the term $D_K$ can then be estimated through
\begin{equation} \label{eqn:estimate_divergence_discrete}
\begin{aligned} 
|D_K| &= \bigg| \; 2\sum_{i,j \in \mathcal{V}} \widetilde{K}_{i,j}^h \sgn (u_i - u_j)
\left( D_i u_j \right) \pi_i \pi_j \bigg|
\\
&= \bigg| \sum_{i,j \in \mathcal{V}} \widetilde{K}_{i,j}^h \sgn (u_i - u_j)
\left( D_i u_j - D_j u_i \right) \pi_i \pi_j \bigg|
\\
&= \bigg| \sum_{i,j \in \mathcal{V}} \widetilde{K}_{i,j}^h \frac{D_i + D_j}{2} |u_i - u_j| \;\pi_i \pi_j
\\
&\quad + \sum_{i,j \in \mathcal{V}} \widetilde{K}_{i,j}^h (D_i - D_j) \frac{u_i + u_j}{2} \sgn(u_i - u_j) \;\pi_i \pi_j \bigg|
\\ 
&\quad \leq
C |\log h|^{\theta} \left( \|D\|_{L^\infty(\mathcal{C})} \|u\|_{h_0,1,\theta;\widetilde{x}} + \|u\|_{L^{p^*}(\mathcal{C})} \|D\|_{h_0,p,p(\theta - 1/p^*);\widetilde{x}} \right).
\end{aligned}
\end{equation}
The last inequality is a consequence of the following two estimations:
Using the bound on the divergence, one has
\begin{equation*}
\begin{aligned}
&\bigg| \sum_{i,j \in \mathcal{V}} \widetilde{K}_{i,j}^h \frac{D_i + D_j}{2} |u_i - u_j| \;\pi_i \pi_j \bigg|
\\
&\leq \|D\|_{L^\infty(\mathcal{V})} \sum_{i,j \in \mathcal{V}} \widetilde{K}_{i,j}^h |u_i - u_j| \;\pi_i \pi_j \leq |\log h|^\theta \|D\|_{L^\infty(\mathcal{C})} \|u\|_{\alpha,1,\theta;\widetilde{x}}.
\end{aligned}
\end{equation*}
Also, by H\"older estimate
\begin{equation*}
\begin{aligned}
&\bigg| \sum_{i,j \in \mathcal{V}} \widetilde{K}_{i,j}^h (D_i - D_j) \frac{u_i + u_j}{2} \sgn(u_i - u_j) \;\pi_i \pi_j \bigg|
\\
&\leq \bigg( \sum_{i,j \in \mathcal{V}} \widetilde{K}_{i,j}^h |D_i - D_j|^p \;\pi_i \pi_j \bigg)^{1/p} \bigg( \sum_{i,j \in \mathcal{V}} \widetilde{K}_{i,j}^h \Big|\frac{u_i + u_j}{2}\Big|^{p^*} \;\pi_i \pi_j \bigg)^{1/p^*}
\\
&\leq C |\log h|^{\theta} \|u\|_{L^{p^*}(\mathcal{C})} \|D\|_{p,p(\theta - 1/p^*);\widetilde{x}}.
\end{aligned}
\end{equation*}
The above H\"older estimate is for $1 < p < \infty$ but can be extended to $p = 1$ in the obvious way.

\subsection{Bounding the discrete commutator term}
We now investigate the discrete commutator term $A_K$ when $K_{i,j}$ is chosen as $\widetilde{K}_{i,j}^h$ in Definition~\ref{defi:semi-norm_discrete} for $\widetilde{x} = (\widetilde{x}_i)_{i \in \mathcal{V}} \in (\R^d)^{\mathcal{V}}$ and $\max\{ \delta x, \sup_{i \in \mathcal{V}}|\widetilde{x}_i - x_i| \} \leq h < 1/2$.
Recall that
\begin{equation} \label{eqn:commutator_A}
\begin{aligned}
A_K / 2 = \sum_{i,j \in \mathcal{V}} \sum_{i' \in \mathcal{V}} (\widetilde{K}_{i',j}^h - \widetilde{K}_{i,j}^h) a_{i',i} |u_i - u_j| \pi_j.
\end{aligned}
\end{equation}
We begin by a short lemma about the scaling of the continuous kernel $K^h$.
\begin{lem} \label{lem:parallelogram}
Take $x, y, s \in \R^d$ such that $0 < h < 1/2$ and $|s| < h$. Then
\begin{equation}\label{eqn:parallelogram_1}
\begin{aligned}
\big| K^h(x-y) - K^h(x-y+s) + \nabla K^h(x-y) \cdot s \big| \leq \frac{C |s|^{2}}{(|x - y| + h)^{d+2}}.
\end{aligned}
\end{equation}
Also,
\begin{equation}\label{eqn:parallelogram_2}
\begin{aligned}
\big| \nabla K^h(x-y+s) - \nabla K^h(x-y) \big| \leq \frac{C |s|}{(|x - y| + h)^{d+2}}.
\end{aligned}
\end{equation}
\end{lem}
We are going to use this lemma to reduce a few terms to simpler forms, with a tolerable error.
In particular, the following lemma mimics the continuous commutator estimate in \cite{BeJa:19}, provided that one can find suitable auxiliary functions $(\widetilde{x}_i)_{i \in \mathcal{V}}$ and $(b_i)_{i \in \mathcal{V}}$.

\begin{lem}\label{lem:linear_system}

Consider the semi-discrete scheme \eqref{eqn:continuity_equation_scheme}
on a mesh $(\mathcal{C},\mathcal{F})$ over $\Omega \subset \R^d$ as in Definition~\ref{defi:partition_of_unity_mesh}, having discretization size $\delta x$ and satisfying the structural assumptions \eqref{eqn:mesh_comparable_1}. 
Let $(a_{i,j})_{(i,j) \in \mathcal{E}}$ be the coefficients of the scheme and let $D = (D_{i})_{i \in \mathcal{V}}$ be the discrete divergence defined as in \eqref{eqn:divergence_discrete}.
Let $\widetilde{b}(x)$ be a continuous velocity field on $\R^d$ and denote $(\widetilde{b}_i)_{i \in \mathcal{V}} = P_{\mathcal{V}} \widetilde{b}$.
Choose virtual coordinates $(\widetilde{x}_i)_{i \in \mathcal{V}}$ on the mesh satisfying
\begin{equation} \label{eqn:drift_time_independent}
\begin{aligned} 
|\widetilde{x}_i - \widetilde{x}_{i'}| &< M_\gamma, \quad \forall (i,i') \in \mathcal{E},
\\ 
|\widetilde{x}_i - x_i| &< M_\beta, \quad \forall i \in \mathcal{V}.
\end{aligned}
\end{equation}
Let
$\widetilde{K}_{i,j}^h$ be as in Definition~\ref{defi:semi-norm_discrete}
corresponding to $(\widetilde{x}_i)_{i \in \mathcal{V}}$
and
let $(r_i(t))_{i \in \mathcal{V}}$ be the residue function given by
\begin{equation} \label{eqn:linear_system_time_independent}
\begin{aligned}
 \sum_{i' \in \partial \{i\}} (\widetilde{x}_{i'} - \widetilde{x}_i) a_{i',i} = \widetilde{b}_i \pi_i + r_i \pi_i, \quad \forall i \in \mathcal{V}.
\end{aligned}
\end{equation}
Then the discrete commutator term $A_K$ given through \eqref{eqn:commutator_A} can be bounded by
\begin{equation} \label{eqn:estimate_commutator_discrete}
\begin{aligned}
|A_K| &\leq C |\log h|^{\theta} \left( \|\dive \widetilde{b}\|_{L^\infty} \|u\|_{h_0,1,\theta;\widetilde{x}} + \|\widetilde{b}\|_{W^{1,q}} \|u\|_{L^{p^*}(\mathcal{C})}\right)
\\
&\quad + C \big( M_\gamma^2 / h^2 \delta x \big) \|(a_{i,j})_{i,j\in \mathcal{V}}\|_{L^q(\mathcal{F})} \|u\|_{L^{q^*}} + C \big(M_\beta / h^2 \big) \|(\widetilde{b}_i)_{i \in \mathcal{V}}\|_{L^q(\mathcal{C})} \|u\|_{L^{q^*}}
\\
&\quad + C(1/h) \|r\|_{L^p} \|u\|_{L^{p^*}},
\end{aligned}
\end{equation}
provided that $1 \leq p < q \leq \infty$, $\theta \geq \max\{1 - 1/q, 1/2\}$, and $M_\gamma < M_\beta < h_0 < h$.

\end{lem}

Note that conditions \eqref{eqn:drift_time_independent} and \eqref{eqn:linear_system_time_independent} 
exactly correspond to \eqref{eqn:drift} and \eqref{eqn:linear_system_time_dependent} in Theorem~\ref{thm:propagation_regularity_discrete} once the time-dependency is removed.
\begin{proof} We have that
\begin{equation*}
\begin{aligned}
A_K/2 &= \sum_{i,j \in \mathcal{V}} \sum_{i' \in \mathcal{V}} (\widetilde{K}_{i',j}^h - \widetilde{K}_{i,j}^h) a_{i',i} |u_i - u_j| \pi_j
\\ 
&= \sum_{i,j \in \mathcal{V}} \sum_{i' \in \mathcal{V}} \nabla K^h(\widetilde{x}_i - \widetilde{x}_j) \cdot (\widetilde{x}_{i'} - \widetilde{x}_i) a_{i'i} |u_i - u_j| \pi_j
\\ 
&\quad+ \sum_{i,j \in \mathcal{V}} \sum_{i' \in \mathcal{V}} \big[ (\widetilde{K}_{i',j}^h - \widetilde{K}_{i,j}^h) - \nabla K^h(\widetilde{x}_i - \widetilde{x}_j) \cdot (\widetilde{x}_{i'} - \widetilde{x}_i) \big] a_{i'i} |u_i - u_j| \pi_j
\\
&\eqdef A_K^{(1)} + A_K^{(2)}
\end{aligned}
\end{equation*}

By Lemma~\ref{lem:parallelogram} and assumption \eqref{eqn:drift_time_independent}, one has that
\begin{equation*}
\begin{aligned}
\big| (\widetilde{K}_{i',j}^h - \widetilde{K}_{i,j}^h) - \nabla K^h(\widetilde{x}_i - \widetilde{x}_j) \cdot (\widetilde{x}_{i'} - \widetilde{x}_i) \big| \leq \frac{C M_\gamma^2}{(|\widetilde{x}_i - \widetilde{x}_j| + h)^{d+2}} \leq \frac{C M_\gamma^2}{(|x_i - x_j| + h)^{d+2}}.
\end{aligned}
\end{equation*}
Therefore
one can bound $A_K^{(2)}$ by 
\begin{equation*}
\begin{aligned}
\big| A_K^{(2)} \big| 
&\leq
\sum_{i,j \in \mathcal{V}} \sum_{i' \in \mathcal{V}} \frac{C M_\gamma^2}{(|x_i - x_j| + h)^{d+2}} a_{i'i} |u_i - u_j| \pi_j
\\
&\leq \sum_{i,j \in \mathcal{V}} \sum_{i' \in \mathcal{V}}  \frac{C M_\gamma^2}{(|x_i - x_j| + h)^{d+2}} a_{i',i} |u_i| \pi_j
+ \sum_{i,j \in \mathcal{V}} \sum_{i' \in \mathcal{V}}  \frac{C M_\gamma^2}{(|x_i - x_j| + h)^{d+2}} a_{i',i} |u_j| \pi_j
\\
& \leq C M_\gamma^2(\delta x)^{-1} \sum_{i, i' : (i,i') \in \mathcal{E}}  a_{i',i} |u_i| \delta x \sum_{j \in \mathcal{V}} \frac{1}{(|x_i - x_j| + h)^{d+2}} \pi_j
\\
&+ C M_\gamma^2(\delta x)^{-1} \left( \sum_{i, i' : (i,i') \in \mathcal{E}}  \big( a_{i',i} \big)^q (\delta x)^{d-q(d-1)}  \right)^{1/q}
\\
&\quad\quad\quad\quad\quad\quad \left( \sum_{i, i' : (i,i') \in \mathcal{E}} \left(\sum_{j \in \mathcal{V}} \frac{1}{(|x_i - x_j| + h)^{d+2}} |u_j| \pi_j \right)^{q^*} (\delta x)^d \right)^{1/q^*}
\\
& \leq C \big( M_\gamma^2 / h^2 \delta x \big) \|(a_{i,j})_{(i,j)\in \mathcal{E}}\|_{L^q(\mathcal{F})} \|u\|_{L^{q^*}(\mathcal{C})}.
\end{aligned}
\end{equation*}

For $A_K^{(1)}$, one can apply the identity \eqref{eqn:linear_system_time_independent} to obtain
\begin{equation*}
\begin{aligned}
A_K^{(1)} &= \sum_{i,j \in \mathcal{V}} \sum_{i' \in \mathcal{V}} \nabla K^h(\widetilde{x}_i - \widetilde{x}_j) \cdot (\widetilde{x}_{i'} - \widetilde{x}_i) a_{i',i} |u_i - u_j| \pi_j
\\ 
&= \sum_{i,j \in \mathcal{V}} \left( \nabla K^h(\widetilde{x}_i - \widetilde{x}_j) \cdot \widetilde{b}_i \pi_i \right) |u_i - u_j| \pi_j
\\ 
&+ \sum_{i,j \in \mathcal{V}} \left( \nabla K^h(\widetilde{x}_i - \widetilde{x}_j) \cdot r_i \pi_i \right) |u_i - u_j| \pi_j
\\
&\eqdef A_K^{(1,1)} + A_K^{(1,2)}.
\end{aligned}
\end{equation*}
Repeating the argument on $A_K^{(2)}$, we bound the residue term $A_K^{(1,2)}$ by
\begin{equation*}
\begin{aligned}
\big|A_K^{(1,2)}\big| &= \left| \sum_{i,j \in \mathcal{V}} \left( \nabla K^h(\widetilde{x}_i - \widetilde{x}_j) \cdot r_i \pi_i \right) |u_i - u_j| \pi_j \right|
\\
&\leq \sum_{i,j \in \mathcal{V}} \frac{C}{(|x - y| + h)^{d+1}} r_i |u_i - u_j| \pi_i \pi_j
\\
&\leq C(1/h) \|r\|_{L^p(\mathcal{C})} \|u\|_{L^{p^*}(\mathcal{C})}.
\end{aligned}
\end{equation*}
Finally, symmetrize the expression of $A_K^{(1,1)}$ to obtain
\begin{equation*}
\begin{aligned}
A_K^{(1,1)} &= \sum_{i,j \in \mathcal{V}} \left( \nabla K^h(\widetilde{x}_i - \widetilde{x}_j) \cdot \widetilde{b}_i \pi_i \right) |u_i - u_j| \pi_j
\\
&= \frac{1}{2}\sum_{i,j \in \mathcal{V}} \nabla K^h(\widetilde{x}_i - \widetilde{x}_j) \cdot (\widetilde{b}_i - \widetilde{b}_j) |u_i - u_j| \pi_i \pi_j.
\end{aligned}
\end{equation*}
Choose measurable sets $(V_i)_{i \in \mathcal{V}} \subset \R^d$ and a piecewise constant extension $u^V \defeq \sum_{i \in \mathcal{V}} u_i \mathbbm{1}_{V_i}$ by Lemma~\ref{lem:kernel_equiv_extension_trick}. Those satisfy
\begin{equation*}
\begin{aligned}
|V_i| = \pi_i = \int_{\R^d} \chi_i, \quad \sup_{x \in V_i} |x - x_i| < 2\delta x, \quad \forall i \in \mathcal{V}, \quad  \quad V_{i} \cap V_{j} = \varnothing, \quad \forall i,j \in \mathcal{V}.
\end{aligned}
\end{equation*}
and
\begin{equation*}
\begin{aligned}
\supp u^V \subset \Big( \Omega + B(0,1) \Big) \subset \Big( \Omega + B(0,3) \Big) \subset \Big( \bigcup_{i \in \mathcal{V}} V_i \Big).
\end{aligned}
\end{equation*}
This leads us to introduce the continuous commutator term
\begin{equation*}
\begin{aligned}
A_K^{(1,1,1)} = \frac{1}{2} \int_{\R^{2d}} \nabla K^h(x - y) \cdot (\widetilde{b}(x) - \widetilde{b}(y))|u^V(x) - u^V(y)| \;\rd x\rd y.
\end{aligned}
\end{equation*}
Notice that $\supp u^V \subset \Omega + B(0,1)$ and $\supp \nabla K^h \in B(0,2)$. Then for $x \notin \Omega + B(0,3)$, either $y \notin \Omega + B(0,1)$, making $|u^V(x) - u^V(y)| = 0$, or $y \in \Omega + B(0,1)$, making $\nabla K^h(x - y) = 0$.
The same argument applies to $y$.
As a consequence, the integral formulating $A_K^{(1,1,1)}$ can be taken over any subset of $\R^{2d}$ including $\big( \Omega + B(0,3) \big)^2$. In particular,
\begin{equation*}
\begin{aligned}
A_K^{(1,1,1)} = \frac{1}{2} \int_{\left( \bigcup_{i\in \mathcal{V}} V_i \right)^2} \nabla K^h(x - y) \cdot (\widetilde{b}(x) - \widetilde{b}(y))|u^V(x) - u^V(y)| \;\rd x\rd y.
\end{aligned}
\end{equation*}
Combine the above discussion with  Lemma~\ref{lem:parallelogram} and assumption \eqref{eqn:drift_time_independent}, one has
\begin{equation*}
\begin{aligned}
2\Big|A_K^{(1,1,1)}-A_K^{(1,1)}\Big|&=\Bigg| \int_{\left( \bigcup_{i\in \mathcal{V}} V_i \right)^2} \nabla K^h(x - y) \cdot (\widetilde{b}(x) - \widetilde{b}(y))|u^V(x) - u^V(y)| \;\rd x\rd y
\\
&\quad\quad - \sum_{i,j \in \mathcal{V}} \nabla K^h(\widetilde{x}_i - \widetilde{x}_j) \cdot (\widetilde{b}_i - \widetilde{b}_j) |u_i - u_j| \pi_i \pi_j \Bigg|
\\
&\leq
\bigg| \sum_{i,j \in \mathcal{V}} 
\int_{V_i \times V_j} \Big( \nabla K^h(x - y) - \nabla K^h(\widetilde{x}_i - \widetilde{x}_j) \Big) \cdot (\widetilde{b}(x) - \widetilde{b}(y)) |u_i - u_j| \;\rd x\rd y \bigg|
\\
&\quad\quad +
\bigg| \sum_{i,j \in \mathcal{V}} 
\int_{V_i \times V_j} \nabla K^h(\widetilde{x}_i - \widetilde{x}_j) \cdot \big[(\widetilde{b}(x) - \widetilde{b}(y)) - (\widetilde{b}_i - \widetilde{b}_j)\big] |u_i - u_j| \;\rd x\rd y \bigg|
\\
&\leq
\sum_{i,j \in \mathcal{V}} \int_{V_i \times V_j} \frac{CM_\beta}{(|x - y| + h)^{d+2}} \cdot \big|\widetilde{b}(x) - \widetilde{b}(y)\big| \; |u_i - u_j| \;\rd x\rd y
\\
&\quad\quad + \sum_{i,j \in \mathcal{V}} \int_{V_i \times V_j} \frac{C}{(|x - y| + h)^{d+1}} \big|(\widetilde{b}(x) - \widetilde{b}(y)) - (\widetilde{b}_i - \widetilde{b}_j)\big| |u_i - u_j| \;\rd x\rd y
 \\
&\leq C \big(M_\beta / h^2 \big) \|\widetilde{b}\|_{L^q} \|u\|_{L^{q^*}(\mathcal{C})} + C(\delta x / h) \|\widetilde{b}\|_{W^{1,q}} \|u\|_{L^{q^*}(\mathcal{C})}.
\end{aligned}
\end{equation*}
Finally, the continuous commutator term $A_K^{(1,1,1)}$ can be estimated by Lemma~16 in \cite{BeJa:19}. The paper \cite{BeJa:19} also considered some non-linearity within the advection equation, which makes the formulations more complicated than what we need here. For the sake of completeness, we thus restate a simplified version of Lemma~16 in \cite{BeJa:19} with the notations of our paper.
\begin{lem} \textnormal{\textbf{(Lemma~16 in \cite{BeJa:19}, reformulated)}}
Assume that for some $1< q,r < \infty$, we have $u \in L^{q^*,1}$ and $b$ belonging to Besov space $B_{q,r}^1$. Then provided $\theta \geq 1 - 1 / r$,
\begin{equation*}
\begin{aligned}
&\quad \left| \int_{\R^{2d}} \nabla K^h (x - y) (b(x) - b(y)) \;\rd x \rd y \right|
\\
&\leq C |\log h|^{\theta} \Bigg( \|\dive b\|_{L^\infty} |\log h|^{-\theta} \int_{\R^{2d}} K^h(x-y) |u(x) - u(y)| \;\rd x \rd y
\;+\; \| \nabla b \|_{B_{q,r}^0} \|u\|_{L^{q^*,1}} \Bigg).
\end{aligned}
\end{equation*}

\end{lem}
We are going to apply this lemma by taking $b=\widetilde{b}$,  $u=u^V$ and $r = \max\{q,2\}$.

First we recall the classical bound $\|\nabla b \|_{B_{q,q \vee 2}^0} \leq C\|\nabla b \|_{L^q} \leq C\| b \|_{W^{1,q}}$. We also remark that $\widetilde{b}$ is defined on $\Omega$, but one can nevertheless extend it to $\R^d$ with $\| \widetilde{b} \|_{W^{1,q}(\R^d)} \leq C \| \widetilde{b} \|_{W^{1,q}(\Omega)}$.
We also recall that, since we consider a bounded domain $\Omega$ and assume $1 \leq p < q \leq \infty$, then o $\|u\|_{L^{q^*,1}} \leq C \|u\|_{L^{p^*}}$.

Therefore, for $\theta \geq 1 - 1/r = \max\{1 - 1/q, 1/2\}$, one has
\begin{equation*}
\begin{aligned}
&2A_K^{(1,1,1)} = \Bigg| \int_{\R^{2d}} \nabla K_h(x - y) \cdot (\widetilde{b}(x) - \widetilde{b}(y))|u^V(x) - u^V(y)| \;\rd x\rd y \Bigg|
\\
&\leq C |\log h|^{\theta} \Bigg( \|\dive \widetilde{b}\|_{L^\infty(\mathcal{C})} |\log h|^{-\theta} \int_{\R^{2d}} K^h(x-y) |u^V(x) - u^V(y)| \;\rd x \rd y
\;+\; \|\widetilde{b}\|_{W^{1,q}} \|u^V\|_{L^{p^*}}\Bigg).
\end{aligned}
\end{equation*}
The terms involving $u^V$ can be further bounded by the discrete density $(u_i)_{i \in \mathcal{V}}$. In particular, applying Lemma~\ref{lem:kernel_equiv_Euclidean} by choosing $u=v=u^V$, $f_k,g_k : \R^d \to \R^d$, $k = 1,2$ satisfying $f_1(x) = f_2(x) = \widetilde{x}_i$ for all $x \in V_i$, and $g_1(x) = g_2(x) = x$ for all $x \in \R^d$, one has
\begin{equation*}
\begin{aligned}
\int_{\R^{2d}} K^h(x-y) |u^V(x) - u^V(y)| \;\rd x \rd y 
\leq C \sum_{i,j \in \mathcal{V}} \widetilde{K}_{i,j}^h |u_i - u_j| \pi_i \pi_j \leq |\log h|^{\theta} \|u\|_{h_0,1,\theta;\widetilde{x}}.
\end{aligned}
\end{equation*}
This finally leads to the estimate
\begin{equation*}
\begin{aligned}
A_K^{(1,1)} 
&\leq C |\log h|^{\theta} \left( \|\dive \widetilde{b}\|_{L^\infty(\mathcal{C})} \|u\|_{h_0,1,\theta;\widetilde{x}} + \|\widetilde{b}\|_{W^{1,q}} \|u\|_{L^{p^*}(\mathcal{C})}\right)
\\
&\quad + C \big(M_\beta / h^2 \big) \|\widetilde{b}\|_{L^q} \|u\|_{L^{q^*}(\mathcal{C})} + C(\delta x / h) \|\widetilde{b}\|_{W^{1,q}} \|u\|_{L^{q^*}(\mathcal{C})}.
\end{aligned}
\end{equation*}
Combine the estimate for $A_K^{(1,1)}$, $A_K^{(1,2)}$ and $A_K^{(2)}$, we conclude \eqref{eqn:estimate_commutator_discrete}, which finishes the proof of Lemma~\ref{lem:linear_system}.
\end{proof}

We are now ready to conclude the proof of Theorem~\ref{thm:propagation_regularity_discrete}.
\begin{proof}[Proof of Theorem~\ref{thm:propagation_regularity_discrete}]
Let us first consider the case $m = 1$, i.e. the $(\widetilde{x}_i)_{i \in \mathcal{V}}$ are time-independent instead of just piecewise constant along time.

Then by Definition~\ref{defi:semi-norm_discrete} and Proposition~\ref{prop:scheme_Kruzkov}, one has that
\begin{equation*}
\begin{aligned}
&\|u(t)\|_{h_0,1,\theta;\widetilde{x}} \leq \|u(0)\|_{h_0,1,\theta;\widetilde{x}} + \int_0^t \sup_{h_0 \leq h \leq 1/2} |\log h|^{-\theta} \frac{\rd}{\rd s} \left( \sum_{i,j \in \mathcal{V}} \widetilde{K}_{i,j}^h |u_i(s) - u_j(s)| \pi_i \pi_j \right) \;\rd s
\\
&\quad \leq
\|u(0)\|_{h_0,1,\theta;\widetilde{x}} + \int_0^t \sup_{h_0 \leq h \leq 1/2} |\log h|^{-\theta} \Big(
A_K(s) + D_K(s) + R_K(s) \Big) \;\rd s
.
\end{aligned}
\end{equation*}
Substituting $A_K(s) + D_K(s) + R_K(s)$ by \eqref{eqn:estimate_residue_discrete}, \eqref{eqn:estimate_divergence_discrete} and \eqref{eqn:estimate_commutator_discrete}, and rearranging the terms, one deduces that
\begin{equation*}
\begin{aligned}
\|u(t)\|_{h_0,1,\theta;\widetilde{x}} &\leq \|u(0)\|_{h_0,1,\theta;\widetilde{x}}
\\ 
&+ C \int_0^t 
\bigg( \|\dive \widetilde{b}(s)\|_{L^\infty(\mathcal{C})} \|u(s)\|_{h_0,1,\theta;\widetilde{x}} + \|\widetilde{b}(s)\|_{W^{1,q}} \|u(s)\|_{L^{p^*}(\mathcal{C})}
\\ 
&\quad \quad \quad + \|D(s)\|_{L^\infty(\mathcal{C})} \|u(s)\|_{h_0,1,\theta;\widetilde{x}} + \|u(s)\|_{L^{p^*}(\mathcal{C})} \|D(s)\|_{h_0,p,p(\theta - 1/p^*);\widetilde{x}} \bigg) \;\rd s
\\
&
\begin{aligned}
+ \, C \int_0^t \bigg(\; &\big( |\log h_0|^{-\theta} M_\gamma^2 / h_0^2 \delta x \big) \|(a_{i,j}(s))_{i,j \in \mathcal{V}}\|_{L^q(\mathcal{F})} \|u(s)\|_{L^{q^*}(\mathcal{C})}
\\ 
 + &\big(|\log h_0|^{-\theta} M_\beta / h_0^2 \big) \|\widetilde{b}(s)\|_{L^q} \|u(s)\|_{L^{q^*}(\mathcal{C})}
\\
 + &(|\log h_0|^{-\theta} \delta x / h_0^2) \|\widetilde{b}(s)\|_{W^{1,q}} \|u(s)\|_{L^{q^*}(\mathcal{C})}
 \\
 +  &(|\log h_0|^{1-\theta})  \|(R_i(s))_{i \in \mathcal{V}}\|_{L^1(\mathcal{C})}  \bigg) \;\rd s
\end{aligned}
\\ 
&+ C \int_0^t  (|\log h_0|^{-\theta}/h_0) \|(r_i(s))_{i \in \mathcal{V}}\|_{L^p(\mathcal{C})} \|u(s)\|_{L^{p^*}(\mathcal{C})}
 \;\rd s.
\end{aligned}
\end{equation*}
We can change the $\|u(s)\|_{h_0,1,\theta;\widetilde{x}}$ norms into $\|u(s)\|_{h_0,1,\theta}$ norms through Proposition~\ref{prop:kernel_equiv} with $h_2=M_\beta$,
\begin{equation*}
\begin{aligned}
\|u(t)\|_{h_0,1,\theta} \leq &\; \alpha\, (L_0 + L_1 + L_2 + L_3)
\\
= & \big(1 + C(M_\beta / h_0) \big)^2\, \Bigg[ \|u(0)\|_{h_0,1,\theta}
\\ 
&+  C\int_0^t 
 \bigg( \|\dive \widetilde{b}(s)\|_{L^\infty(\mathcal{C})} \|u(s)\|_{h_0,1,\theta}  + \|\widetilde{b}(s)\|_{W^{1,q}} \|u(s)\|_{L^{p^*}(\mathcal{C})}
\\ 
&\quad \quad \quad \quad + \|D(s)\|_{L^\infty(\mathcal{C})} \|u(s)\|_{h_0,1,\theta} + \|u(s)\|_{L^{p^*}(\mathcal{C})} \|D(s)\|_{h_0,p,p(\theta - 1/p^*)} \bigg) \;\rd s
\\
&
\begin{aligned}
+ \, C \int_0^t \bigg(\; &\big( |\log h_0|^{-\theta} M_\gamma^2 / h_0^2 \delta x \big) \|(a_{i,j}(s))_{i,j \in \mathcal{V}}\|_{L^q(\mathcal{F})} \|u(s)\|_{L^{q^*}(\mathcal{C})}
\\ 
 + &\big(|\log h_0|^{-\theta} M_\beta / h_0^2 \big) \|\widetilde{b}(s)\|_{L^q} \|u(s)\|_{L^{q^*}(\mathcal{C})}
\\
 + &(|\log h_0|^{-\theta} \delta x / h_0^2) \|\widetilde{b}(s)\|_{W^{1,q}} \|u(s)\|_{L^{q^*}(\mathcal{C})}
 \\
 +  &(|\log h_0|^{1-\theta})  \|(R_i(s))_{i \in \mathcal{V}}\|_{L^1(\mathcal{C})}  \bigg) \;\rd s
\end{aligned}
\\ 
&+ C \int_0^t  (|\log h_0|^{-\theta}/h_0) \|(r_i(s))_{i \in \mathcal{V}}\|_{L^p(\mathcal{C})} \|u(s)\|_{L^{p^*}(\mathcal{C})}
 \;\rd s \quad\quad\quad\quad\quad\quad \Bigg].
\end{aligned}
\end{equation*}
Here we rewrite the last term of $L_2$ by
\begin{equation*}
\begin{aligned}
\int_0^t \|(R_i(s))_{i \in \mathcal{V}}\|_{L^1(\mathcal{C})} \;\rd s = \int_0^t \bigg( \sum_{i \in \mathcal{V}} -R_i(s) \pi_i \bigg) \;\rd s &= \int_0^t - \frac{\rd}{\rd s} \sum_{i \in \mathcal{V}} u_i(s) \pi_i \;\rd s
\\
&= \|u(0)\|_{L^1(\mathcal{C})} - \|u(t)\|_{L^1(\mathcal{C})}.
\end{aligned}
\end{equation*}
where we use $R_i \leq 0$ and identity \eqref{eqn:total_leaking}. The coefficient $\big(1 + C(M_\beta / h_0) \big)$ is multiplied twice because Proposition~\ref{prop:kernel_equiv} is actually applied to the left and right hand side separately. Since $k(t) = \min \{k: t < t_k\} = 1$ (as we assume $m = 1$) we have $k(t) + 1 = 2$, so all coefficients matches to \eqref{eqn:propagation_Gronwall}, which finishes the proof for the case $m=1$.

When $m > 1$, within each interval $[t_{k-1},t_k]$ we still have constant $(\widetilde{x}_i^{(k)})_{i \in \mathcal{V}}$ constant, and the semi-norm $\|u(t)\|_{h_0,1,\theta;\widetilde{x}^{(k)}}$ propagates exactly as above. 
However, at every endpoint $t_k$ one need to shift from $(\widetilde{x}_i^{(k)})_{i \in \mathcal{V}}$ to $(\widetilde{x}_i^{(k+1)})_{i \in \mathcal{V}}$, yielding an extra $C(M_\beta / h_0)$ factor.

Define
\begin{equation*}
\begin{aligned} 
L_1(\tau,t) &= C \int_{\tau}^t 
\bigg( \|\dive \widetilde{b}(s)\|_{L^\infty(\mathcal{C})} \|u(s)\|_{h_0,1,\theta} + \|\widetilde{b}(s)\|_{W^{1,q}} \|u(s)\|_{L^{p^*}(\mathcal{C})}
\\ 
&\quad \quad \quad \quad + \|D(s)\|_{L^\infty(\mathcal{C})} \|u(s)\|_{h_0,1,\theta} + \|u(s)\|_{L^{p^*}(\mathcal{C})} \|D(s)\|_{h_0,p,p(\theta - 1/p^*)} \bigg) \;\rd s
\\
L_2(\tau,t) &= \, C \int_{\tau}^t \bigg(\big( |\log h_0|^{-\theta} M_\gamma^2 / h_0^2 \delta x \big) \|(a_{i,j}(s))_{i,j \in \mathcal{V}}\|_{L^q(\mathcal{F})} \|u(s)\|_{L^{q^*}(\mathcal{C})}
\\ 
&\quad \quad \quad \quad + \big(|\log h_0|^{-\theta} M_\beta / h_0^2 \big) \|\widetilde{b}(s)\|_{L^q} \|u(s)\|_{L^{q^*}(\mathcal{C})}
\\
&\quad \quad \quad \quad + (|\log h_0|^{-\theta} \delta x / h_0^2) \|\widetilde{b}(s)\|_{W^{1,q}} \|u(s)\|_{L^{q^*}(\mathcal{C})}
 \\
&\quad \quad \quad \quad + (|\log h_0|^{1-\theta})  \|(R_i(s))_{i \in \mathcal{V}}\|_{L^1(\mathcal{C})}  \bigg) \;\rd s
\\ 
L_3(\tau,t) &= C \int_{\tau}^t  (|\log h_0|^{-\theta}/h_0) \|(r_i(s))_{i \in \mathcal{V}}\|_{L^p(\mathcal{C})} \|u(s)\|_{L^{p^*}(\mathcal{C})}
 \;\rd s.
\end{aligned}
\end{equation*}
We now argue by induction that
\begin{equation*}
\begin{aligned} 
\|u(t)\|_{h_0,1,\theta;\widetilde{x}^{(k)}} \leq \big(1 + C(M_\beta / h_0) \big)^{k(t)-1} (\|u(0)\|_{h_0,1,\theta;\widetilde{x}^{(1)}} + L_1(0,t) + L_2(0,t) + L_3(0,t))
\end{aligned}
\end{equation*}
by induction. The base case $k = 1$ was obtained as before, and for $k > 1$, one has
\begin{equation*}
\begin{aligned}
\|u(t)\|_{h_0,1,\theta;\widetilde{x}^{(k)}}
&\leq \|u(t_k)\|_{h_0,1,\theta;\widetilde{x}^{(k)}}
+ L_1(t_{k-1},t) + L_2(t_{k-1},t) + L_3(t_{k-1},t)
\\
&\leq
\big(1 + C(M_\beta / h_0) \big)  \|u(t_k)\|_{h_0,1,\theta;\widetilde{x}^{(k-1)}} + L_1(t_{k-1},t) + L_2(t_{k-1},t) + L_3(t_{k-1},t)
\\
&\leq \big(1 + C(M_\beta / h_0) \big)^{k(t)-1}(\|u(0)\|_{h_0,1,\theta;\widetilde{x}^{(1)}} + L_1(0,t) + L_2(0,t) + L_3(0,t)).
\end{aligned}
\end{equation*}
Finally, multiplying by $\big(1 + C(M_\beta / h_0) \big)$ twice more, we are able to replace the discrete semi-norm on both sides to $\|u(t)\|_{h_0,1,\theta}$ or $\|u(0)\|_{h_0,1,\theta}$. This gives
\begin{equation*}
\begin{aligned}
\|u(t)\|_{h_0,1,\theta} 
&\leq \big(1 + C(M_\beta / h_0) \big)^{k(t)+1}(\|u(0)\|_{h_0,1,\theta} + L_1 + L_2 + L_3)
\\
&= \alpha\,(L_0 + L_1 + L_2 + L_3),
\end{aligned}
\end{equation*}
which finishes the proof of Theorem~\ref{thm:propagation_regularity_discrete}.

\end{proof}
\section{Proof of Theorem~\ref{thm:residue_term_estimate}} \label{sec:residue_term_estimate}

This section is devoted to the proof of Theorem~\ref{thm:residue_term_estimate}.
We first note that all of our estimates  are on domains with bounded measure, which lets us immediately bound any $L^p$ or $W^{s,p}$ norms by $L^q$ or $W^{s,q}$ with any $q\geq p$.

Within this section, the generic constant $C$ that we use depends on $\Omega$, the exponents $p,q$ and the constant in the structural assumptions $\eqref{eqn:mesh_comparable_1}$, and also on $T$ and the exponent $s$ in the statement of Theorem~\ref{thm:residue_term_estimate}.

\medskip

$\bullet$ {\em Step 1: Constructing the space and time partitions.}
Choose $m\in \mathbb{N}^*$, and introduce the straightforward time partition 
\begin{equation*}
\begin{aligned}
&0 = t_{0} < t_{1} < \dots < t_{m} = T,
\\
&\tau = \frac{T}{m}, \quad t_l = (l/m)T \quad \forall l = 0,\dots,m.
\end{aligned}
\end{equation*}
A very small choice of time step $\tau$ will lead to large terms $\alpha$ in \eqref{eqn:propagation_Gronwall}, while a larger value of $\tau$ allows the velocity field to oscillate more in each time interval making controlling $L_3$ in \eqref{eqn:propagation_Gronwall} more difficult. Thus, determine the optimal choice of $\tau$ turns out to a key step of the proof.

the partition is more complicated in the spatial direction.
We divide the mesh into large partitions roughly corresponds to large hypercubes of size $\eta$ with $\eta \gg \delta x$ that will be determined later.
At this moment, it suffices to assume that we have for example $8 \, \delta x \leq \eta \ll 1$.

More precisely, we divide $\Omega$ into subdomains $\{\Omega_k\}_{k \in J}$, roughly centered around points $\{y_k\}_{k \in J} \in \R^d$ such that
\begin{equation*}
\begin{aligned}
B(y_k; \eta) \subset \Omega_k \subset B(y_k; C\eta) \quad \text{ and } \quad |\partial \Omega_k| \leq C \eta^{d-1}.
\end{aligned}
\end{equation*}
Next, choose a partition $\{\mathcal{V}_k\}_{k \in J}$ of the index set $\overline{\mathcal{V}}_\Omega$, by assigning $i \in \overline{\mathcal{V}}_\Omega$ to any $\mathcal{V}_k$ such that
\begin{equation*}
\begin{aligned}
\supp \chi_i \cap \Omega_k \neq \varnothing.
\end{aligned}
\end{equation*}
By definition $\supp \chi_i \cap \Omega \neq \varnothing$ and one can find at least one $\Omega_k$ intersecting $\supp \chi_i$ so that $\bigcup_{k\in J} \mathcal{V}_k=\overline{\mathcal{V}}_\Omega$.

Then, for all $k \in J$, define
\begin{equation*}
\begin{aligned}
\psi_k \defeq \sum_{i \in \mathcal{V}_k} \chi_{i},
\quad
U_k \defeq \supp \psi_k, \quad U_k' \defeq \{x \in \R^d : \psi_k(x) = 1\}.
\end{aligned}
\end{equation*}
By definition $U_k' \subset U_k$ and it is easy to verify that
\begin{equation*}
\begin{aligned}
U_k \subset \Omega_k + B_{\delta x}, \quad \Omega_k \subset U_k' + B_{\delta x}, \quad \text{ and } \quad
C^{-1} \eta^d \leq \|\psi_k\|_{L^1(\Omega)} \leq \|\psi_k\|_{L^1} \leq C \eta^d.
\end{aligned}
\end{equation*}
Moreover, for all $k \in J$, define the ``boundary'' of $\mathcal{V}_k$ as those indices that do not intersect with any index in another part of the domain or
\begin{equation*}
\begin{aligned}
\partial \mathcal{V}_k \defeq \mathcal{V}_k \setminus \{i \in \mathcal{V}_k\,|\; \text{ if } j \in \mathcal{V} \text{ and } \supp \chi_i \cap \supp \chi_j \neq \varnothing, \text{ then } \psi_k(\supp \chi_j) \equiv 1 \}.
\end{aligned}
\end{equation*}
Then this boundary has codimension 1 in the sense that by
decomposing $\psi_k = \psi_k^{(0)} + \psi_k^{(1)}$
\begin{equation*}
\begin{aligned}
\psi_k^{(0)} \defeq \sum_{i \in (\mathcal{V}_k \setminus \partial \mathcal{V}_k)} \chi_i, \quad \psi_k^{(1)} \defeq \sum_{i \in \partial \mathcal{V}_k} \chi_i, \quad U_k'' \defeq \{x \in \R^d : \psi_k^{(0)}(x) = 1\},
\end{aligned}
\end{equation*}
one has
\begin{equation*}
\begin{aligned}
\Omega_k \subset U_k'' + B_{4\delta x}, \quad (U_k \setminus U_k'') \subset \partial \Omega_k + B_{4\delta x}, \quad
\|\psi_k^{(1)}\|_{L^1} \leq C \eta^{d-1} \delta x.
\end{aligned}
\end{equation*}
Observe that the number $|J|$ of domains $\Omega_k$  can be estimated by $|J| \sim |\Omega| / \eta^d$. Hence the previous estimates yield
\begin{equation*}
\begin{aligned}
\bigg| \bigcup_{k \in J} (U_k \setminus U_k'') \bigg| \leq C (\delta x / \eta) |\Omega|, \quad
\bigg\|\sum_{k \in J} \psi_k^{(1)}\bigg\|_{L^1} \leq C (\delta x / \eta) |\Omega|.
\end{aligned}
\end{equation*}
For later discussion, we define $\partial \mathcal{V} = \bigcup_{k \in J}\partial \mathcal{V}_{k}$ as the set of all boundary indices, and $\partial \mathcal{E} = \{(i,j) \in \mathcal{E}: i \in \partial \mathcal{V} \text{ or } j \in \partial \mathcal{V}\}$ the set of all boundary edges.

\medskip

$\bullet$ {\em Step 2: Constructing the virtual coordinates in Theorem~\ref{thm:propagation_regularity_discrete}.}
Introduce
\begin{equation*}
\begin{aligned}
(\widetilde{b}_{i}(t))_{i \in \mathcal{V}} = P_{\mathcal{C}} \widetilde{b}(t), \quad (\widetilde{a}_{i,j}(t))_{i,j \in \mathcal{V}} = P_{\mathcal{F}} \widetilde{b}(t).
\end{aligned}
\end{equation*}
the discretization of $\widetilde{b}(t,x)$ on faces and cells as in \eqref{eqn:proj_to_cell} and \eqref{eqn:proj_to_edge} respectively.

We introduce another piecewise constant velocity field corresponding to the partition we just constructed:
For all $1 \leq l \leq m$ and all $k \in J$, we take the average of $\widetilde{b}(t,x)$ on $[t_{l-1},t_{l}] \times \mathcal{V}_k$ in the following sense
\begin{equation*}
\begin{aligned}
\bar b^{l,k} \defeq \frac{1}{|t_{l} - t_{l-1}| \, \|\psi_k\|_{L^1(\Omega)}}\int_{[t_{l-1},t_{l}] \times \Omega} \widetilde{b}(t,x) \psi_k(x) \;\rd t \rd x,
\end{aligned}
\end{equation*}
and define the ``piecewise'' extension $\bar b (t,x)$ by
\begin{equation*}
\begin{aligned}
\bar b(t,x) &\defeq \sum_{k \in J} \psi_k(x) \bar b^{l,k}, &&\quad \forall (t,x) \in [t_{l-1},t_{l}) \times \Omega.
\end{aligned}
\end{equation*}
We finally introduce as before the discretization on faces and cells of $\bar b(t,x)$,
\begin{equation*}
\begin{aligned}
(\bar b_{i}(t))_{i \in \mathcal{V}} = P_{\mathcal{C}} \bar b(t), \quad (\bar a_{i,j}(t))_{i,j \in \mathcal{V}} = P_{\mathcal{F}} \bar b(t).
\end{aligned}
\end{equation*}
The extension $\bar b (t,x)$ from $(\bar b^{l,k})_{1\leq l \leq m, k \in J}$ is not exactly piecewise. Nevertheless, by our construction in Step 1, each $\psi_k$ has compact support on $U_k$, with $\psi_k = 1$ on  the set $U_k'$, with $U_k\setminus U_k'$ small. In such sets $U_k'$, one has that $\bar b = \bar b^{l,k}$.
Moreover, for interior indices $i \in \mathcal{V}_k \setminus \partial \mathcal{V}_k$,
$\bar a_{i',i}(t)$ and $\bar b_{i}(t)$ are not only the discretization of $\bar b(t,\cdot)$, but also coincide with the discretization of the constant velocity field $\bar b^{l,k}$.

Theorem~\ref{thm:residue_term_estimate} assumes that Definition~\eqref{eqn:linear_system_constant_residue} applies for constant fields. This yields virtual coordinates $(\hat x_{i}(b_\text{c}))_{i \in \mathcal{V}}, b_\text{c} \in \R^d$,
\begin{equation} \label{eqn:linear_system_constant_asym}
\begin{aligned}
 \sum_{i' \in \mathcal{V}} \Big(\hat x_{i'}\big( \bar b^{l,k} \big) - \hat x_{i}\big( \bar b^{l,k} \big)\Big) \bar a_{i',i}(t) = \bar b_{i}(t) \pi_{i} + \hat r_{i}\big(\bar b^{l,k}\big) \pi_{i},
\quad
\forall t \in [t_{l-1}, t_{l}], i \in \mathcal{V}_k \setminus \partial \mathcal{V}_k.
\end{aligned}
\end{equation}
Inspired by \eqref{eqn:linear_system_constant_asym}, we choose piecewise constant in time virtual coordinates on the mesh $(\widetilde{x}_{i}(t))_{i \in \mathcal{V}}$ by
\begin{equation*} 
\begin{aligned}
\widetilde{x}_{i}(t) = \hat x_{i}\left( \bar b_{i}(t) \right), \quad \forall t \in [0,T], i \in \mathcal{V}.
\end{aligned}
\end{equation*}
Then in the interiors indices $i\in\overline{\mathcal{V}}_\Omega \setminus \partial \mathcal{V}$, \eqref{eqn:linear_system_constant_asym} implies that
\begin{equation} \label{eqn:linear_system_constant_asym_1}
\begin{aligned}
 \sum_{i' \in \mathcal{V}} (\widetilde{x}_{i'}(t) - \widetilde{x}_{i}(t)) \bar a_{i',i}(t) = \bar b_{i}(t) \pi_{i}
+ \hat r_{i}(\bar b_{i}(t)) \pi_{i},
\quad
\forall t \in [0,T], i \in \overline{\mathcal{V}}_\Omega \setminus \partial \mathcal{V}.
\end{aligned}
\end{equation}
By our construction, we can see that on each time interval $(t_{k-1},t_{k})$, the virtual coordinates $(\widetilde{x}_{i}(t))_{i \in \mathcal{V}}$ are time-independent and we may use the notation $(\widetilde{x}_{i}^{(k)})_{i \in \mathcal{V}}$ as in Theorem~\ref{thm:propagation_regularity_discrete}.

It is straightforward to deduce the uniform bounds
\begin{equation*}
\begin{aligned} 
\sup_{(i,i') \in \mathcal{E}, 1 \leq k \leq m}
\big|\widetilde{x}_{i}^{(k)} - \widetilde{x}_{i'}^{(k)}\big| &< 2M_\gamma,
\\ 
\sup_{i \in \mathcal{V}, 1 \leq k \leq m}
|\widetilde{x}_{i}^{(k)} - x_i| &< 2M_\beta,
\end{aligned}
\end{equation*}
because the $(\widetilde{x}_{i}^{(k)})_{i \in \mathcal{V}}$ are obtained through $(\hat x_{i}(b_\text{c}))_{i \in \mathcal{V}}$.
Therefore, those virtual coordinates satisfy the requirement of Theorem~\ref{thm:propagation_regularity_discrete}.

We reformulate the residue equation \eqref{eqn:linear_system_time_dependent} as
\begin{equation} \label{eqn:linear_system_time_dependent_asym}
\begin{aligned}
 \sum_{i' \in \mathcal{V}} (\widetilde{x}_{i'}(t) - \widetilde{x}_{i}(t)) a_{i',i}(t) = \widetilde{b}_{i}(t) \pi_{i} + r_{i}(t) \pi_{i},
\quad \forall t \in [0,T],\; i \in \mathcal{V}.
\end{aligned}
\end{equation}
Note that if $i \in \mathcal{V} \setminus \overline{\mathcal{V}}_\Omega$, one has $a_{i',i} \equiv 0$ for all $i' \in \mathcal{V}$ and $\widetilde{b}_i = (P_{\mathcal{F}}\widetilde{b})_{i} = 0$. Hence the residue $r_i$ vanishes for all $i \in \mathcal{V} \setminus \overline{\mathcal{V}}_\Omega$. 

Subtracting \eqref{eqn:linear_system_constant_asym_1} from \eqref{eqn:linear_system_time_dependent_asym}, we obtain
\begin{equation} \label{eqn:linear_system_difference}
\begin{aligned}
 \sum_{i' \in \mathcal{V}} (\widetilde{x}_{i'}(t) - \widetilde{x}_{i}(t)) \big(a_{i',i}(t) - \bar a_{i',i}(t) \big) = \big( \widetilde{b}_{i}(t) - \bar b_{i}(t) \big) \pi_{i}
+ \big( r_{i}(t) - \hat r_{i}(\bar b_{i}(t)) \big) \pi_{i},
\\
\forall t \in [0,T], i \in \overline{\mathcal{V}}_\Omega \setminus \partial \mathcal{V}.
\end{aligned}
\end{equation}
By definition, we have $|\hat r_{i}(\bar b_{i}(t))| \leq (\hat r_\text{max})_{i} |\bar b_{i}(t)|$ with
in addition, 
\begin{equation*}
\begin{aligned}
\|(\widetilde{a}_{i,j} - \bar a_{i,j})_{i,j \in \mathcal{V}}\|_{L^p([0,T] \times \mathcal{F})}
= \|P_{\mathcal{F}}(\widetilde{b} - \bar b)\|_{L^p([0,T] \times \mathcal{F})}
\leq
C\|\widetilde{b} - \bar b\|_{L^p([0,T] \times \Omega)},
\\
\|(\widetilde{b}_{i} - \bar b_{i})_{i \in \mathcal{V}}\|_{L^p([0,T] \times \mathcal{C})}
= \|P_{\mathcal{C}}(\widetilde{b} - \bar b)\|_{L^p([0,T] \times \mathcal{F})}
\leq
C\|\widetilde{b} - \bar b\|_{L^p([0,T] \times \Omega)}.
\end{aligned}
\end{equation*}
Therefore, the main obstacle to bound $\|r\|_{L^p([0,T] \times \mathcal{C})}$ is to derive good estimates on $\|\widetilde{b} - \bar b\|_{L^p}$.

\medskip

$\bullet$ {\em Step 3: Bounding $\|\widetilde{b} - \bar b\|_{L^p}$.}
We introduce the average in time of $b(t,x)$ by
\begin{equation*}
\begin{aligned}
\bar b^{l}(x) &= \dashint_{[t_{l-1},t_{l}]} \widetilde{b}(t,x) \;\rd t, \quad \forall x \in \Omega,
\\
\bar b'(t,x) &= \bar b^{l}(x), \quad \forall (t,x) \in [t_{l-1},t_{l}) \times \Omega.
\end{aligned}
\end{equation*}
It is obvious that the two ways of averaging of velocity field
\begin{equation*}
\begin{aligned}
\widetilde{b} \mapsto \bar b , \quad \widetilde{b} \mapsto \bar b'
\end{aligned}
\end{equation*}
and the discretizations $P_{\mathcal{C}}$, $P_{\mathcal{F}}$ are all linear mappings.

We can first quantify the oscillation in time by comparing $\widetilde{b}$ and $\bar b'$.
For fixed $x \in \R^d$, the function $\bar b'(\cdot,x)$ is constant on each time interval $[t_{l-1},t_{l})$, $1 \leq l \leq m$.
Therefore,
\begin{equation*}
\begin{aligned}
\int_{[0,T] \times \Omega} \left|\widetilde{b}(t,x) - \bar b'(t,x)\right|^p \;\rd x\rd t
&= \int_{[0,T] \times \Omega} \left|\widetilde{b}(t,x) - \dashint_{I(t)} \widetilde{b}(\tau,x) \rd \tau \right|^p \;\rd x\rd t
\\
&\leq \int_{[0,T] \times \Omega} \dashint_{I(t)} \left|\widetilde{b}(t,x) - \widetilde{b}(\tau,x) \right|^p \;\rd \tau \rd x\rd t,
\end{aligned}
\end{equation*}
where $I(t)$ denotes the interval $\tau\in[t_{l-1},t_{l})$.
  
Since $t_l-t_{l-1}=\tau$, we have that
\begin{equation*}
\begin{aligned}
\left( \int_{[0,T] \times \Omega} \left|\widetilde{b}(t,x) - \bar b'(t,x)\right|^p \;\rd x\rd t \right)^{1/p}
\leq C\tau \|\widetilde{b}\|_{L^p_x(W^{1,p}_t)([0,T] \times \Omega)},
\end{aligned}
\end{equation*}
Through interpolation, this shows that for $0 \leq s \leq 1$,
\begin{equation*} 
\begin{aligned}
\|\widetilde{b} - \bar b'\|_{L^p([0,T] \times \Omega)} \leq  C \tau^s \|\widetilde{b}\|_{L^{p}_x(W^{s,p}_t)([0,T] \times \Omega)}.
\end{aligned}
\end{equation*}
We can also bound spatial oscillations on $\bar b$ thanks to $\bar b'$. For any $t$, denote $l$ s.t. $t\in [t_{l-1}, t_l)$, and write
\begin{equation*}
\begin{aligned}
\bar b(t,x) = \sum_{k \in J} \psi_k(x) \bar b^{l,k} &= \sum_{k \in J} \psi_k(x) \frac{1}{|t_{l} - t_{l-1}| \, \|\psi_k\|_{L^1(\Omega)}}\int_{[t_{l-1},t_{l}] \times \Omega} \widetilde{b}(t,y) \psi_k(y) \;\rd t \rd y
\\
&= \sum_{k \in J} \psi_k(x) \frac{1}{\|\psi_k\|_{L^1(\Omega)}}\int_{\Omega} \bar b'(t,y) \psi_k(y) \; \rd y.
\end{aligned}
\end{equation*}
Moreover, since $\sum_{k \in J} \psi_k(x) = \sum_{i \in \overline{\mathcal{V}}_\Omega} \chi_i(x) = 1$ for all $x \in \Omega$, one has
\begin{equation*}
\begin{aligned}
\bar b'(t,x) - \bar b(t,x)
&= \sum_{k \in J} \psi_k(x) \bar b'(t,x) - \sum_{k \in J} \psi_k(x) \frac{1}{\|\psi_k\|_{L^1(\Omega)}}\int_{\Omega} \bar b'(t,y) \psi_k(y) \; \rd y
\\
&= \sum_{k \in J} \psi_k(x) \frac{1}{\|\psi_k\|_{L^1(\Omega)}}\int_{\Omega} \big[ \bar b'(t,x) - \bar b'(t,y) \big] \psi_k(y) \; \rd y.
\end{aligned}
\end{equation*}
Therefore,
\begin{equation*}
\begin{aligned}
&\quad \int_{[0,T] \times \Omega} \left|\bar b'(t,x) - \bar b(t,x)\right|^q \;\rd x\rd t
\\
&= \int_{[0,T] \times \Omega} \left| \sum_{k \in J} \psi_k(x) \frac{1}{\|\psi_k\|_{L^1(\Omega)}}\int_{\Omega} \big[ \bar b'(t,x) - \bar b'(t,y) \big] \psi_k(y) \; \rd y \right|^q \;\rd x\rd t
\\
&\leq \int_{[0,T] \times \Omega} \sum_{k \in J} \psi_k(x) \frac{1}{\|\psi_k\|_{L^1(\Omega)}}\int_{\Omega} \big| \bar b'(t,x) - \bar b'(t,y) \big|^q \psi_k(y) \; \rd y \rd x\rd t
\\
&\leq C \int_{[0,T]} \sum_{k \in J} \frac{1}{|U_k \cap \Omega|} \int_{U_k \cap \Omega}\int_{U_k \cap \Omega} \big| \bar b'(t,x) - \bar b'(t,y) \big|^q \;\rd y\rd x.
\end{aligned}
\end{equation*}
We recall that the last part of our assumption \eqref{eqn:mesh_comparable_1} states that $\big| \{k \in \mathcal{V} : (\supp \chi_{k}) \cap B(x;\delta x) \neq \varnothing\} \big| \leq C$ for all $x \in \R^d$. From their construction, any point $x\in \Omega$ also belongs to at most $C_\Omega$ sets $U_k$ so that
\begin{equation*}
\begin{aligned}
\int_{[0,T] \times \Omega} \left|\bar b'(t,x) - \bar b(t,x)\right|^q \;\rd x\rd t
\leq C \int_{[0,T]} \frac{1}{|B_{C \eta}|} \int_{\Omega^2} \mathbbm{1}_{B_{C \eta}}(x-y) \big| \bar b'(t,x) - \bar b'(t,y) \big|^q \;\rd y\rd x,
\end{aligned}
\end{equation*}
by which we conclude
\begin{equation*} 
\begin{aligned}
\|\bar b' - \bar b\|_{L^q([0,T] \times \Omega)} \leq C \eta \|\widetilde{b}\|_{L^{q}_t(W^{1,q}_x)([0,T] \times \Omega)}.
\end{aligned}
\end{equation*}
Finally for any $1 \leq p < q \leq \infty$, one obtains that
\begin{equation} \label{eqn:Poincare_interpolate_sum_7}
\begin{aligned}
\|\widetilde{b} - \bar b\|_{L^p([0,T] \times \Omega)} \leq C \Big( \tau^s \|\widetilde{b}\|_{L^{p}_x(W^{s,p}_t)([0,T] \times \Omega)}
+ \eta \|\widetilde{b}\|_{L^{q}_t(W^{1,q}_x)([0,T] \times \Omega)} \Big).
\end{aligned}
\end{equation}
As we mentioned in Step 2, one can also bound $(\widetilde{a}_{i,j} - \bar a_{i,j})_{i,j \in \mathcal{V}}$ and $(\widetilde{b}_{i} - \bar b_{i})_{i \in \mathcal{V}}$ by the right hand side of \eqref{eqn:Poincare_interpolate_sum_7}.

\medskip

$\bullet$ {\em Step 4: Optimizing all parameters.}
We finally combine all previous estimates to try to derive the best bound on the residue term $(r_{i}(t))_{i \in \mathcal{V}}$.

On the interior set $\overline{\mathcal{V}}_\Omega \setminus \partial \mathcal{V}$, by expanding \eqref{eqn:linear_system_difference} and using that $|\widetilde{x}_{i'}(t) - \widetilde{x}_{i}(t)| \leq 2M_\gamma$, we have that
\begin{equation*}
\begin{aligned}
\|r \, \mathbbm{1}_{\overline{\mathcal{V}}_\Omega \setminus \partial \mathcal{V}}\|_{L^p([0,T] \times \mathcal{C})}
&\leq C(M_\gamma / \delta x) \|(a_{i,j} - \bar a_{i,j})_{i,j \in \mathcal{V}}\|_{L^p([0,T] \times \mathcal{F})}
+ \|(\widetilde{b}_{i} - \bar b_{i})_{i \in \mathcal{V}}\|_{L^p([0,T] \times \mathcal{C})}
\\
&\quad + \|((\hat r_\text{max})_{i} |\bar b_{i}|)_{i \in \mathcal{V}}\|_{L^p([0,T] \times \mathcal{C})} 
\\
&\leq C(M_\gamma / \delta x) \|(a_{i,j} - \widetilde{a}_{i,j})_{i,j \in \mathcal{V}}\|_{L^p([0,T] \times \mathcal{F})}
\\
&\quad+ \; C(M_\gamma / \delta x) \|(\widetilde{a}_{i,j} - \bar a_{i,j})_{i,j \in \mathcal{V}}\|_{L^p([0,T] \times \mathcal{F})} + \|(\widetilde{b}_{i} - \bar b_{i})_{i \in \mathcal{V}}\|_{L^p([0,T] \times \mathcal{C})}
\\
&\quad+ \; C\|\hat r_\text{max}\|_{L^{(1/p-1/q)^{-1}}([0,T] \times \mathcal{C})} \|(\bar b_{i})_{i \in \mathcal{V}}\|_{L^{q}([0,T] \times \mathcal{C})}.
\end{aligned}
\end{equation*}
We recall that the admissible family of virtual coordinates has residue bound $M_\xi$ in $L^{1/(1/p - 1/q)}$. Applying \eqref{eqn:Poincare_interpolate_sum_7} and using $1 < (M_\gamma / \delta x)$ leads to
\begin{equation}
\begin{aligned}
\|r \, \mathbbm{1}_{\overline{\mathcal{V}}_\Omega \setminus \partial \mathcal{V}}\|_{L^p([0,T] \times \mathcal{C})} &\leq C(M_\gamma / \delta x) \| (a_{i,j})_{i,j \in \mathcal{V}} - P_{\mathcal{F}} \widetilde{b} \, \|_{L^p([0,T] \times \mathcal{F})}
\\
&\quad+  \; C(M_\gamma / \delta x) \Big( \tau^s \|\widetilde{b}\|_{L^{p}_x(W^{s,p}_t)([0,T] \times \Omega)} + \eta \|\widetilde{b}\|_{L^{q}_t(W^{1,q}_x)([0,T] \times \Omega)} \Big)
\\
&\quad+ \; C M_\xi \|\widetilde{b}\|_{L^{q}([0,T]\times \Omega)}.
\end{aligned}
\end{equation}
As for the boundary $\partial \mathcal{V}$, we directly expand \eqref{eqn:linear_system_time_dependent_asym} to find that
\begin{equation}
\begin{aligned}
&\|r \, \mathbbm{1}_{\partial \mathcal{V}}\|_{L^p([0,T] \times \mathcal{C})}
\\
&\quad\leq C(M_\gamma / \delta x) \|(a_{i,j})_{i,j \in \mathcal{V}} \, \mathbbm{1}_{\partial \mathcal{E}}\|_{L^p([0,T] \times \mathcal{F})}
+ \|(\widetilde{b}_{i})_{i \in \mathcal{V}} \, \mathbbm{1}_{\partial \mathcal{V}}\|_{L^p([0,T] \times \mathcal{C})}
\\
&\quad\leq C(M_\gamma / \delta x) \Big( (\delta x / \eta) T |\Omega| \Big)^{1/p-1/q} \big( \|(a_{i,j})_{i,j \in \mathcal{V}}\|_{L^{q}([0,T] \times \mathcal{F})} + \|\widetilde{b}\|_{L^{q}([0,T] \times \Omega)} \big),
\end{aligned}
\end{equation}
by H\"older inequality.

Since the residue $r_i$ vanishes for all $i \in \mathcal{V} \setminus \overline{\mathcal{V}}_\Omega$, we have
\begin{equation} \label{eqn:residue_bound}
\begin{aligned}
&\|r\|_{L^p([0,T] \times \mathcal{C})} \leq \|r \, \mathbbm{1}_{\overline{\mathcal{V}}_\Omega \setminus \partial \mathcal{V}}\|_{L^p([0,T] \times \mathcal{C})} + \|r \, \mathbbm{1}_{\partial \mathcal{V}}\|_{L^p([0,T] \times \mathcal{C})}
\\
&\quad\leq C(M_\gamma / \delta x) \| (a_{i,j})_{i,j \in \mathcal{V}} - P_{\mathcal{F}} \widetilde{b} \, \|_{L^p([0,T] \times \mathcal{F})}
\\
&\qquad+ C(M_\gamma / \delta x) \Bigg( \tau^s \|\widetilde{b}\|_{L^{p}_x(W^{s,p}_t)([0,T] \times \Omega)} +
\eta \|\widetilde{b}\|_{L^{q}_t(W^{1,q}_x)([0,T] \times \Omega)}
\\
&\quad \hphantom{\leq C(\delta x_{(n)})^{\gamma - 1} }+  (\delta x / \eta)^{1/p-1/q} \big( \|(a_{i,j}(t))_{i,j \in \mathcal{V}}\|_{L^{q}([0,T] \times \mathcal{F})} + \|\widetilde{b}\|_{L^{q}([0,T] \times \Omega)} \big)
\Bigg)
\\
&\qquad+ C M_\xi \|\widetilde{b}\|_{L^{q}([0,T]\times \Omega)},
\end{aligned}
\end{equation}
We are now ready to choose the parameters $\eta$ and $\tau$. We also need to control the term $\alpha = \big(1 + C(M_\beta / h_0) \big)^{k(t)+1}$ in Theorem~\ref{thm:propagation_regularity_discrete}, where $k(t)$ represents the number of times $(\widetilde{x}_{i}(t))_{i \in \mathcal{V}}$ jumps within $[0,t]$.
Notice that by increasing the constant $C$ we have $\alpha \leq \exp\big( Ck(t)(M_\beta / h_0) \big)$, and we can bound $k(t)$ by $k(T) \leq  CT/\tau$.
To control $L_3$ and $\alpha$ simultaneously, we use the following choice
\begin{equation*}
\begin{aligned}
\eta = (\delta x)^{\frac{1/p-1/q}{1+(1/p-1/q)}}, \quad \tau = (\delta x \, M_\beta / M_\gamma)^{1/(1+s)}.
\end{aligned}
\end{equation*}
It first results the claimed bound on $\alpha$ in \eqref{eqn:propagation_Gronwall_optimize}, namely
\[
\alpha = \exp\big( C(1/h_0) (M_\beta^s M_\gamma / \delta x)^{1/(1+s)} \big).
\]
This also yields a bound on the main residue term $L_3$ in Theorem~\ref{thm:propagation_regularity_discrete} by
\begin{equation*}
\begin{aligned}
L_3 &= C(|\log h_0|^{-\theta}/h_0) \; \|r\|_{L^1_t L^p_x([0,t] \times \mathcal{C})} \|u\|_{L^\infty_t L^{p^*}_x([0,t] \times \mathcal{C})}
\\
&\leq C(|\log h_0|^{-\theta}/h_0) \|r\|_{L^p_t L^p_x([0,t] \times \mathcal{C})} \|u\|_{L^\infty_t L^{p^*}_x([0,t] \times \mathcal{C})}.
\end{aligned}
\end{equation*}
Inserting  \eqref{eqn:residue_bound} on $\|r\|_{L^p([0,T] \times \mathcal{C})}$ finally provides
\begin{equation*}
\begin{aligned}
&L_3 \leq C\, 
(|\log h_0|^{-\theta}/h_0)
\Bigg[
(M_\gamma / \delta x) \big\| (a_{i,j})_{i,j \in \mathcal{V}} - P_{\mathcal{F}} \widetilde{b} \, \big\|_{L^p([0,T] \times \mathcal{F})}
\\
& \quad\quad\quad\quad
+(M_\beta^s M_\gamma / \delta x)^{1/(1+s)}  \|\widetilde{b}\|_{L^{p}_x(W^{s,p}_t)}
\\
& \quad\quad\quad\quad
+ \left( M_\gamma (\delta x)^{-\frac{1}{1+(1/p-1/q)}} \right) \bigg(
\|(a_{i,j}(t))_{i,j \in \mathcal{V}}\|_{L^{q}([0,T] \times \mathcal{F})}
+\|\widetilde{b}\|_{L^{q}_t(W^{1,q}_x)}
\bigg)
\\
& \quad\quad\quad\quad
+ M_\xi   \|\widetilde{b}\|_{L_t^{q}(L_x^{q})}
 \Bigg] \|u\|_{L^\infty_t L^{p^*}_x([0,t] \times \mathcal{C})}
,
\\
\end{aligned}
\end{equation*}
which finishes the proof of Theorem~\ref{thm:residue_term_estimate}.
%
%
\section{Proof of Theorem~\ref{thm:auxiliary_function_existence}} \label{sec:auxiliary_function_existence}

We first reduce, in subsection \ref{subsec:linear_system}, the infinite linear system \eqref{eqn:linear_system} to a finite linear system whose variables are $\{\hat x_i\}_{i \in \mathcal{V}_0}$, by making use of the periodic nature of the mesh.
Due to the geometric nature of the polygon meshes, 
both the matrix and the inhomogeneous term in the linear system have certain properties, which we focus on in subsections \ref{subsec:coefficients} and \ref{subsec:linear_lemmas}.
Finally, in subsection \ref{subsec:linear_uniform_boundedness} we conclude the uniform boundedness result.
\subsection{The linear system for periodic meshes} \label{subsec:linear_system}
We rewrite \eqref{eqn:linear_system} as
\begin{equation} \label{eqn:linear_system_1}
\begin{aligned}
\textstyle \left( \sum_{j \in \mathcal{V}} a_{j,i} (b_\text{c}) \right) \hat x_i (b_\text{c}) = \left( \sum_{j \in \mathcal{V}} a_{j,i} (b_\text{c}) \hat x_{j} (b_\text{c}) \right) - b_\text{c} \pi_i, \quad \forall i \in \mathcal{V}.
\end{aligned}
\end{equation}
Since the mesh is periodic as in Definition~\ref{defi:periodic_mesh}, we are looking for periodic solutions as well, namely solutions satisfying
\begin{equation} \label{eqn:def_periodic_solution}
\begin{aligned}
\hat x_{[m](i)}(b_\text{c}) = \hat x_{i}(b_\text{c}) + [m] L, \quad \forall b_\text{c} \in \R^d, [m] \in \Z^d, i \in \mathcal{V},
\end{aligned}
\end{equation}
where $[m] L = \sum_{k=1}^d m_{k} L_k \in \R^d$.

The following lemma makes explicit the finite linear systems that the variables $\hat x_i$, reduced to ${i \in \mathcal{V}_0}$, need to solve to be solutions to \eqref{eqn:linear_system_1} over the full mesh.
\begin{lem} \label{lem:periodici_solution}
Let  $(\mathcal{C}, \mathcal{F}) = \big( \{\chi_i\}_{i \in \mathcal{V}}, \{\bm{n}_{i,j}\}_{(i,j) \in \mathcal{E}} \big)$ be a periodic mesh as in Definition~\ref{defi:periodic_mesh} and $b \equiv b_\text{c}$ be a constant velocity field.
Consider a function $(\hat x_i)_{i \in \mathcal{V}}$ defined on all cells, satisfying \eqref{eqn:def_periodic_solution}.
It is a solution of \eqref{eqn:linear_system_1} if and only if its restriction $(\hat x_i)_{i \in \mathcal{V}_0}$ satisfying the following finite linear system: for all $i \in \mathcal{V}_0$,
\begin{equation} \label{eqn:linear_system_2}
\begin{aligned}
&\quad \textstyle \left( \sum_{j \in \mathcal{V}_0} \left( \sum_{[m] \in \Z^d} a_{[m](j),i}(b_\text{c}) \right) \right) \hat x_i(b_\text{c})
\\
&= \textstyle \left( \sum_{j \in \mathcal{V}_0} \left( \sum_{[m] \in \Z^d} a_{[m](j),i}(b_\text{c}) \right) \hat x_{j}(b_\text{c}) \right)
\\
&+ \textstyle \left( \sum_{j \in \mathcal{V}_0} \left( \sum_{[m] \in \Z^d} a_{[m](j),i}(b_\text{c}) [m]L \right) \right)
- b_\text{c} \pi_i.
\end{aligned}
\end{equation}
In the formulation above we let $a_{[m](j),i} = 0$ if $[m](j)$ lie outside the mesh $\mathcal{V}$.
\end{lem}
\begin{proof}
By Definition~\ref{defi:periodic_mesh}, for any $i \in \mathcal{V}_0$ and any $l \in \mathcal{V}$, such that $a_{l,i} \neq 0$, there exists unique $(j,[m]) \in \mathcal{V}_0 \times \Z^d$ such that $l = [m](j)$. Therefore \eqref{eqn:linear_system_1} is identical to 
\begin{equation*}
\begin{aligned}
\textstyle \left( \sum_{j \in \mathcal{V}_0} \sum_{[m] \in \Z^d} a_{j,i} (b_\text{c}) \right) \hat x_i (b_\text{c}) = \left( \sum_{j \in \mathcal{V}_0} \sum_{[m] \in \Z^d} a_{j,i} (b_\text{c}) \hat x_{j} (b_\text{c}) \right) - b_\text{c} \pi_i, \quad \forall i \in \mathcal{V}_0.
\end{aligned}
\end{equation*}
By the periodic condition \eqref{eqn:def_periodic_solution}, this is also identical to \eqref{eqn:linear_system_2}, which finishes the proof.
\end{proof}

We now introduce matrix notations on \eqref{eqn:linear_system_2} to simplify the discussion in later subsections,
\begin{equation*}
\begin{aligned}
\textstyle A(b_\text{c}) = (A_{ij}(b_\text{c}))_{i,j \in \mathcal{V}_0}, \quad A_{ij}(b_\text{c}) = \sum_{[m] \in \Z^d} a_{[m](i),j}(b_\text{c}).
\end{aligned}
\end{equation*}
Let $A^T(b_\text{c}) = (A_{ji}(b_\text{c}))_{i,j \in \mathcal{V}_0}$ denote the usual transpose matrix.
In addition, let
\begin{equation*} 
\begin{aligned}
\Phi(b_\text{c}) &= \textnormal{diag} \{\Phi_{ii}(b_\text{c})\}_{i \in \mathcal{V}_0}, \quad \textstyle \Phi_{ii}(b_\text{c})= \sum_{l \in \mathcal{V}_0} \sum_{[m] \in \Z^d} a_{[m](l),i}(b_\text{c}),
\\
\varphi(b_\text{c}) &= (\varphi_i(b_\text{c}))_{i \in \mathcal{V}_0}, \quad \textstyle \varphi_i(b_\text{c}) = \left( \sum_{j \in \mathcal{V}_0} \left( \sum_{[m] \in \Z^d} a_{[m](j),i} (b_\text{c}) [m]L \right) \right)
- b_\text{c} \pi_i.
\end{aligned}
\end{equation*}
Then the linear system \eqref{eqn:linear_system_2} can be rewritten as
\begin{equation*}
\begin{aligned}
\Phi_{ii}(b_\text{c}) \hat x_i (b_\text{c})
&= \textstyle \left( \sum_{j \in \mathcal{V}_0} A_{ji}(b_\text{c}) \hat x_{j}(b_\text{c}) \right) + \varphi_i(b_\text{c}), \quad \forall i \in \mathcal{V}_0,
\end{aligned}
\end{equation*}
or more compactly,
\begin{equation} \label{eqn:linear_system_3}
\begin{aligned}
(\Phi - A^T) (b_\text{c}) \hat x(b_\text{c}) =  \varphi(b_\text{c}),
\end{aligned}
\end{equation}
where $(\Phi - A^T)(b_\text{c})$ is a square matrix in $\R^{\mathcal{V}_0 \times \mathcal{V}_0}$ and $\hat x(b_\text{c}), \varphi(b_\text{c}) \in \R^{\mathcal{V}_0 \times d}$. We emphasize that \eqref{eqn:linear_system_3} each $d$ coordinate separately, namely it should be understood as
\[
(\Phi - A^T) (b_\text{c}) \hat x(b_\text{c})_\alpha =  \varphi(b_\text{c})_\alpha,\quad\alpha=1\dots d,
\]
with the same matrix $(\Phi - A^T) (b_\text{c})$ for each coordinate.

We may also omit the variable $b_\text{c} \in \R^d$ in the bracket when there is no ambiguity.
%
\subsection{Recasting \eqref{eqn:linear_system_3} into a discrete diffusion operator}  \label{subsec:coefficients}
We can characterize the matrix $(\Phi - A^T)$ in the following manner.
\begin{prop} \label{prop:diagonal_dominant}
Define the space of discrete diffusion operators as
\begin{equation*}
\begin{aligned}
\mathcal{M}(n) \defeq \left\{M \in \R^{n \times n} : M_{ii} \geq 0, M_{ij} \leq 0, \textstyle \sum_{l = 1}^n M_{il} = \sum_{l = 1}^n M_{li} = 0,
\; \forall i,j = 1,\dots,n, \, i \neq j
\right\}.
\end{aligned}
\end{equation*}
Then for all $b_\text{c} \in \R^d$,
$(\Phi - A^T)(b_\text{c}) \in \mathcal{M}(|\mathcal{V}_0|)$.
\end{prop}
\begin{proof}
The condition $M_{ii} \geq 0$ in the definition of $\mathcal{M}(n)$ is actually redundant.
It can be easily derived by the conditions $M_{ij} \leq 0$ and $\textstyle \sum_{l = 1}^n M_{il} = 0$.
Thus, if suffice to check that $(\Phi - A^T)$ satisfies the remaining properties in the definition.

Firstly, since $A^T$ is non-negative and $\Phi$ is diagonal, it is obvious that the non-diagonal entries of $(\Phi - A^T)$ are non-positive.

Secondly, the identity $\sum_{l \in \mathcal{V}_0} (\Phi - A^T)_{il} = 0$ can be derived by an expansion
\begin{equation*}
\begin{aligned}
\textstyle \sum_{l \in \mathcal{V}_0} (\Phi - A^T)_{il} 
&= \textstyle \Phi_{ii} - \sum_{l \in \mathcal{V}_0}
A_{li}
\\
&= \textstyle \sum_{l \in \mathcal{V}_0} \sum_{[m] \in \Z^d} a_{[m](l),i} - \sum_{l \in \mathcal{V}_0} \sum_{[m] \in \Z^d} a_{[m](l),i} = 0.
\end{aligned}
\end{equation*}
Now turning to the last property, as the velocity field is constant, we have that $a_{[m](i),l} = a_{i,[-m](l)}$.
By taking the same expansion as before, we have
\begin{equation*}
\begin{aligned}
\textstyle \sum_{l \in \mathcal{V}_0} (\Phi - A^T)_{li} 
&= \textstyle \sum_{l \in \mathcal{V}_0} \sum_{[m] \in \Z^d} a_{[m](l),i} - \sum_{l \in \mathcal{V}_0} \sum_{[m] \in \Z^d} a_{[m](i),l}
\\
&= \textstyle \sum_{l \in \mathcal{V}_0} \sum_{[m] \in \Z^d} a_{[m](l),i} - \sum_{l \in \mathcal{V}_0} \sum_{[m] \in \Z^d} a_{i,[-m](l)}.
\end{aligned}
\end{equation*}
Hence to prove $\sum_{l \in \mathcal{V}_0} (\Phi - A^T)_{li} = 0$, it suffices to show that 
\begin{equation*}
\begin{aligned}
\textstyle \sum_{l \in \mathcal{V}} a_{i,l} - \sum_{l \in \mathcal{V}} a_{l,i} = 0,
\end{aligned}
\end{equation*}
Since the constant velocity field is divergence-free,
applying the divergence theorem to the extended cell $C_i$, we have
\begin{equation*}
\begin{aligned}
{\textstyle \sum_{l \in \mathcal{V}} a_{i,l} - \sum_{l \in \mathcal{V}} a_{l,i}}
&= \sum_{l \in \mathcal{V}} \int \big( b_\text{c} \cdot \bm{n}_{l,i}(x) \big)^- \; \rd x - \sum_{l \in \mathcal{V}} \int \big( b_\text{c} \cdot \bm{n}_{l,i}(x) \big)^+ \; \rd x
\\
&= - \sum_{l \in \mathcal{V}} \int \big( b_\text{c} \cdot \bm{n}_{l,i}(x) \big) \; \rd x
\\
&= \int b_\text{c} \cdot \nabla \chi_i(x) \;\rd x = - \int_{C_i} (\dive b_\text{c}) \chi_i(x) \;\rd x = 0,
\end{aligned}
\end{equation*}
which finish the proof.
\end{proof}

Our next result is an inequality bounding the entries of $\varphi(b_\text{c})$ by the entries of $(\Phi - A^T)(b_\text{c})$, which still relies on the divergence theorem but in a more intricate way.

\begin{prop} \label{prop:vanishing_average} For all $b_\text{c} \in \R^d$ and $\mathcal{V}_1 \subseteq \mathcal{V}_0$, the inhomogeneous term $\varphi(b_\text{c})$ in the linear system \eqref{eqn:linear_system_3}
satisfies
\begin{equation*}
\begin{aligned}
\textstyle \big| \sum_{l \in \mathcal{V}_1} \varphi(b_\text{c}) \big| \leq C(\mathcal{V}_0) \sum_{i \in \mathcal{V}_1, j \in \mathcal{V}_0 \setminus \mathcal{V}_1} \big( |A_{ij}(b_\text{c})| + |A_{ji}(b_\text{c})| \big),
\end{aligned}
\end{equation*}
where 
\begin{equation*}
\begin{aligned}
C(\mathcal{V}_0) = \sup_{x,y \in \bigcup_{i \in \mathcal{V}_0} \supp \chi_i}|x - y|.
\end{aligned}
\end{equation*}

\end{prop}
\begin{proof} 

By choosing an appropriate basis of $\R^d$, we may assume $0 \in \bigcup_{i \in \mathcal{V}_0} \supp \chi_i$ and $b_\text{c} = (0,\dots,0,1)$ without loss of generality.
For $r = 1,\dots,(d-1)$, we then have
\begin{equation*} 
\begin{aligned}
(\varphi_{i})_r = \textstyle \sum_{l \in \mathcal{V}_0} \left( \sum_{[m] \in \Z^d} a_{[m](l),i} \sum_{p=1}^n m_p (L_p)_r \right), \quad i \in \mathcal{V}_0,
\end{aligned}
\end{equation*}
while for $k = d$, the equation reads
\begin{equation*} 
\begin{aligned}
(\varphi_{i})_d = \textstyle \sum_{l \in \mathcal{V}_0} \left( \sum_{[m] \in \Z^d} a_{[m](l),i} \sum_{p=1}^n m_p (L_p)_d \right) - \pi_i, \quad i \in \mathcal{V}_0.
\end{aligned}
\end{equation*}
For $r = 1,\dots,d-1$, notice that $\dive \big( x_r \; b_\text{c} \big) = \big( \partial_d x_r \big) = 0$, so that
\begin{equation*}
\begin{aligned}
- \int x_r \; b_\text{c} \cdot \nabla \left( {\textstyle\sum_{i \in \mathcal{V}_1}} \chi_i(x) \right) \;\rd x = \int \dive \big( x_r \; b_\text{c} \big) \cdot \left( {\textstyle\sum_{i \in \mathcal{V}_1}} \chi_i(x) \right) \;\rd x = 0.
\end{aligned}
\end{equation*}
For $k = d$, we have instead that $\dive \big( x_d \; b_\text{c} \big) = \big( \partial_d x_d \big) = 1$, leading
\begin{equation*}
\begin{aligned}
- \int x_d \; b_\text{c} \cdot \nabla \left( {\textstyle\sum_{i \in \mathcal{V}_1}} \chi_i(x) \right) \;\rd x = \int \dive \big( x_d \; b_\text{c} \big) \cdot \left( {\textstyle\sum_{i \in \mathcal{V}_1}} \chi_i(x) \right) \;\rd x = {\textstyle\sum_{i \in \mathcal{V}_1}} \pi_i.
\end{aligned}
\end{equation*}
We can summarize all those relation in vector form, as
\begin{equation*} 
\begin{aligned}
\varphi_{i} = \textstyle \sum_{l \in \mathcal{V}_0} \left( \sum_{[m] \in \Z^d} a_{[m](l),i} \sum_{p=1}^n m_p L_p \right) - b_\text{c} \pi_i, \quad i \in \mathcal{V}_0,
\end{aligned}
\end{equation*}
and
\begin{equation*}
\begin{aligned}
- \int x \; b_\text{c} \cdot \nabla \left( {\textstyle\sum_{i \in \mathcal{V}_1}} \chi_i(x) \right) \;\rd x =  {\textstyle\sum_{i \in \mathcal{V}_1}} b_\text{c} \; \pi_i.
\end{aligned}
\end{equation*}
For any $i,j$, denote
\begin{equation*}
\begin{aligned}
X_{j,i}^+ &\defeq \int x \big( b_\text{c} \cdot \bm{n}_{j,i}(x) \big)^+ \; \rd x,
\\
X_{j,i}^- &\defeq \int x \big( b_\text{c} \cdot \bm{n}_{j,i}(x) \big)^- \; \rd x,
\\
X_{j,i} &\defeq X_{j,i}^+ - X_{j,i}^- = \int x \big( b_\text{c} \cdot \bm{n}_{j,i}(x) \big) \; \rd x.
\end{aligned}
\end{equation*}
Notice that
$\bm{n}_{[m](j),i}(x) = \bm{n}_{j,[-m](i)}(x- {\textstyle\sum_{p=1}^n m_p L_p})$ by our periodic assumption, so that one obtains that
\begin{equation*} 
\begin{aligned}
&\quad\quad \left( X_{[m](j),i}^+ - X_{[-m](i),j}^- \right)
\\
&= \int x \big( b_\text{c} \cdot \bm{n}_{[m](j),i}(x) \big)^+ \; \rd x - \int x \big( b_\text{c} \cdot \bm{n}_{[-m](i),j}(x) \big)^- \; \rd x
\\
&= \int x \big( b_\text{c} \cdot \bm{n}_{[m](j),i}(x) \big)^+ \; \rd x - \int x \big( b_\text{c} \cdot \bm{n}_{j,[-m](i)}(x) \big)^+ \; \rd x
\\
&= \int x \big( b_\text{c} \cdot \bm{n}_{[m](j),i}(x) \big)^+ \; \rd x - \int (x- {\textstyle\sum_{p=1}^n m_p L_p})\big( b_\text{c} \cdot \bm{n}_{[m](j),i}(x) \big)^+ \; \rd x
\\
&= \int \big( b_\text{c} \cdot \bm{n}_{[m](j),i}(x) \big)^+ \; \rd x \;{\textstyle\sum_{p=1}^n m_p L_p}
\\
&= \textstyle a_{[m](j),i} \sum_{p=1}^n m_p L_p.
\end{aligned}
\end{equation*}
Hence $\varphi_{i}$ can be reformulated as
\begin{equation*}
\begin{aligned}
\varphi_{i} = \textstyle \sum_{j \in \mathcal{V}_0} \sum_{[m] \in \Z^d} \left( X_{[m](j),i}^+ - X_{[-m](i),j}^- \right) - b_\text{c} \pi_i, \quad i \in \mathcal{V}_0,
\end{aligned}
\end{equation*}
and the summation of $\varphi_{i}$ over any $\mathcal{V}_1\subset\mathcal{V}$ reads
\begin{equation*}
\begin{aligned}
\textstyle \sum_{i \in \mathcal{V}_1} \varphi_{i} = \sum_{i \in \mathcal{V}_1} \sum_{j \in \mathcal{V}_0} \sum_{[m] \in \Z^d} \left( X_{[m](j),i}^+ - X_{[-m](i),j}^- \right) - \sum_{i \in \mathcal{V}_1} b_\text{c} \pi_i.
\end{aligned}
\end{equation*}
Decompose
\begin{equation*} 
\begin{aligned}
&\quad \textstyle \sum_{i \in \mathcal{V}_1} \sum_{j \in \mathcal{V}_0} \sum_{[m] \in \Z^d} \left( X_{[m](j),i}^+ - X_{[-m](i),j}^- \right)
\\
&= \textstyle \sum_{i \in \mathcal{V}_1} \sum_{j \in \mathcal{V}_1} \sum_{[m] \in \Z^d} \left( X_{[m](j),i}^+ - X_{[-m](i),j}^- \right)
\\
&\quad + \textstyle \sum_{i \in \mathcal{V}_1} \sum_{j \in \mathcal{V}_0 \setminus \mathcal{V}_1} \sum_{[m] \in \Z^d} \left( X_{[m](j),i}^+ - X_{[-m](i),j}^- \right)
\\
&\eqdef Y_1 + Y_2.
\end{aligned}
\end{equation*}
The term $Y_1$ can be further simplified. By switching $i$ and $j$ and take the inverse $m$ to $-m$ in $\Z^d$, we have
\begin{equation*}
\begin{aligned}
Y_1 &= \textstyle \sum_{i \in \mathcal{V}_1} \sum_{j \in \mathcal{V}_1} \sum_{[m] \in \Z^d} X_{[m](j),i}^+ - \sum_{i \in \mathcal{V}_1} \sum_{j \in \mathcal{V}_1} \sum_{[m] \in \Z^d} X_{[-m](i),j}^-
\\
&= \textstyle \sum_{i \in \mathcal{V}_1} \sum_{j \in \mathcal{V}_1} \sum_{[m] \in \Z^d} X_{[m](j),i}.
\end{aligned}
\end{equation*}
Moreover, 
\begin{equation*}
\begin{aligned}
{\textstyle \sum_{i \in \mathcal{V}_1} b_\text{c} \pi_i } &= \int x \; b_\text{c} \cdot \left( {\textstyle - \sum_{i \in \mathcal{V}_1}} \nabla \chi_i(x) \right) \;\rd x
\\
&= \int x \; b_\text{c} \cdot \left( {\textstyle \sum_{i \in \mathcal{V}_1} \sum_{j \in \mathcal{V}_0} \sum_{[m] \in \Z^d} } \bm{n}_{[m](j),i}(x) \right) \;\rd x
\\
&= \textstyle \sum_{i \in \mathcal{V}_1} \sum_{j \in \mathcal{V}_0} \sum_{[m] \in \Z^d} X_{[m](j),i}
\\
&= \textstyle \sum_{i \in \mathcal{V}_1} \sum_{j \in \mathcal{V}_1} \sum_{[m] \in \Z^d} X_{[m](j),i}
+ \sum_{i \in \mathcal{V}_1} \sum_{j \in \mathcal{V}_0 \setminus \mathcal{V}_1} \sum_{[m] \in \Z^d} X_{[m](j),i}
\\
&= Y_1 + \left[ Y_2 + \textstyle \sum_{i \in \mathcal{V}_1} \sum_{j \in \mathcal{V}_0 \setminus \mathcal{V}_1} \sum_{[m] \in \Z^d} \left( X_{[m](j),i} - X_{[m](j),i}^+ + X_{[-m](i),j}^- \right) \right].
\end{aligned}
\end{equation*}
Therefore, if we define
\begin{equation*}
\begin{aligned}
Y_3 &\defeq \textstyle \sum_{i \in \mathcal{V}_1} \sum_{j \in \mathcal{V}_0 \setminus \mathcal{V}_1} \sum_{[m] \in \Z^d} \left( X_{[m](j),i} - X_{[m](j),i}^+ + X_{[-m](i),j}^- \right)
\\
&= \textstyle \sum_{i \in \mathcal{V}_1} \sum_{j \in \mathcal{V}_0 \setminus \mathcal{V}_1} \sum_{[m] \in \Z^d} \left( -X_{[m](j),i}^- + X_{[-m](i),j}^- \right)
\end{aligned}
\end{equation*}
then we obtain that
\begin{equation*}
\begin{aligned}
&\quad \textstyle \sum_{i \in \mathcal{V}_1} \varphi_i = (Y_1 + Y_2) - (Y_1 + Y_2 + Y_3) = - Y_3.
\end{aligned}
\end{equation*}
Since
\begin{equation*}
\begin{aligned}
\left| -X_{[m](j),i}^- + X_{[-m](i),j}^- \right| \leq C(\mathcal{V}_0) \left( a_{i, [m](j)} + a_{j,[-m](i)} \right) = C(\mathcal{V}_0) \left( a_{[-m](i),j} + a_{[m](j),i} \right),
\end{aligned}
\end{equation*}
by taking the summation over $i \in \mathcal{V}_1, j \in \mathcal{V}_0$ and $[m] \in \Z^d$, it is straightforward that
\begin{equation*}
\begin{aligned}
|Y_3| \leq C(\mathcal{V}_0) \sum_{i \in \mathcal{V}_1, j \in \mathcal{V}_0 \setminus \mathcal{V}_1} \big( |A_{ij}(b_\text{c})| + |A_{ji}(b_\text{c})| \big),
\end{aligned}
\end{equation*}
which completes the proof.
\end{proof}
\subsection{Some properties of discrete diffusion operators} \label{subsec:linear_lemmas}

To begin with, let us recall the definition of irreducibility.
A matrix is irreducible if it is not similar via a permutation of indices to a block upper triangular matrix with more than one block of strictly positive size.
An equivalent definition is the following: Each matrix $M$ can be associated to a directed graph $G$, with $n$ vertices labeled with $1,\dots,n$.
There is an edge from $i$ to $j$ in $G$ if and only if $M_{ij} \neq 0$.
Then $M$ is irreducible if and only if $G$ is strongly connected, i.e. one can reach any vertex starting from any vertex.

Let $I = \{1,\dots, n\}$, and consider any $J, \widetilde{J} \subseteq I$. We denote by the submatrix $M(J, \widetilde{J})$, the matrix obtained by deleting from $M$ all rows whose indexes are not in $J$ and all columns whose indexes are not in $\widetilde{J}$. More precisely, if $J=\{j_1,\ldots j_k\}$ and $\widetilde{J}=\{ \widetilde{j}_1,\ldots \widetilde{j}_{\widetilde{k}}\}$, $M(J, \widetilde{J})$ it is defined by
\begin{equation} \notag
\begin{aligned}
M(J, \widetilde{J})_{i_1,i_2} = M_{j_{i_1},\widetilde{j}_{i_2}}.
\end{aligned}
\end{equation}

The following lemma shows that a discrete diffusion operator, if it is not irreducible, must be block-diagonal up to a permutation.
\begin{lem} \label{lem:block_diagonal}
For $M \in \mathcal{M}(n)$,
there exists a decomposition $I_1 \cup \dots \cup I_m = I = \{1,\dots, n\}$, such that for all $k \in \{1,\dots,m\}$, the diagonal square submatrix $M(I_k,I_k)$ is irreducible and for all $k,l \in \{1,\dots,m\}, k \neq l$, the submatrix $M(I_k,I_l) = 0$. Moreover, the null space and image of $M$ are 
\begin{equation} \notag
\begin{aligned}
N(M) &= \spn \{ \mathbbm{1}_{I_k} \}_{k=1}^m,
\\
\Range M &= \textstyle \{x \in \R^I : \sum_{i \in I_k} x_i = 0, \forall 1 \leq k \leq m \}.
\end{aligned}
\end{equation}
\end{lem}

\begin{proof}
We are going to prove there exists a decomposition $I_1 \cup \dots \cup I_m = I = \{1,\dots, n\}$ by induction on the dimension $n$. When $n = 1$ there is nothing to prove.
Let's consider $n \geq 2$ and assume the decomposition exists for $1,\dots,(n-1)$. If $M$ is an irreducible $n \times n$ matrix, again there is nothing to prove. 
If $M$ is reducible, then it is similar to a block upper diagonal matrix via a permutation of indices, which means there exists $J \subseteq I = \{1, \dots, n\}$ s.t. $M_{ij} = 0$ for all $i \in I \setminus J$ and $j \in J$. Thus
\begin{equation} \notag
\begin{aligned}
0 = \sum_{i \in I} \sum_{j \in J} M_{ij} = \sum_{i \in J} \sum_{j \in J} M_{ij} + \sum_{i \in I \setminus J} \sum_{j \in J} M_{ij} = \sum_{i \in J} \sum_{j \in J} M_{ij}.
\end{aligned}
\end{equation}
On the other hand,
\begin{equation} \notag
\begin{aligned}
\sum_{i \in J} \sum_{j \in I} M_{ij} = 0.
\end{aligned}
\end{equation}
Hence
\begin{equation} \notag
\begin{aligned}
\sum_{i \in J} \sum_{j \in I \setminus J} M_{ij} = \sum_{i \in J} \sum_{j \in I} M_{ij} - \sum_{i \in J} \sum_{j \in J} M_{ij} = 0.
\end{aligned}
\end{equation}
Since the off-diagonal entries are all non-negative, one must have $M_{ij} = 0$ for all $i \in J$ and $j \in I \setminus J$. Therefore $M(I,I\setminus J) = 0$ and $M(I\setminus J, I) = 0$.

Furthermore, $M(J,J) \in \mathcal{M}(|J|)$ and $M(I\setminus J, I\setminus J) \in \mathcal{M}(|I \setminus J|)$. Note that $|J|, |I \setminus J| < n$.
Applying the induction argument on $M(J,J)$ and $M(I\setminus J, I\setminus J)$, we get decompositions $J = I_1^{(1)} \cup \dots \cup I_{m^{(1)}}^{(1)}$ and $I \setminus J = I_1^{(2)} \cup \dots \cup I_{m^{(2)}}^{(2)}$.
It is easy to verify that the decomposition
\begin{equation} \notag
\begin{aligned}
I = \left( I_1^{(1)} \cup \dots \cup I_{m^{(1)}}^{(1)} \right) \cup \left( I_1^{(2)} \cup \dots \cup I_{m^{(2)}}^{(2)} \right)
\end{aligned}
\end{equation}
satisfies the properties asserted in the lemma.

It remains to determine the null space and range of $M$. Assume first that $M$ is irreducible.
Let $x \in N(M)$ and
\begin{equation} \notag
\begin{aligned}
\textstyle J = \{i \in I: x_i = \max_l x_l \}.
\end{aligned}
\end{equation}
By construction, $J \neq \varnothing$ and we can argue by contradiction that $J = I$.
If $J \neq I$ and $J \neq \varnothing$, by the irreducibility of $M$, one can find $i \in J$, $j \in I \setminus J$ such that $M_{ij} > 0$.
However, for any $i$ and so in particular any $i \in J$,
\begin{equation} \notag
\begin{aligned}
\textstyle \sum_{j} M_{ij} x_j = 0.
\end{aligned}
\end{equation}
By our assumption of $M$, 
\begin{equation} \notag
\begin{aligned}
\textstyle \sum_{j \in I \setminus \{i\}} \left( M_{ij} x_j \right)-  \left( \sum_{j \in I \setminus \{i\}} M_{ij}\right) \max_l x_l = 0,
\end{aligned}
\end{equation}
which implies that $x_j = \max_l x_l$ if $M_{ij} > 0$.
Therefore, $j \in J$ whenever $i \in J$ and $M_{ij} > 0$, a contradiction. In conclusion, we have $J = I$ and $x_i = $ constant for all $i \in I$ or  $N(M) = \spn \{\mathbbm{1}\}$. As any column summation of $M$ is zero, for any $x \in \R^I$, one has
\begin{equation} \notag
\begin{aligned}
\textstyle \sum_{i} \left( \sum_{j} M_{ij} x_j \right) = \sum_{j} \left(\sum_{i} M_{ij}\right)x_j = 0.
\end{aligned}
\end{equation}
Since $\textnormal{codim} \Range M = \dim N(M) = 1$, one has $\Range M = \{x \in \R^I : \sum_{i \in I} x_i = 0\}$.
This concludes the proof when $M$ is irreducible. 

When $M$ is reducible, we know, up to a permutation, that $M$ is block diagonal and each diagonal block $M(I_k,I_k)$ is irreducible.
Applying the previous result to each $M(I_k,I_k)$ completes the proof.
\end{proof}

The next lemma allows us to estimate the $\ell^\infty$ norm of $M^{-1}$ when the non-negative entries of $M$ are bounded from both above and below.
\begin{lem} \label{lem:maximum_principle}
Define
\begin{equation} \notag
\begin{aligned}
\mathcal{M}(n,\eta_0,\eta_1) = \{M \in \mathcal{M}(n) : M_{ij} = 0 \textnormal{ or } \eta_0 < |M_{ij}|< \eta_1, \; \forall 1\leq i,j \leq n, i \neq j \; \}.
\end{aligned}
\end{equation}
Then there exists a constant $C(n,\eta_0,\eta_1) = C_2(n)\eta_1^{(n-2)} \eta_0^{-(n-1)}$ s.t. for any matrix $M \in \mathcal{M}(n,\eta_0,\eta_1)$ and $\varphi \in \Range M$, there exists one $\hat x$ satisfying $M \hat x = \varphi$ and $\| \hat x \|_{\ell^\infty} \leq C(n,\eta_0,\eta_1) \|\varphi\|_{\ell^\infty}$.
Moreover, for fixed $M \in \mathcal{M}(n,\eta_0,\eta_1)$ let us take the decomposition $I = I_1\cup \dots \cup I_m$ as in Lemma~\ref{lem:block_diagonal}.
The solution $\hat x$ above is uniquely determined by imposing the conditions
\begin{equation} \notag
\begin{aligned}
\textstyle \sum_{i \in I_k} \hat x_i = 0, \quad \forall k = 1, \dots, m.
\end{aligned}
\end{equation}
\end{lem}
\begin{proof}
Let us first assume that $M$ is irreducible.
Let $x$ be a solution of $Mx = \varphi$. 
Since $N(M) = \textnormal{span}\{\mathbbm{1}\}$, we have that $M(x - \lambda \mathbbm{1}) = \varphi$ for all $\lambda \in \R$.
By taking $\lambda = \frac{1}{n}\sum_{i} x_i$ and $\hat x = (x - \lambda \mathbbm{1})$ we have $M \hat x = \varphi$ and $\sum_{i} \hat x_i = 0$.
This solution is uniquely determined and it remains to show that there is a uniform bound $\| \hat x \|_{\ell^\infty} \leq C(n,\eta_0,\eta_1) \|\varphi\|_{\ell^\infty}$.

Define
\begin{equation} \notag
\begin{aligned}
J_0 &= \{1 \leq i \leq n : \hat x_i = \max_{1\leq j \leq n} \hat x_j \},
\\
J_k &= \{1 \leq i \leq n : \exists j \in J_{k-1} \text{ s.t. } |M_{ji}| > 0\} \cup J_{k-1} , \quad \forall k \geq 1, \vphantom{\max_{1\leq j \leq d} \hat x_j \},}
\\
D_k &= \max_{i \in J_k} \{(\max_{1 \leq j \leq n} \hat x_j) - \hat x_i \}.
\end{aligned}
\end{equation}
For $i \in J_k \setminus J_{k-1}$, take $j \in J_{k-1}$ s.t. $|M_{ji}| > 0$, the equality on the $j$-th entry now reads
\begin{equation} \notag
\begin{aligned}
\textstyle -|M_{ji}| \hat x_i + \left( \sum_{k \neq j} |M_{jk}| \right) \hat x_j - \left( \sum_{k \neq i,j} |M_{jk}| \hat x_k \right) = \varphi_{j}.
\end{aligned}
\end{equation}
By the fact that the summation of each row of $M$ is zero, one can rewritten the above equation as
\begin{equation} \notag
\begin{aligned}
\textstyle |M_{ji}| (\max_{l}\hat x_l - \hat x_i) = 
\left( \sum_{k \neq j} |M_{jk}| \right) (\max_l \hat x_l - \hat x_j)
+ \varphi_{j} - \left( \sum_{k \neq i,j} |M_{jk}| (\max_l \hat x_l - \hat x_k) \right),
\end{aligned}
\end{equation}
by which we have
\begin{equation} \notag
\begin{aligned}
D_k \leq \left[ 1 + (n - 2)\frac{\eta_1}{\eta_0} \right] D_{k-1} + \frac{1}{\eta_0}\|\varphi\|_{\ell^\infty}.
\end{aligned}
\end{equation}
By definition, $D_0 = 0$. Hence by induction we obtain
\begin{equation} \notag
\begin{aligned}
D_k \leq  \frac{\left[ 1 + (n - 2)\frac{\eta_1}{\eta_0} \right]^k - 1}{(n - 2) \eta_1}\|\varphi\|_{\ell^\infty}.
\end{aligned}
\end{equation}
By the irreducibility of $M$, unless $J_k = \{1 \leq i \leq n\}$ we have $J_{k} \subsetneq J_{k+1}$ (i.e. $J_{k}$ is a proper subset of $J_k$). This implies that $|J_{k+1}|\geq |J_k|+1$ and since $|J_0| \geq 1$, it proves that
that
$J_{n-1} = \{1 \leq i \leq n\}$.

Therefore, by taking $C(n,\eta_0,\eta_1)$ and $C_2(n)$ such that
\begin{equation} \notag
\begin{aligned}
C(n,\eta_0,\eta_1)  = C_2(n)\eta_1^{(n-2)} \eta_0^{-(n-1)} \geq D_{n-1},
\end{aligned}
\end{equation}
we have $(\max_l \hat x_l - \min_l \hat x_l) \leq C(n,\eta_0,\eta_1) \|\varphi\|_{\ell^\infty}$.
By $\sum_{i} \hat x_i = 0$ we have $\min_l \hat x_l \leq 0 \leq \max_l \hat x_l$. Hence
 $\| \hat x \|_{\ell^\infty} \leq C(n,\eta_0,\eta_1) \|\varphi\|_{\ell^\infty}$.

When $M$ is reducible, we know, up to a permutation, that $M$ is block diagonal and each diagonal block $M(I_k,I_k)$ is irreducible.
Applying the previous result to each $M(I_k,I_k)$ completes the proof.
\end{proof}
%
\subsection{Uniform boundedness} \label{subsec:linear_uniform_boundedness}
We are finally able to show our uniform boundedness result.
\begin{thm} \label{thm:linear_uniform_boundedness}
Consider $M \in \mathcal{M}(n)$ and $\varphi \in \R^n$ satisfying that $\exists C_0 > 0$, $\forall I' \subseteq I = \{1,\dots n\}$,
\begin{equation} \label{eqn:inhomogeneous_bound}
\begin{aligned}
\big| \textstyle \sum_{j \in I'} \varphi \big| \leq C_0 \sum_{j \in I', i \in I \setminus I'} \big( |M_{ij}| + |M_{ji}| \big).
\end{aligned}
\end{equation}
Then there exists $x \in \R^n$ satisfying
\begin{equation} \notag
\begin{aligned}
Mx = \varphi, \quad \text{and} \quad \|x\|_{\ell^\infty} \leq C_0 C_1(n),
\end{aligned}
\end{equation}
for some constant $C_1(n)$ depending only on the dimension $n$.
\end{thm}

Assuming for the moment that this theorem is correct. We can then immediately derive Theorem~\ref{thm:auxiliary_function_existence}.

\begin{proof}[Proof of Theorem~\ref{thm:auxiliary_function_existence}]
Combining Proposition~\ref{prop:diagonal_dominant}~and~\ref{prop:vanishing_average}, we see that the linear system \eqref{eqn:linear_system_3} satisfies the condition in Theorem~\ref{thm:linear_uniform_boundedness} with $n = |\mathcal{V}_0|$, $M = (\Phi - A^T)$, $C_0 = C(\mathcal{V}_0)$.
Moreover, since we assume $\bigcup_{i \in \mathcal{V}_0} \supp \chi_i$ being connected in Definition~\ref{defi:periodic_mesh}, it is easy to verify that for some constant $C$
\begin{equation*}
\begin{aligned}
C(\mathcal{V}_0) = \sup_{x,y \in \bigcup_{i \in \mathcal{V}_0} \supp \chi_i}|x - y| \leq C\,|\mathcal{V}_0| \delta x,
\end{aligned}
\end{equation*}
where $\delta x$ is the discretization size of the mesh.

Hence by applying Theorem~\ref{thm:linear_uniform_boundedness} to \eqref{eqn:linear_system_3}, we conclude that there are solutions $\hat x(b_\text{c})$ for all possible $b_\text{c}$, with $\ell^\infty$ norm uniformly bounded by
\begin{equation*}
\begin{aligned}
\|\hat x\|_{\ell^\infty} \leq C(\mathcal{V}_0)C_1(|\mathcal{V}_0|) \leq C(|\mathcal{V}_0|) \delta x,
\end{aligned}
\end{equation*}
where $C(|\mathcal{V}_0|)$ depends only on $|\mathcal{V}_0|,$ the number of cell functions in a period.
We can have the solutions satisfy that $\forall \lambda > 0 , \hat x_{i}(b_\text{c}) = \hat x_{i}(\lambda b_\text{c})$ simply by redefining $\hat x_{i}(b_\text{c}) = \hat x_{i}(b_\text{c}/|b_\text{c}|)$ for all $b_\text{c}$ s.t. $|b_\text{c}| \neq 0$.
Finally, by Lemma~\ref{lem:periodici_solution} we can extend our solution $\hat x(b_\text{c})$ to a periodic solution on the entire mesh, satisfying all properties claimed in Theorem~\ref{thm:auxiliary_function_existence}.
\end{proof}

We now conclude the section with the proof of Theorem~\ref{thm:linear_uniform_boundedness}:

\begin{proof}[Proof of Theorem~\ref{thm:linear_uniform_boundedness}]
Let us begin with the solvability of $Mx = \varphi$.
Take the decomposition $I = I_1\cup \dots \cup I_m$ as in Lemma~\ref{lem:block_diagonal}.
By Lemma~\ref{lem:block_diagonal},
\begin{equation*}
\begin{aligned}
\Range M &= \textstyle \{x \in \R^I : \sum_{i \in I_k} x_i = 0, \forall 1 \leq k \leq m \}.
\end{aligned}
\end{equation*}
For all $k = 1,\dots,m$, we have $M(I_k, I \setminus I_k) = M(I \setminus I_k, I_k) = 0$.
By \eqref{eqn:inhomogeneous_bound} this implies that
$\sum_{I_k} \varphi = 0$, $\forall k = 1,\dots, m$. Therefore $\varphi \in \Range M$ and $Mx = \varphi$ is solvable.

We now turn to the proof of $\ell^\infty$ bound of $x$.
We can WLOG assume $C_0 = 1$ and $\max_{i \neq j} |M_{ij}| = \eta_1 = 1$, because the general case can be reduced to it by a scaling $(\eta_1^{-1} M) (C_0^{-1} x) = (\eta_1^{-1} C_0^{-1} \varphi)$.

The result is immediate when $n = 1$ and 
we prove the cases $n \geq 2$ by induction. Assume the theorem holds for any $p \times p$ matrices with $p \leq (n-1)$, we are going to show that the theorem holds for $n \times n$ matrices.

First, since
\begin{equation} \notag
\begin{aligned}
\mathcal{M}_0(n) \defeq \mathcal{M}(n) \cap \{M \in \R^{n \times n}: \max_{i \neq j} |M_{ij}| = 1\}
\end{aligned}
\end{equation}
is compact, it suffices to show that there is a local bound.
More explicitly, we are going to show that for any $M^{(0)} \in \mathcal{M}_0(n)$, there is an open neighborhood $U \ni M^{(0)}$, s.t. for all $M \in U \cap \mathcal{M}_0(n)$ and $\varphi$ satisfying \eqref{eqn:inhomogeneous_bound} with $C_0 = 1$, there is a constant $C = C(n, U \cap \mathcal{M}_0(n))$ s.t.
one can take $x \in \R^n$ satisfying $Mx = \varphi$ and $\|x\|_{\ell^\infty} \leq C(n, U \cap \mathcal{M}_0(n))$.
Then by compactness, we conclude immediately that there is a uniform bound $C_1(n) = C(n, \mathcal{M}_0(n))$.

For arbitrary $M^{(0)} \in \mathcal{M}_0(n)$, introduce the irreducible decomposition $I = I_1 \cup \dots \cup I_m$ as in Lemma~\ref{lem:block_diagonal}.
For $1\leq k \leq m$, define $E_k = \textnormal{span}\{\mathbbm{1}_{\{i\}}\}_{i \in I_k}$, $F_k = \{x \in E_k : \sum_i x_i = 0\}$.
In addition, define $E_0 = \textnormal{span}\{\mathbbm{1}_{I_k}\}_{k = 1}^{m}$, $F_0 = \{x \in E_0 : \sum_i x_i = 0\}$. Note that we have
\begin{equation} \notag
\begin{aligned}
\R^n = (E_1 \oplus \dots \oplus E_{m}) = (F_1 \oplus \dots \oplus F_{m}) \oplus E_0 = (F_1 \oplus \dots \oplus F_{m}) \oplus F_0 \oplus \textnormal{span}\{\mathbbm{1}_{I}\}.
\end{aligned}
\end{equation}
Since there are non-zero non-diagonal entries in $M^{(0)}$, there exists $k$ such that $|I_k| > 1$. Hence $m < n$.

Let
\begin{equation} \notag
\begin{aligned}
\eta_0 &= \min \big\{ |M_{ij}^{(0)}| : i \neq j, |M_{ij}^{(0)}| \neq 0\big\},
\\
r &= \frac{1}{2nC(n,\eta_0,1)(1 + 2\max_{1 \leq p < n}C_1(p))},
\end{aligned}
\end{equation}
where the constant $C(n,\eta_0,1)$ is as in Lemma~\ref{lem:maximum_principle}.

Define the open set
\begin{equation} \notag
\begin{aligned}
U_r(M^{(0)}) = \left\{M \in \R^{n \times n} : \max_{1 \leq i,j \leq n} |M - M^{(0)}| < r \right\}.
\end{aligned}
\end{equation}
Define also the following linear mappings,
\begin{equation} \notag
\begin{aligned}
P: \R^n &\to E_0 = \textnormal{span}\{\mathbbm{1}_{I_k}\}_{k = 1}^{m} \subseteq \R^n
\\
x &\mapsto \sum_{k = 1}^{m} \left( \frac{1}{|I_k|} \sum_{i \in I_k} x_i \right) \mathbbm{1}_{I_k},
\end{aligned}
\end{equation}
i.e. we take the average on each $I_k$,
\begin{equation} \notag
\begin{aligned}
Q: \R^{m} &\to E_0 = \textnormal{span}\{\mathbbm{1}_{I_k}\}_{k = 1}^{m} \subseteq \R^n
\\
y &\mapsto \sum_{k=1}^{m} y_k \mathbbm{1}_{I_k},
\end{aligned}
\end{equation}
i.e. we project on the canonical basis of $E_0$, and finally
\begin{equation} \notag
\begin{aligned}
W: \R^{m} &\to \R^{m}
\\
(y_1,\dots, y_{m}) &\mapsto (|I_1|y_1,\dots, |I_{m}|y_{m}).
\end{aligned}
\end{equation}
For $x \in \R^n$ solving $Mx = \varphi$, introduce the decomposition
\begin{equation} \notag
\begin{aligned}
x = x_1 + x_2 + \lambda \mathbbm{1}_I, \quad x_1 \in (F_1 \oplus \dots \oplus F_{m}), \; x_2 \in F_0.
\end{aligned}
\end{equation}
Without loss of generality, we can assume $\lambda = 0$. We are going to discussion two situations.

\medskip

\noindent \emph{Case 1:} $\textstyle \|x_2\|_{\ell^\infty} \geq \big(2\max_{1 \leq p < n}C_1(p)\big) \|x_1\|_{\ell^\infty}$. Since $F_0\subset E_0$, $x_2$ belongs to the image of $Q$ and as $Q$ is trivially one-to-one from $\R^m$ to $E_0$, we can define $y_2 = Q^{-1} (x_2) \in \R^{m}$.

Since $M(x_1 + x_2) = \varphi$, we have
\begin{equation} \label{eqn:inhomogeneous_bound_induction}
\begin{aligned}
(WQ^{-1} PM Q) (y_2) = (WQ^{-1} PM) (x_2) = (WQ^{-1} P) (\varphi) - (WQ^{-1} PM) (x_1).
\end{aligned}
\end{equation}
Since
\begin{equation} \notag
\begin{aligned}
\textstyle (WQ^{-1} PM Q)_{kl} = \sum_{i \in I_k}\sum_{j \in I_l} M_{ij},
\end{aligned}
\end{equation}
it is easy to verify that $(WQ^{-1} PM Q) \in \mathcal{M}(m)$.

Our goal is to apply the induction argument to the $m \times m$ linear system \eqref{eqn:inhomogeneous_bound_induction}.
To apply the induction argument, we need to provide proper estimates on the  terms in the right-hand side.
For any $J' \subseteq J = \{1,\dots,m\}$, we have
\begin{equation} \notag
\begin{aligned}
& \textstyle \quad \sum_{k \in J'}\big[ (WQ^{-1} P) (\varphi) - (WQ^{-1} PM) (x_1) \big]_{k}
\\
&=
\textstyle \sum_{k \in J'}\big[ (WQ^{-1} P) (\varphi) \big]_{k} - \textstyle \sum_{k \in J'}\big[ (WQ^{-1} PM) (x_1) \big]_{k}
\\
&\eqdef L_1 + L_2.
\end{aligned}
\end{equation}
By our assumption on $\varphi$,
\begin{equation} \notag
\begin{aligned}
\textstyle |L_1| = \big| \sum_{k \in J'} \sum_{i \in I_k} \varphi_{i} \big| & \textstyle \leq \sum_{k \in J', l \in J \setminus J'} \sum_{i \in I_k, j \in I_l} \big( |M_{ij}| + |M_{ji}| \big)
\\
&= \textstyle \sum_{k \in J', l \in J \setminus J'} \left( \big| (WQ^{-1} PM Q)_{kl} \big| + \big| (WQ^{-1} PM Q)_{lk} \big| \right)
.
\end{aligned}
\end{equation}
On the other hand,
\begin{equation} \notag
\begin{aligned}
 |L_2| &= \textstyle \big| \sum_{k \in J'} \sum_{i \in I_k} \sum_{j \in I} M_{ij}(x_1)_j \big| = \big| \sum_{j \in I} \sum_{k \in J'} \sum_{i \in I_k} M_{ij}(x_1)_j \big|
\\
&\leq \textstyle \sum_{l \in J' }\sum_{j \in I_l} \big| \sum_{k \in J'} \sum_{i \in I_k} M_{ij}(x_1)_j \big| + \sum_{l \in J \setminus J' }\sum_{j \in I_l} \big| \sum_{k \in J'} \sum_{i \in I_k} M_{ij}(x_1)_j \big|
\\
&= \textstyle \sum_{l \in J' }\sum_{j \in I_l} \big| \sum_{k \in J \setminus J'} \sum_{i \in I_k} M_{ij}(x_1)_j \big| + \sum_{l \in J \setminus J' }\sum_{j \in I_l} \big| \sum_{k \in J'} \sum_{i \in I_k} M_{ij}(x_1)_j \big|
\\
&\leq \textstyle \sum_{k \in J', l \in J \setminus J'} \sum_{i \in I_k, j \in I_l} \big( |M_{ij}| + |M_{ji}| \big) \|x_1\|_{\ell^\infty}
\\
&= \textstyle \sum_{k \in J', l \in J \setminus J'} \left( \big| (WQ^{-1} PM Q)_{kl} \big| + \big| (WQ^{-1} PM Q)_{lk} \big| \right) \|x_1\|_{\ell^\infty}
.
\end{aligned}
\end{equation}
In conclusion we have
\begin{equation} \notag
\begin{aligned}
 |L_1 + L_2| \leq \textstyle (1 + \|x_1\|_{\ell^\infty}) \sum_{k \in J', l \in J \setminus J'} \left( \big| (WQ^{-1} PM Q)_{kl} \big| + \big| (WQ^{-1} PM Q)_{lk} \big| \right),
\end{aligned}
\end{equation}
which satisfies necessary assumption for the induction argument with $C_0 = (1 + \|x_1\|_{\ell^\infty})$. As a consequence, we have
\begin{equation} \notag
\begin{aligned}
\|x_2\|_{\ell^\infty} = \|y_2\|_{\ell^\infty} \leq (1 + \|x_1\|_{\ell^\infty}) C_1(m) \leq C_1(m) + \|x_2\|_{\ell^\infty}/2,
\end{aligned}
\end{equation}
by using the relation between $\|x_1\|_{\ell^\infty}$ and $\|x_2\|_{\ell^\infty}$ assumed at the beginning of the case.

Hence, as claimed
\begin{equation} \notag
\begin{aligned}
\|x_2\|_{\ell^\infty} &\leq 2 C_1(m) \leq 2\max_{1 \leq p < n}C_1(p),
\\
\|x\|_{\ell^\infty} &\leq \|x_1\|_{\ell^\infty} + \|x_2\|_{\ell^\infty} \leq 1 + 2\max_{1 \leq p < n}C_1(p).
\end{aligned}
\end{equation}

\medskip

\noindent \emph{Case 2:} $\textstyle \|x_2\|_{\ell^\infty} < \big(2\max_{1 \leq p < n}C_1(p)\big) \|x_1\|_{\ell^\infty}$. Since $M(x_1 + x_2) = \varphi$, we have
\begin{equation} \notag
\begin{aligned}
M^{(0)} (x_1) = M^{(0)} (x_1 + x_2) &= M (x_1 + x_2) + (M^{(0)} - M) (x_1 + x_2)
\\
&= \varphi + (M^{(0)} - M) (x_1 + x_2).
\end{aligned}
\end{equation}
On each $I_k \times I_k$ block ($k = 1,\dots m''$) we can apply Lemma~\ref{lem:maximum_principle}.
For $M \in U_r(M^{(0)}) \cap \mathcal{M}(n)$, this gives
\begin{equation} \notag
\begin{aligned}
\textstyle \|x_1\|_{\ell^\infty} &\leq C(n, \eta_0, 1) \big( \|\varphi\|_{\ell^\infty} + rd(1 + 2\max_{1 \leq p < n}C_1(p))\|x_1\|_{\ell^\infty} \big)
\\
&= C(n, \eta_0, 1) \|\varphi\|_{\ell^\infty} + \|x_1\|_{\ell^\infty}/2.
\end{aligned}
\end{equation}
By \eqref{eqn:inhomogeneous_bound}, for all $i \in I$,
\begin{equation} \notag
\begin{aligned}
|\varphi_i| \leq \sum_{j \in I \setminus \{i\}} \big( |M_{ij}| + |M_{ji}| \big) \leq 2(n-1).
\end{aligned}
\end{equation}
Hence
\begin{equation} \notag
\begin{aligned}
\|x_1\|_{\ell^\infty} &\leq 4(n-1) C(n,\eta_0,1),
\\
\|x\|_{\ell^\infty} &\leq \|x_1\|_{\ell^\infty} + \|x_2\|_{\ell^\infty} \leq 4(n-1) C(n,\eta_0,1) \big(1 + 2\max_{1 \leq p < n}C_1(p) \big).
\end{aligned}
\end{equation}
This finishes the study of Case 2.

\bigskip

Summarizing the results from Case 1 and Case 2, for arbitrary $M^{(0)} \in \mathcal{M}_0(n)$ we have $U_r(M^{(0)})$ such that for all $M \in U_r(M^{(0)})$ and $\varphi$ satisfying \eqref{eqn:inhomogeneous_bound} with $C_0 = 1$, 
one can take $x \in \R^d$ satisfying
\begin{equation} \notag
\begin{aligned}
Mx = \varphi, \quad \text{and} \quad \|x\|_{\ell^\infty} \leq 4(n-1) C(n,\eta_0,1) \big(1 + 2\max_{1 \leq p < n}C_1(p) \big).
\end{aligned}
\end{equation}
Recall that $U_r(M^{(0)})$ is give by
\begin{equation} \notag
\begin{aligned}
r &= \frac{1}{2nC(n,\eta_0,1)(1 + 2\max_{1 \leq p < n}C_1(p))},
\\
U_r(M^{(0)}) &= \left\{M \in \R^{n \times n} : \max_{1 \leq i,j \leq n} |M - M^{(0)}| < r \right\}.
\end{aligned}
\end{equation}
In conclusion, we can take
\begin{equation} \notag
\begin{aligned}
C\big(n, U_r(M^{(0)}) \cap \mathcal{M}_0(n)\big) = 4(n-1) C(n,\eta_0,1) \big(1 + 2\max_{1 \leq p < n}C_1(p) \big).
\end{aligned}
\end{equation}
We can conclude the proof by compactness.
\end{proof}

We finish this section by explaining how one may be able to find a (very rough) upper bound in the previous proof. By Lemma~\ref{lem:maximum_principle}, we have $C(n,\eta,1) = C_2(n)\eta^{-(n-1)}$ where $C_2(n)$ only depends on the dimension $n$.
For any $M \in \mathcal{M}(n)$ let us define
\begin{equation} \notag
\begin{aligned}
\eta(M) &\defeq \min \big\{ |M_{ij}| : i \neq j, |M_{ij}| \neq 0\big\},
\\
r(M) &\defeq \frac{1}{2nC_2(n)[\eta(M)]^{-(n-1)}(1 + 2\max_{1 \leq p < n}C_1(p))},
\end{aligned}
\end{equation}
We are going to argue that for sufficiently small $\eta_{\min} = \eta_{\min}(n) > 0$ and arbitrary $M^{(0)} \in \mathcal{M}_0(n)$,
there exists $M_* \in \mathcal{M}(n)$ such that $\eta(M_*) \geq \eta_{\min}$ and $M^{(0)} \in U_{r(M_*)}(M_*)$.
Once this argument is proved and $\eta_{\min}$ is given explicitly, by the proof of Theorem~\ref{thm:linear_uniform_boundedness}, we have
\begin{equation} \notag
\begin{aligned}
C\big(n, U_{r(M_*)}(M_*) \cap \mathcal{M}_0(n)\big) \leq 4(n-1) C_2(n)[\eta_{\min}(n)]^{-(n-1)} \big(1 + 2\max_{1 \leq p < n}C_1(p) \big).
\end{aligned}
\end{equation}
Therefore, 
we can take
\begin{equation} \notag
\begin{aligned}
C_1(n) = 4(n-1) C_2(n)[\eta_{\min}(n)]^{-(n-1)} \big(1 + 2\max_{1 \leq p < n}C_1(p) \big),
\end{aligned}
\end{equation}
which gives an explicit induction relation of $C_1(n)$.

We now describe how to find such $M_*$.
Let us construct a sequence $\{M^{(k)}\} \subset \mathcal{M}(n)$ (starting with $M^{(0)}$) by the following iterations,
\begin{enumerate}
\item If $M^{(k)} = 0$, we stop the sequence. Otherwise, take
  \[
  (i^{(k)},j^{(k)}) \in \argmin_{(i,j):i\neq j, |M_{ij}^{(k)}| \neq 0} |M_{ij}^{(k)}|.
  \]
  
\item For $M^{(k)} \neq 0$, take the decomposition $I = I_1^{(k)} \cup \dots \cup I_m^{(k)}$ as in Lemma~\ref{lem:block_diagonal}.
Assume WLOG that $i^{(k)},j^{(k)} \in I_1^{(k)}$. 
By the equivalent definition of irreducible matrix that the associated directed graph is strongly connected, one can take a path $j^{(k)} = j^{(k)}_0, j^{(k)}_1, \dots , j^{(k)}_p = i^{(k)}$ such that $M_{j^{(k)}_l, j^{(k)}_{l+1}}^{(k)} \neq 0$ for all $l = 1,\dots, (p-1)$.

Let $P^{(k)}$ be the permutation matrix given by
\begin{equation} \notag
P_{i,j}^{(k)} = \left\{
\begin{aligned}
&1 \quad &&\text{ if } (i,j) = (i^{(k)},j^{(k)})
\\
&&&\text{ or } (i,j) = (j^{(k)}_l, j^{(k)}_{l+1}), \; l = 1,\dots, (p-1)
\\
&&&\text{ or } i = j \in I \setminus \{ j^{(k)}_0, j^{(k)}_1, \dots , j^{(k)}_p \}
\\
\\
&0 && \text{ else}.
\end{aligned} \right.
\end{equation}

\item Take 
$M^{(k+1)} = \Bigg( M^{(k)} + \eta(M^{(k)})(P^{(k)} - I) \Bigg)$.
\end{enumerate}

For any $(i,j) \in I^2$ such that $i \neq j$ and $(P^{(k)} - I)_{ij} = 1$, we have $M^{(k)}_{i,j} < 0$.
By the definition of $\eta(M^{(k)})$ it is easy to verify that the non-diagonal entries of $M^{(k+1)}$ given in this way are again non-positive. Hence $M^{(k+1)} \in \mathcal{M}(n)$ if $M^{k} \in \mathcal{M}(n)$.
By induction, the above procedure produces a sequence $\{M^{(k)}\} \subset \mathcal{M}(n)$.
We are going to argue that for sufficiently small $\eta_{\min}$, there must be an adequate candidate for $M_*$ in $\{M^{(k)}\}$ before the sequence ends.

First, recall the definition $\eta(M) \defeq \min \big\{ |M_{ij}| : i \neq j, |M_{ij}| \neq 0\big\}$, it is straightforward to see that $M_{i^{(k)},j^{(k)}}^{(k+1)} = 0$.
Since there are only $n^2 - n$ non-zero non-diagonal entries and our process eliminates at least one entry at a time, the process must terminate somewhere before step $n^2 - n$.

Secondly, observe that $\|M^{(k+1)} - M^{(k)}\|_{\ell^\infty} \leq \eta(M^{(k)})$ and $|M_{i,j}^{(k+1)}| \leq |M_{i,j}^{(k)}|$, $\forall i\neq j$.
By induction, 
\begin{equation} \notag
\begin{aligned}
\|M^{(k+1)} - M^{(0)}\|_{\ell^\infty} \leq {\textstyle \sum_{l = 0}^{k} \eta(M^{(l)})}, \quad \max_{i \neq j} |M_{ij}^{(k+1)}| \leq 1.
\end{aligned}
\end{equation}
Thus, if
\begin{equation} \notag
\begin{aligned}
{\textstyle \sum_{l = 0}^{k} \eta(M^{(l)})} \leq r(M^{(k+1)}) = \frac{1}{2n C_2(n)\eta(M^{(k+1)})^{-(n-1)} (1 + 2\max_{1 \leq p < n}C_1(p))},
\end{aligned}
\end{equation}
or equivalently
\begin{equation} \notag
\begin{aligned}
\left( 2n C_2(n) \Big(1 + 2\max_{1 \leq p < n}C_1(p)\Big) {\textstyle \sum_{l = 0}^{k} \eta(M^{(l)})} \right)^{1/(n-1)} \leq \eta(M^{(k+1)}),
\end{aligned}
\end{equation}
then it is guaranteed that $M^{(0)} \in U_{r(M^{(k+1)})} ( M^{(k+1)} )$.

To ensure that $M^{(k+1)}$ is an adequate candidate of $M_*$, we also need that $\eta(M^{(k+1)}) \geq \eta_{\min}$.
Let us define the function
\begin{equation} \label{eqn:induction_minimum_coef}
\begin{aligned}
\rho(r) \defeq \max \left\{\eta_{\min}, \left( {2nC_2(n)} \big(1 + 2\max_{1 \leq p < n}C_1(p) \big) r \right)^{1/(n-1)} \right\},
\end{aligned}
\end{equation}
and reformulate what we just discussed as the following:
If $\eta(M^{(k+1)}) \geq \rho({\textstyle \sum_{l = 0}^{k} \eta(M^{(l)})})$,
we  have that ${\textstyle \sum_{l = 0}^{k} \eta(M^{(l)})} \leq r(M^{(k+1)})$ and $\eta(M^{(k+1)}) \geq \eta_{\min}$, so we can take $M_* = M^{(k+1)}$.

Let us assume that no candidate of $M_*$ appears until step $m$. Then we should have
$\eta(M^{(k)}) <  \rho({\textstyle \sum_{l = 0}^{k-1} \eta(M^{(l)})})$ for $k = 1,\dots m$.
Define $\sigma^{(0)} = \eta_{\min}$ and define $\sigma^{(k)}$ inductively by 
\begin{equation} \notag
\begin{aligned}
\sigma^{(k)} = \rho({\textstyle \sum_{l = 0}^{k-1} \sigma^{(l)}}).
\end{aligned}
\end{equation}
Obviously we have $\eta(M^{(0)}) \leq \sigma^{(0)}$.
Note that $\rho$ is increasing for any fixed $\eta_{\min} > 0$. Therefore we have $\eta(M^{(k)}) \leq \rho({\textstyle \sum_{l = 0}^{k-1} \eta(M^{(l)})}) \leq \rho({\textstyle \sum_{l = 0}^{k-1} \sigma^{(l)}}) \leq \sigma^{(k)}$ provided that $\eta(M^{(l)}) \leq \sigma^{(l)}$ for all $l = 0,1, \dots , (k-1)$.
Applying the induction argument, we conclude that $\eta(M^{(k)}) \leq \sigma^{(k)}$ for all $k = 0,1, \dots, m$.

Let us discuss the growth of $\sigma^{(k)}$.
Note that for $r \to 0^+$ and $\eta_{\min} \to 0^+$ we have $\rho(r) \to 0^+$. Therefore $\lim_{\eta_{\min} \to 0} \sigma^{(k)} = \lim_{\eta_{\min} \to 0} \rho({\textstyle \sum_{l = 0}^{k-1} \sigma^{(l)}}) = 0$ provided that $\lim_{\eta_{\min} \to 0} \sigma^{(l)} = 0$ for all $l = 1, \dots , (k-1)$.
Applying the induction argument, we conclude that $\lim_{\eta_{\min} \to 0} \sigma^{(k)} = 0$ for all $k \geq 0$.
In particular, by taking sufficiently small $\eta_{\min}$, one would have ${\textstyle \sum_{l = 0}^{n^2-n} \sigma^{(l)}} < 1$.

Since $\max_{i \neq j} |M_{ij}^{(0)}| = 1$, let us take $i,j$ s.t. $|M_{ij}^{(0)}| = 1$.
Observing that
\begin{equation} \notag
\begin{aligned}
|M_{ij}^{(m+1)}| \geq 1 - {\textstyle \sum_{l = 0}^{m} \eta(M^{(l)})} \geq 1 - {\textstyle \sum_{l = 0}^{m} \sigma^{(l)}}.
\end{aligned}
\end{equation}
For $m < n^2 - n$, this implies that $|M_{ij}^{(m+1)}| > 0$, hence the iteration does not terminate at step $(m+1)$ either.
By induction, unless there is an adequate candidate of $M_*$ found, the iteration does not terminate before step $n^2 - n$.
Recall that the iteration must terminate somewhere before step $n^2 - n$ as the non-zero entries are reducing, we conclude that an adequate candidate of $M_*$ must appear somewhere before step $n^2 - n$.
%
\section{Proof of remaining lemmas and propositions} \label{sec:kernel_equiv}
In this section we collect the remaining missing proofs of various technical lemmas. Let us begin with Lemma~\ref{lem:kernel_equiv_Euclidean}.

\begin{proof}[Proof of Lemma~\ref{lem:kernel_equiv_Euclidean}] 
The equation \eqref{eqn:key_estimate_perturb} is equivalent to 
\begin{equation*}
\begin{aligned}
\int_{\R^{2d}} \left(K_g^h(x,y)-K_f^h(x,y)\right) |u(x) - v(y)|^p \;\rd x\rd y
\leq C(h_1/h_0) \int K_f^h(x,y) |u(x) - v(y)|^p \;\rd x \rd y.
\end{aligned}
\end{equation*}
Notice that $\phi(x)$, yet $K^h(x)$, has compact support in the ball $B(0;2)$. Also, notice that
\begin{equation*}
\begin{aligned}
|f_1(x) - f_2(y)| \geq |x - y| - |x - f_1(x)| - |y - f_2(y)| \geq |x - y| - 2h_1,
\end{aligned}
\end{equation*}
where $h_1 \leq 1/4$ by our assumption.
Hence, $K_f^h(x,y)$ and $K_g^h(x,y)$ are non-zero only if $|x - y| \leq 5/2$.

We further take the decomposition
\begin{equation*}
\begin{aligned}
\{(x,y) \in \R^{2d} : |x - y| \leq 5/2\} = \{|x - y| \leq 1/2\} \cup \{1/2 < |x - y| \leq 5/2\},
\end{aligned}
\end{equation*}
by which we can rewrite the double integral as
\begin{equation*}
\begin{aligned}
&\quad \int_{\R^{2d}} \left(K_g^h(x,y)-K_f^h(x,y)\right) |u(x) - v(y)|^p \;\rd x\rd y
\\
&= \left(\int_{|x - y| \leq 1/2} + \int_{1/2 < |x - y| \leq 5/2}\right) 
 \left(K_g^h(x,y)-K_f^h(x,y)\right) |u(x) - v(y)|^p \;\rd x\rd y.
\end{aligned}
\end{equation*}
On the set $\{|x - y| \leq 1/2\}$, we have $|f_1(x)-f_2(y)| \leq |x - y| + 2h_1 \leq 1$.
Hence 
\begin{equation*}
\begin{aligned}
\phi(|f_1(x)-f_2(y)|) = \phi(|g_1(x)-g_2(y)|) = 1, \quad \forall |x-y| \leq 1/2.
\end{aligned}
\end{equation*}
Thus we have
\begin{equation*}
\begin{aligned}
&\quad \int_{|x - y| \leq 1/2} \left(K_g^h(x,y)-K_f^h(x,y)\right) |u(x) - v(y)|^p \;\rd x\rd y.
\\
&\leq \sup_{|x - y| \leq 1/2} \frac{\left|K_g^h(x,y)-K_f^h(x,y)\right|}{K_f^h(x,y)}
\int_{\R^{2d}} K_f^h(x,y) |u(x) - v(y)|^p \;\rd x\rd y,
\end{aligned}
\end{equation*}
where the coefficient before the integral can be bounded by
\begin{equation*}
\begin{aligned}
\sup_{|x - y| \leq 1/2} \frac{\left|K_g^h(x,y)-K_f^h(x,y)\right|}{K_f^h(x,y)} &=  \sup_{|x - y| \leq 1/2} \frac{\left|\frac{1}{\left( |g_1(x) - g_2(y)| + h \right)^d} - \frac{1}{\left( |f_1(x) - f_2(y)| + h \right)^d}\right|}{\frac{1}{\left( |f_1(x) - f_2(y)| + h \right)^d}}
\\
&\leq \frac{4h_1 \; d/h^{d+1}}{1/h^d}
\leq C h_1/h_0.
\end{aligned}
\end{equation*}
Moreover, on the set $\{1/2 < |x - y| \leq 5/2\}$,
\begin{equation*}
\begin{aligned}
&\quad \int_{1/2 < |x - y| \leq 5/2} \left(K_g^h(x,y)-K_f^h(x,y)\right) |u(x) - v(y)|^p \;\rd x\rd y
\\ \notag
&\leq \int_{1/2 < |x - y| \leq 5/2} \left(K_g^h(x,y)-K_f^h(x,y)\right) 5^{p-1}\Big(  |u(x) - v(4x/5+y/5)|^p
\\ \label{eqn:kernel_term_1}
& \quad\quad\quad + |v(4x/5+y/5) - u(3x/5+2y/5)|^p + |u(3x/5+2y/5) - v(2x/5+3y/5)|^p
\\
& \quad\quad\quad + |v(2x/5+3y/5) - u(x/5+4y/5)|^p + |u(x/5+4y/5) - v(y)|^p  \Big) \;\rd x\rd y.
\end{aligned}
\end{equation*}
As an example, let us look at the second term. With a change of variable
\begin{equation*}
\begin{aligned}
w = 3/5 + 2/5, \quad z = x/5 - y/5,
\end{aligned}
\end{equation*}
one has that
\begin{equation*}
\begin{aligned}
&\quad 5^{p-1}\int_{1/2 < |x - y| \leq 5/2} \left(K_g^h(x,y)-K_f^h(x,y)\right) |v(4x/5+y/5) - u(3x/5+2y/5)|^p \;\rd x\rd y
\\
&= 5^{d+p-1} \int_{1/10 < |z| \leq 1/2} \left(K_g^h(w+2z,w-3z)-K_f^h(w+2z,w-3z)\right)
|u(w) - v(w + z)| \;\rd w \rd z
\\
&\leq 5^{d+p-1} \sup_{w,z \in \R^d, 1/10 < |z| \leq 1/2} \frac{\left|K_g^h(w+2z,w-3z)-K_f^h(w+2z,w-3z)\right|}{K_f^h(w,w+z)}
\\
&\quad\quad\quad\quad \int_{\R^{2d}} K_f^h(w,w+z) |u(w) - v(w+z)| \;\rd w\rd z,
\end{aligned}
\end{equation*}
where the coefficient before the integral can be bounded by
\begin{equation*}
\begin{aligned}
&\; \sup_{w,z \in \R^d, 1/10 < |z| \leq 1/2} \frac{\left|K_g^h(w+2z,w-3z)-K_f^h(w+2z,w-3z)\right|}{K_f^h(w,w+z)}
\\
= &\; \sup_{w,z \in \R^d, 1/10 < |z| \leq 1/2} \frac{\left|K_g^h(w+2z,w-3z)-K_f^h(w+2z,w-3z)\right|}{\frac{1}{\left( |f_1(w) - f_2(w+z)| + h \right)^d}}
\\
\leq &\; C h_1 \leq Ch_1/h_0.
\end{aligned}
\end{equation*}
The other four terms can be bounded by the same approach.

In conclusion, we have
\begin{equation}
\begin{aligned}
&\quad \int_{\R^{2d}} K_g^h(x,y) |u(x) - v(y)| \;\rd x\rd y
\\
&= \int_{\R^{2d}} \left(K_g^h(x,y) - K_f^h(x,y) \right) |u(x) - v(y)| \;\rd x\rd y 
+ \int_{\R^{2d}} K_f^h(x,y) |u(x) - v(y)| \;\rd x\rd y 
\\
&\leq \left(1 +  Ch_1/h_0\right) \int K_f^h(x,y) |u(x) - v(y)| \;\rd x \rd y.
\end{aligned}
\end{equation}
\end{proof}

Next we prove Lemma~\ref{lem:kernel_equiv_extension_trick}.
\begin{proof}[Proof of Lemma~\ref{lem:kernel_equiv_extension_trick}]
We first choose measurable sets $(V_i)_{i \in \mathcal{V}} \subset \R^d$ by the following. Divide $\R^d$ into small hypercubes 
\begin{equation*}
\begin{aligned}
Q_{[m]} = \prod_{k=1}^d \Big[\big(n_k/\sqrt{d}\big) \delta x, \big((n_k+1)/\sqrt{d}\big) \delta x\Big), \quad \forall [m] = (n_1, \dots, n_d) \in \Z^d.
\end{aligned}
\end{equation*}
where each hypercube has diameter $\delta x$.

Then, in each hypercube $Q_{[m]}$ choose measurable sets $V_{i,[m]}$, $i \in \mathcal{V}$ satisfying
\begin{equation*}
\begin{aligned}
V_{i,[m]} \subset Q_{[m]}, \quad |V_{i,[m]}| = \int_{Q_{[m]}} \chi_i, \quad \forall i \in \mathcal{V}, \quad \text{ and } \quad V_{i,[m]} \cap V_{j,[m]} \quad \forall i,j \in \mathcal{V},
\end{aligned}
\end{equation*}
which is always possible since $\sum_{i \in \mathcal{V}} \chi_i \leq 1$ by our definition.

Choose $V_i = \bigcup_{[m]} V_{i,[m]}$ so that
\begin{equation*}
\begin{aligned}
|V_i| = \pi_i = \int_{\R^d} \chi_i, \quad \sup_{x \in V_i} |x - x_i| < 2\delta x, \quad \forall i \in \mathcal{V}, \quad \text{ and } \quad V_{i} \cap V_{j} \quad \forall i,j \in \mathcal{V}.
\end{aligned}
\end{equation*}
Moreover, recall our assumption \eqref{eqn:mesh_comparable_2} that $\sum_{i \in \mathcal{V}} \chi_i(x) = 1, \forall x \in \Omega + B(0,4)$.
For any hypercube $Q_{[m]} \subset \Omega + B(0,4)$, one has $\sum_{i \in \mathcal{V}} |V_{i,[m]}| = |Q_{[m]}|$. Since we have assumed  $\delta x \leq 1/16$, it is easy to verify that $Q_{[m]} \cap \big( \Omega + B(0,3) \big) \neq \varnothing$ implies $Q_{[m]} \subset \Omega + B(0,4)$.

Then up to modification on a negligible set, one has that
\begin{equation*}
\begin{aligned}
\textstyle \Big( \bigcup_{i \in \mathcal{V}} V_i \Big) \cap \Big( \Omega + B(0,3) \Big) &= \textstyle \Big( \bigcup_{i \in \mathcal{V}} \bigcup_{[m] \in \Z^d} V_{i, [m]} \Big) \cap \Big( \Omega + B(0,3) \Big)
\\
&= \textstyle \bigcup_{[m] \in \Z^d}  \Big( \bigcup_{i \in \mathcal{V}} V_{i, [m]} \Big) \cap \Big( \Omega + B(0,3) \Big)
\\
&= \Omega + B(0,3).
\end{aligned}
\end{equation*}
Recall our choice of piecewise constant extension $u^V \defeq \sum_{i \in \mathcal{V}} u_i \mathbbm{1}_{V_i}$ and
\begin{equation*}
f(x) = \left\{
\begin{aligned}
&\widetilde{x}_i, \quad &&\text{ for } x \in V_i,\; i \in \mathcal{V},
\\
&x, \quad &&\text{ for } x \notin \textstyle \bigcup_{i \in \mathcal{V}} V_i.
\end{aligned} \right.
\end{equation*}
Notice that $u_i \neq 0$ only if $\supp \chi_i \subset \Omega$ and we have assumed $\delta x \leq 1/16$. Hence the extended function $u^V$ satisfies $\supp u^V \subset \Omega + B(0,1)$. Also, by our assumption it is straightforward that
\begin{equation*}
\begin{aligned}
\sup_{x \in \R^d} |x - f(x)| \leq \sup_{i \in \mathcal{V}}\sup_{x \in V_i} |x - x_i| + \sup_{i \in \mathcal{V}} |x_i - \widetilde{x}_i| \leq 2\delta x + h_2 < 1/4.
\end{aligned}
\end{equation*}
Finally, let us consider the integral
\begin{equation*}
\begin{aligned}
\int_{\R^{2d}} K^h\big(f(x) - f(y)\big) |u^V(x) - u^V(y)|^p \;\rd x\rd y.
\end{aligned}
\end{equation*}
We have proved $\supp u^V \subset \Omega + B(0,1)$ and by definition $\supp K^h \in B(0,2)$.
Therefore, for $x \notin \Omega + B(0,3)$, either $y \notin \Omega + B(0,1)$, making $|u^V(x) - u^V(y)| = 0$, or $y \in \Omega + B(0,1)$, making $K^h(x - y) = 0$.

The same argument also applies to $y$ and as a consequence, the above integral  can be taken instead over any subset of $\R^{2d}$ including $\big( \Omega + B(0,3) \big)^2$. In particular, it can be reformulated as
\begin{equation*}
\begin{aligned}
\int_{\left( \bigcup_{i\in \mathcal{V}} V_i \right)^2} K^h\big(f(x) - f(y)\big) |u^V(x) - u^V(y)| \;\rd x\rd y = \sum_{i,j \in \mathcal{V}} \widetilde{K}_{i,j}^h |u_i - u_j|^p \pi_i \pi_j,
\end{aligned}
\end{equation*}
which completes the proof.
\end{proof}

We are now ready to prove Lemma~\ref{lem:compactness_regularity}, Proposition~\ref{prop:kernel_equiv} and Proposition~\ref{prop:discrete_regularity_extended}. Let us start with Proposition~\ref{prop:kernel_equiv}, which is immediate:
\begin{proof}[Proof of Proposition~\ref{prop:kernel_equiv}]
By Lemma~\ref{lem:kernel_equiv_extension_trick}, choose measurable sets $(V_i)_{i \in \mathcal{V}} \subset \R^d$, take piecewise constant extension $u^V \defeq \sum_{i \in \mathcal{V}} u_i \mathbbm{1}_{V_i}$ and take $f_1^{(1)},f_2^{(1)},f_1^{(2)},f_2^{(2)} : \R^d \to \R^d$ by
\begin{equation*}
\begin{aligned}
\text{for } k = 1,2, \quad f_1^{(k)}(x) = f_2^{(k)}(x) = \left\{
\begin{aligned}
&\widetilde{x}_i^{(k)}, \quad &&\text{ for } x \in V_i,\; i \in \mathcal{V},
\\
&x, \quad &&\text{ for } x \notin \textstyle \bigcup_{i \in \mathcal{V}} V_i.
\end{aligned} \right.
\end{aligned}
\end{equation*}
Then it is straightforward that
\begin{equation*}
\begin{aligned}
\|u\|_{h_0,p,\theta;\widetilde{x}^{(k)}}^p &= \sup_{h_0 \leq h \leq 1/2} |\log h|^{-\theta} \sum_{i,j \in \mathcal{V}} \widetilde{K}^h\big(\widetilde{x}_i^{(k)} - \widetilde{x}_j^{(k)}\big) |u_i - u_j|^p \pi_i \pi_j
\\
&= \sup_{h_0 \leq h \leq 1/2} |\log h|^{-\theta} \int_{\R^{2d}} K^h\big(f_1^{(k)}(x)-f_2^{(k)}(y)\big) |u^V(x) - u^V(y)|^p \;\rd x \rd y,
\end{aligned}
\end{equation*}
and
\begin{equation*}
\begin{aligned}
\text{for } k,l = 1,2, \quad \sup_{x \in \R^d} |x - f_l^{(k)}(x)| \leq 2\delta x + h_2 \leq 3h_2.
\end{aligned}
\end{equation*}
We now apply Lemma~\ref{lem:kernel_equiv_Euclidean} with $h_1 = 3h_2 < 1/4$, which gives that
\begin{equation*}
\begin{aligned}
\|u\|_{h_0,p,\theta;\widetilde{x}^{(2)}}^p \leq \left(1 +  Ch_2/h_0 \right) \|u\|_{h_0,p,\theta;\widetilde{x}^{(1)}}^p.
\end{aligned}
\end{equation*}
Noticing $(1 + x)^{1/p} \leq 1 + x$ for all $x \geq 0$, we conclude that
\begin{equation*}
\begin{aligned}
\|u\|_{h_0,p,\theta;\widetilde{x}^{(2)}} \leq \left(1 +  Ch_2/h_0 \right) \|u\|_{h_0,p,\theta;\widetilde{x}^{(1)}},
\end{aligned}
\end{equation*}
which finishes the proof.

\end{proof}

Next, let us prove Proposition~\ref{prop:discrete_regularity_extended}.

\begin{proof}[Proof of Proposition~\ref{prop:discrete_regularity_extended}]

Choose any labeling of the index set $\mathcal{V}$, and define
\begin{equation*}
\begin{aligned}
J : \R^d \times [0,1] &\to \mathcal{V}
\\
(y, \omega) &\mapsto \textstyle \min \{i \in \mathcal{V} : \sum_{j \leq i} \chi_j(y) \geq \omega\}.
\end{aligned}
\end{equation*}
Notice that there are only a bounded number of nonzero $\chi_i(y)$ at any point $y$ by our assumption.

Define also
\begin{equation*}
\begin{aligned}
F : \mathcal{V} &\to \Omega_{e}
\\
i &\mapsto x_{i}
\end{aligned}
\end{equation*}
Then for all $y \in \R^d$, $\omega \in [0,1]$,
\begin{equation*}
\begin{aligned}
\big| (F \circ J)(y;\omega) - y \big| \leq 2 \delta x, \quad u^\chi(y) = \sum_{i \in \mathcal{V}} u_i \chi_i(y) = \int_{0}^{1} u_{J(y,\omega)} \;\rd \omega.
\end{aligned}
\end{equation*}
By Lemma~\ref{lem:kernel_equiv_Euclidean}, for all $\omega_1,\omega_2 \in [0,1]$
\begin{equation*}
\begin{aligned}
&\quad \quad \int_{\R^{2d}} K^h(x,y) |u_{J(x,\omega_1)} - u_{J(y,\omega_2)}|^p \;\rd x\rd y
\\
&\leq (1 + C\delta x/ h) \int_{\R^{2d}} K^h\Big((F \circ J)(x,\omega_1),(F \circ J)(y,\omega_2)\Big) |u_{J(x,\omega_1)} - u_{J(y,\omega_2)}|^p \;\rd x\rd y.
\end{aligned}
\end{equation*}
Therefore, notice that $K^h\big(F(i),F(j)\big) = K^h\big(x_{i},x_{j}\big) = K_{i,j}^h$, which implies that
\begin{equation*}
\begin{aligned}
& \int_{\R^{2d}} K^h(x,y) |u^\chi(x) - u^\chi(y)|^p \rd x\rd y 
\\
&\quad= \int_{\R^{2d}} K^h(x,y) \left| \int_0^1 u_{J(x,\omega_1)} \;\rd \omega_1 - \int_0^1 u_{J(y,\omega_2)} \;\rd \omega_2 \right|^p \rd x\rd y 
\\
&\quad \leq \int_{[0,1]^2} \int_{\R^{2d}} K^h(x,y) \left| u_{J(x,\omega_1)} - u_{J(y,\omega_2)} \right|^p \rd x\rd y \rd \omega_1 \rd \omega_2
\\
&\quad \leq (1 + C\delta x/ h) \int_{[0,1]^2} \int_{\R^{2d}} K^h\Big((F \circ J)(x,\omega_1),(F \circ J)(y,\omega_2)\Big) \left| u_{J(x,\omega_1)} - u_{J(y,\omega_2)} \right|^p\rd x\rd y \rd \omega_1 \rd \omega_2
\\
&\quad= (1 + C\delta x/ h) \sum_{i,j \in \mathcal{V}} K_{i,j}^h |u_{i} - u_{j}|^p \pi_{i} \pi_{j}.
\end{aligned}
\end{equation*}

Again by Lemma~\ref{lem:kernel_equiv_Euclidean},
\begin{equation*}
\begin{aligned}
&\quad \sum_{i,j \in \mathcal{V}} K_{i,j}^h \Big| (P_{\mathcal{C}}u)_{i} - (P_{\mathcal{C}}u)_{j} \Big|^p \pi_i \pi_j
\\
&= \sum_{i,j \in \mathcal{V}} K_{i,j}^h \left|\frac{1}{\pi_{i}} \int_{\R^d} u(x) \chi_{i}(x) \;\rd x - \frac{1}{\pi_{j}} \int_{\R^d} u(y) \chi_{j}(y) \;\rd y\right|^p \pi_{i} \pi_{j}
\\
&\leq \sum_{i,j \in \mathcal{V}} \int_{\R^{2d}} K_{i,j}^h \big| u(x) - u(y)\big|^p \chi_{i}(x)\chi_{j}(y) \rd x \rd y
\\
&= \int_{\R^{2d}} K^h\Big((F \circ J)(x,\omega_1),(F \circ J)(y,\omega_2)\Big) \big| u(x) - u(y)\big|^p \rd x \rd y \rd \omega_1 \rd \omega_2
\\
&\leq (1 + C\delta x/ h) \int_{\R^{2d}} K^h(x,y) |u(x) - u(y)| \;\rd x\rd y,
\end{aligned}
\end{equation*}
concluding the proof.
\end{proof}

Lemma~\ref{lem:compactness_regularity} can then be derived from Proposition~\ref{prop:discrete_regularity_extended}.

\begin{proof}[Proof of Lemma~\ref{lem:compactness_regularity}]
Since we have assumed $h > h_0 > \delta x$ and $\|(u_i)_{i \in \mathcal{V}}\|_{h_0,p,\theta} \leq L$, by Proposition~\ref{prop:discrete_regularity_extended},
\begin{equation*}
\begin{aligned}
\int_{\R^{2d}} K^h(x,y) |u^\chi(x) - u^\chi(y)|^p \rd x\rd y
&\leq (1 + C\delta x/ h) \sum_{i,j \in \mathcal{V}} K_{i,j}^h |u_{i} - u_{j}|^p \pi_{i} \pi_{j} 
\\
&\leq C |\log h|^{\theta},
\end{aligned}
\end{equation*}
where the constant $C$ may depend on $L$.

Introduce the renormalization factor
\begin{equation*}
\begin{aligned}
C_h = |\log \theta| / \|K^h\|_{L^1}.
\end{aligned}
\end{equation*}
Then $C_h$ is bounded form above and below uniformly with respect to $h$, and the renormalized kernel $\bar K^h$ reads
\begin{equation*}
\begin{aligned}
\bar K^h(x) = K^h(x) / \|K^h\|_{L^1} = C_h |\log \theta|^{-1} K^h(x).
\end{aligned}
\end{equation*}
Hence
\begin{equation*}
\begin{aligned}
\|u^\chi - \bar K^h \star u^\chi \|_{L^p}^p &= (C_h |\log h|^{-1})^{p} \int_{\R^d} \left( \int_{\R^d} K^{h}(x-y) \big( u^\chi(x) - u^\chi(y) \big) \rd y \right)^p \rd x
\\
&\leq C |\log h|^{-p} \|K^h\|_{L^1}^{p-1} \int_{\R^{2d}} K^h(x-y) |u^\chi(x) - u^\chi(y)|^p \;\rd y\rd x
\\
&\leq C|\log h|^{-1} \int_{\R^{2d}} K^h(x-y) |u^\chi(x) - u^\chi(y)|^p \;\rd y\rd x.
\end{aligned}
\end{equation*}
This implies
\begin{equation} \label{eqn:approximate_by_convolution}
\begin{aligned}
\|u^\chi - \bar K^h \star u^\chi\|_{L^p}^p \leq C |\log h|^{\theta - 1}
\end{aligned}
\end{equation}

\end{proof}

We have finished the proofs of all lemmas and propositions in Section~\ref{subsec:compactness_and_propagation} but it remains to prove Proposition~\ref{prop:two_ways_projection} as claimed in the proof of Theorem~\ref{thm:periodic_compactness_result}.

\begin{proof}[Proof of Proposition~\ref{prop:two_ways_projection}]
Let $(\mathcal{C}, \mathcal{F})$
be a mesh as in Definition~\ref{defi:partition_of_unity_mesh} over $\Omega \subset \R^d$ such that \eqref{eqn:mesh_comparable_1} hold.
Assume that each face function $\bm{n}_{i,j} \in \mathcal{F}$ is of form $\bm{n}_{i,j}(x) = \bm{N}_{i,j} w_{i,j}(x), \forall x \in \R^d$, where $\bm{N}_{i,j}$ is a unit vector and $w_{i,j}$ is a scalar function.

Then for $1 \leq p \leq \infty$,
\begin{equation*}
\begin{aligned}
\big\| P_{\mathcal{F}}' b - P_{\mathcal{F}} b \big\|_{L^{p}([0,T] \times \mathcal{F})} \leq C \delta x \|b\|_{L^{p}_t(W^{1,p}_x)},
\end{aligned}
\end{equation*}
where the constant only depends on $p$ and the constant in the structural assumption \eqref{eqn:mesh_comparable_1}.

We are going to first prove the inequality for any fixed time $t$ and To simplify the notation we omit $t$ in all the calculations.

By definition, we have
\begin{equation*}
\begin{aligned}
\big( P_{\mathcal{F}} b - P_{\mathcal{F}}' b \big)_{i,j} = \int_{\R^d} \Big( b(x) \cdot \bm{N}_{i,j} \Big)^+ w_{i,j}(x) \;\rd x - \left( \int_{\R^d} b(x) \cdot \bm{N}_{i,j} w_{i,j}(x) \;\rd x \right)^+.
\end{aligned}
\end{equation*}
We introduce the more general function
\begin{equation*}
\begin{aligned}
I(b, \bm{N}, w) \defeq \int_{\R^d} \Big( b(x) \cdot \bm{N} \Big)^+ w(x) \;\rd x - \left( \int_{\R^d} b(x) \cdot \bm{N} w(x) \;\rd x \right)^+
\end{aligned}
\end{equation*}
where $\bm{N} \in \R^d$ is a unit vector and $w$ is a non-negative, bounded function with compact support.
It is straightforward that $I(b,\bm{N},w) \geq 0$ and for any two  functions $v, w \geq 0$,
\begin{equation*}
\begin{aligned}
&\quad I(b,\bm{N},w) + I(b,\bm{N},v) 
\\
&= \int_{\R^d} \Big( b(x) \cdot \bm{N} \Big)^+ \big(w(x) + v(x)\big) \;\rd x - \left( \int_{\R^d} b(x) \cdot \bm{N} w(x) \;\rd x \right)^+  - \left( \int_{\R^d} b(x) \cdot \bm{N} v(x) \;\rd x \right)^+
\\
&\leq \int_{\R^d} \Big( b(x) \cdot \bm{N} \Big)^+ \big(w(x) + v(x)\big) \;\rd x - \left( \int_{\R^d} b(x) \cdot \bm{N} \big(w(x) + v(x)\big) \;\rd x \right)^+
\\
&= I(b,\bm{N},w + v).
\end{aligned}
\end{equation*}
Hence if  $0 \leq w \leq u$, then $I(b,\bm{N},w) \leq I(b,\bm{N},u)$.

Moreover, $I(b,\bm{N},w)$ is directly bounded by the following inequality
\begin{equation*}
\begin{aligned}
I(b,\bm{N},w)
 \leq \sup_{\lambda}\left\{ \int_{\R^d} \Big( b(x) \cdot \bm{N} - \lambda \Big)^+ w(x) \;\rd x - \left( \int_{\R^d} \Big( b(x) \cdot \bm{N} - \lambda \Big) w(x) \;\rd x \right)^+ \right\}.
\end{aligned}
\end{equation*}
And it is easy to verify that (in distributional sense)
\begin{equation*}
\begin{aligned}
&\quad \frac{\partial}{\partial \lambda} \left\{ \int_{\R^d} \Big( b(x) \cdot \bm{N} - \lambda \Big)^+ w(x) \;\rd x - \left( \int_{\R^d} \Big( b(x) \cdot \bm{N} - \lambda \Big) w(x) \;\rd x \right)^+ \right\}
\\
= &-\int_{\R^d} \mathbbm{1}\Big\{b(x) \cdot \bm{N} - \lambda \geq 0\Big\} w(x) \;\rd x
+ \mathbbm{1} \left\{ \int_{\R^d} \Big( b(x) \cdot \bm{N} - \lambda \Big) w(x) \;\rd x \geq 0 \right\} \int_{\R^d} w(x)\;\rd x.
\end{aligned}
\end{equation*}
Thus the maximum is attained at
\begin{equation*}
\begin{aligned}
\lambda^* = \left( \int_{\R^d} b(x) \cdot \bm{N} w(x) \;\rd x \right) / \left( \int_{\R^d} w(x) \;\rd x \right),
\end{aligned}
\end{equation*}
thus
\begin{equation*}
\begin{aligned}
I(b,\bm{N},w)
&\leq \int_{\R^d} \Big( b(x) \cdot \bm{N} - \lambda^* \Big)^+ w(x) \;\rd x
\\
&\leq \frac{1}{\|w\|_{L^1}} \int_{\R^{2d}} |b(x) - b(y)|w(x)w(y) \;\rd x \rd y.
\end{aligned}
\end{equation*}
Notice that $w_{i,j} \leq C\delta x^{-1} \mathbbm{1}_{B(x_i;\delta x)}$ by our structural assumptions \eqref{eqn:mesh_comparable_1}.
Therefore,
\begin{equation*}
\begin{aligned}
\big( P_{\mathcal{F}} b - P_{\mathcal{F}}' b \big)_{i,j} &= I(b,\bm{N}_{i,j},w_{i,j}) \leq I(b,\bm{N}_{i,j}, C(\delta x)^{-1} \mathbbm{1}_{B(x_i;\delta x)})
\\
&\leq \frac{C (\delta x)^{-1}}{\|B(x_i;\delta x)\|_{L^1}} \int_{\R^{2d}} |b(x) - b(y)| \mathbbm{1}_{B(x_i;\delta x)}(x) \mathbbm{1}_{B(x_i;\delta x)}(y) \;\rd x \rd y.
\end{aligned}
\end{equation*}
When $p = \infty$, we have
\begin{equation*}
\begin{aligned}
\big\| P_{\mathcal{F}} b - P_{\mathcal{F}}' b \big\|_{L^\infty(\mathcal{F})} &= \sup_{i,j \in \mathcal{V}} \big( P_{\mathcal{F}} b - P_{\mathcal{F}}' b \big)_{i,j} \; (\delta x)^{-(d-1)}
\\
&\leq C(\delta x)^{d-1} \; \delta x \|b\|_{W^{1,\infty}} \; (\delta x)^{-(d-1)}
\\
&= C\delta x \|b\|_{W^{1,\infty}}.
\end{aligned}
\end{equation*}
When $p = 1$, we have
\begin{equation*}
\begin{aligned}
\big\| P_{\mathcal{F}} b - P_{\mathcal{F}}' b \big\|_{L^1(\mathcal{F})} &= \sum_{i,j \in \mathcal{V}} \big( P_{\mathcal{F}} b - P_{\mathcal{F}}' b \big)_{i,j} \delta x = \sum_{(i,j) \in \mathcal{E}} \big( P_{\mathcal{F}} b - P_{\mathcal{F}}' b \big)_{i,j} \delta x
\\
&= \sum_{(i,j) \in \mathcal{E}} \frac{C}{\|B(x_i;\delta x)\|_{L^1}} \int_{\R^{2d}} |b(x) - b(y)| \mathbbm{1}_{B(x_i;\delta x)}(x) \mathbbm{1}_{B(x_i;\delta x)}(y) \;\rd x \rd y
\\
&\leq \sum_{(i,j) \in \mathcal{E}} \frac{C}{\|B(x_i;\delta x)\|_{L^1}} \int_{x \in B(x_i;\delta x)} \int_{|z| \leq 2\delta x} |b(x) - b(x+z)| \;\rd x \rd y
\\
&\leq \frac{C}{\|B(x_i;\delta x)\|_{L^1}} \int_{x \in \R^d} \int_{|z| \leq 2\delta x} |b(x) - b(x+z)| \;\rd x \rd y
\\
&\leq C\delta x \|b\|_{W^{1,1}}.
\end{aligned}
\end{equation*}
An interpolation completes the case $1 \leq p \leq \infty$, i.e.
\begin{equation*}
\begin{aligned}
\big\| P_{\mathcal{F}} b - P_{\mathcal{F}}' b \big\|_{L^p(\mathcal{F})} \leq C\delta x \|b\|_{W^{1,p}}.
\end{aligned}
\end{equation*}
Integrating now over time, we conclude that
\begin{equation*}
\begin{aligned}
\big\| P_{\mathcal{F}}' b - P_{\mathcal{F}} b \big\|_{L^{p}([0,T] \times \mathcal{F})} \leq C \delta x \|b\|_{L^{p}_t(W^{1,p}_x)}.
\end{aligned}
\end{equation*}

\end{proof}

\noindent {\bf Acknowledgments.} P.~E. Jabin and D. Zhou were partially supported by NSF DMS Grant 2049020, 2219397, and 2205694.
\nocite{*}
\bibliography{refs}{} 
\bibliographystyle{siam} 

\end{document}